\documentclass[12pt]{amsart}
\usepackage[utf8]{inputenc}
\usepackage[left=1.2in, right=1.2in, top=1in]{geometry}

\usepackage{esint}
\usepackage{amsmath,amssymb,amsthm,mathrsfs,color,times,textcomp,verbatim,mathtools}
\allowdisplaybreaks
\usepackage{xcolor}
\usepackage[colorlinks=true]{hyperref}
\hypersetup{urlcolor=blue, citecolor=red, linkcolor=blue}
\usepackage[square,numbers]{natbib}
\numberwithin{equation}{section}
\theoremstyle{plain}
\newtheorem{theorem}{Theorem}[section]
\newtheorem{proposition}[theorem]{Proposition}
\newtheorem{lemma}[theorem]{Lemma}

\usepackage{graphicx}

\newtheorem{remark}[theorem]{Remark}

\title[Brezis-Nirenberg problem]{On Brezis' First Open Problem: A Complete Solution}

\author{Liming Sun}
\address{State Key Laboratory of Mathematical Sciences, Academy of Mathematics and Systems Science, Chinese Academy of Sciences, Beijing 100190, China}
\email{lmsun@amss.ac.cn}
\author{Jun-cheng Wei}
\address{ Department of Mathematics, Chinese University of Hong Kong, Shatin, Hong Kong}
\email{jcwei@math.ubc.ca}

\author{Wen Yang}
\address{Department of Mathematics, Faculty of Science and Technology, University of Macau, Taipa, Macau}
\email{wenyang@um.edu.mo}
\date{\today \,(Last Typeset)}

\newcommand{\ba}{\begin{array}}
\newcommand{\ea}{\end{array}}
\newcommand{\bed}{\begin{aligned}}
\newcommand{\eed}{\end{aligned}}
\newcommand{\la}{\lambda}

\newcommand{\R}{{\mathbb R}}
\newcommand{\e}{\varepsilon}

\def \SK{\Sigma_K}
\def \w1{w}
\def \W2{W}

\DeclareMathOperator{\csch}{csch}

\DeclareMathOperator{\arcsinh}{arcsinh}

\begin{document}

% \begin{center}
%     \textit{To the memory of Ha{\"\i}m Brezis, }
% \end{center}

\begin{abstract}
In 2023, H.\,Brezis \cite{brezis2023some} published a list of his ``favorite open problems", which he described as challenges he had ``raised throughout his career and has resisted so far". We provide a complete resolution to the first one--Open Problem 1.1--in Brezis's favorite open problems list: the existence of solutions to the long-standing Brezis-Nirenberg problem on a three-dimensional ball. Furthermore, using the building blocks of Del Pino-Musso-Pacard-Pistoia sign-changing solutions to the Yamabe problem,  we establish the existence of infinitely many sign-changing, nonradial solutions for the full range of the parameter.

%In 2023, H. Brezis \cite{brezis2023some} published a collection of his ``favorite open problems" that he raised throughout his career and has resisted so far. We provide a complete resolution to the first problem--Open Problem 1.1-- on this list, which concerns the existence of solutions to the Brezis-Nirenberg problem in a three-dimensional ball, a question that has remained open for more than 40 years. Furthermore, we go beyond merely proving existence: we establish the presence of infinitely many sign-changing, nonradial solutions for the full range  of the parameter.

\end{abstract}

\maketitle
\tableofcontents

\section{Introduction}

\subsection{Motivation and main results} Let $\Omega$ be a smooth and bounded domain in $\R^N$. In their seminal work, Brezis and Nirenberg \cite{brezis1983positive} studied the following problem:
\begin{align} \begin{cases}\label{BN} \Delta u + \lambda u + |u|^{\frac{4}{N-2}} u = 0 &\text{in } \Omega, \\
 u = 0 &\text{on } \partial \Omega, \end{cases} \end{align}
and established the existence of at least one positive solution under the following conditions: $0 < \lambda < \lambda_1$ if $N \geq 4$, and $0 < \lambda_* < \lambda < \lambda_1$ if $N = 3$. Here, $\lambda_1$ denotes the first eigenvalue of the Laplacian, and $\lambda_{*}$ is a domain-dependent constant, which was later quantified by Druet \cite{druet2002elliptic} in terms of Robin functions. When $\Omega$ is the unit ball in $\R^N$, Brezis and Nirenberg showed that $\lambda_{*}=\frac{\lambda_1}{4}$ and positive solution exists if and only if $ \la \in (\frac{\lambda_1}{4}, \lambda_1)$. Moreover, as a consequence of the classical Pohozaev’s identity,   solutions do not exist if $\la \leq 0$ and $\Omega$ is star-shaped.

Since the pioneering work of Brezis-Nirenberg \cite{brezis1983positive} in 1983, the study of nonlinear elliptic equations involving critical Sobolev exponents has been an area of intense research.
A natural question arises regarding the existence of \textit{sign-changing} solutions to (\ref{BN}). In the case of (\ref{BN}), several existence results have been established for dimensions \( N \geq 4 \). In these higher dimensions, sign-changing solutions can be obtained for every \( \lambda \in (0, \lambda_1 (\Omega)) \) and even for \( \lambda > \lambda_1 (\Omega) \) (see \cite{Capozzi1985, Cerami1984, Castro2003, Cerami1986, Clapp2005, DS2002, SZ2010}). In particular, Capozzi-Fortunato-Palmieri \cite{Capozzi1985} proved that for \( N=4 \) and \( \lambda >0 \), provided that \( \lambda \not\in \sigma(-\Delta) \) (the spectrum of \( -\Delta \) in \( H_0^1 (\Omega) \)), problem \eqref{BN} admits a nontrivial solution. The same conclusion holds for all \( \lambda >0 \) when \( N \geq 5 \). Clapp-Weth \cite{Clapp2005} considered the existence of multiple solutions to the Brezis-Nirenberg problem for dimensions $N\geq 4$.
However, the existence of sign-changing solutions in the three-dimensional case ($N=3$) poses challenges similar to those encountered for positive solutions, while also introducing additional complexities. This problem remains unresolved, even in the simplified scenario of a unit ball, and is the central focus of \textbf{Open Problem 1.1} as formulated by H. Brezis in \cite{brezis2023some}. Let \( \Omega \) be the unit ball \( B_1 \) in \( \mathbb{R}^3 \). Consider the following problem:
\begin{align}
\begin{cases}\label{main-eqn}
    \Delta u+\la u+u^5=0 &\text{in }B_1,\\
    u=0 &\text{on }\partial B_1.
\end{cases}
\end{align}

H.\,Brezis' first open problem, implicitly raised by Brezis and Nirenberg [Remark (6)(d),\cite{brezis1983positive}], is as follows:
\medskip

\noindent
{\bf Open Problem 1.1} (Implicit in \cite{brezis1983positive}) Assume that
\begin{equation}
\label{BNcon10}
    0<\la <\frac{\la_1}{4}.
    \end{equation}
Does there exist a non-trivial solution $ u \not \equiv 0$ to (\ref{main-eqn})?
\medskip

As remarked by H.\,Brezis, under the condition (\ref{BNcon10}), any solution to (\ref{main-eqn}), if exists, must be non-radial and sign-changing. When $\lambda>\lambda_1$ there are sign-changing solutions but there is no positive solution to \eqref{main-eqn}. See \cite{brezis1983positive}. They are obtained by bifurcation from non-radial sign-changing eigenfunctions.

Our main result provides a complete resolution to Brezis' Open Problem 1.1 and, in fact, establishes even more.

\begin{theorem}
\label{th1}
  Assume that
\begin{equation}
\label{BNcon1}
    0<\la <+\infty.
    \end{equation}
    Then there are infinitely many (sign-changing) solutions to (\ref{main-eqn}).
\end{theorem}

There has been extensive research aimed at finding nontrivial solutions to the Brezis-Nirenberg problem \eqref{main-eqn} in three dimensions. Most of these studies focus on cases where \( \lambda \) is greater than \( \lambda_1/4 \) or a small perturbation of \( \lambda_1/4 \) and the solutions are positive; for further details, we refer the reader to Del Pino-Dolbeault-Musso \cite{del2004brezis}, Musso-Salazar \cite{musso2018multispike}
and the references therein. For constructions in higher dimensions, we refer to Iacopetti-Vaira \cite{IV1, Vaira2018}, Musso-Pistoia \cite{Musso2002}, Musso-Serena-Vaira \cite{Musso2024}, Pistoia-Vaira \cite{PV2022}, Premoselli \cite{Prem2022}, Rey \cite{Rey1990} and Robert-Vetois \cite{RV2013, Robert2014}. For bubbling analysis to the Brezis-Nirenberg problem we refer to Bahri-Li-Rey \cite{Li1995}, Han \cite{Han1991}, K\"onig-Laurin \cite{Konig2024}, Frank-Konig-Kova\v{r}\'{\i}k \cite{Frank2024}, Malchiodi-Mayer \cite{Malchiodi2021}, Rey \cite{Rey1990}, Wei \cite{Wei1998} and references therein.   There are many related works on the Brezis-Nirenberg problem which are virtually impossible to give an exhaustive bibliography. In addition to the ones mentioned above, one can see \cite{Amadori_2021,Bartsch2006, Cao2021, Dammak2017, Hebey2004, Hebey2014, Musso2024, Thizy2018, Vaira2015} and references therein.

In this paper, we prove that for any positive \( \lambda \), there exist infinitely many solutions to \eqref{main-eqn}. Motivated by the nonexistence result of Brezis and Nirenberg for \( \lambda \in (0, \lambda_1/4) \), we focus on the existence of sign-changing solutions.  Unlike most of the existing literature mentioned before, which relies on the standard positive Talenti solution, the building blocks of our construction are the   {\em sign-changing} solutions to the Yamabe problem
\begin{equation}
\label{1.yamabe}
\Delta u + |u|^{\frac{4}{N-2}} u = 0 \quad \text{in} \quad \mathbb{R}^N
\end{equation}
constructed by Del Pino-Musso-Pacard-Pistoia \cite{del2011large} which we shall describe next.

\subsection{Nodal solutions to Yamabe problem and the non-degeneracy}

Problem (\ref{1.yamabe}) is a canonical equation in the so-called constant scalar curvature problem with conformal metrics, i.e. Yamabe problem. In an earlier seminal paper, Gidas-Ni-Nirenberg \cite{gidas1979symmetry} employed the moving plane method to classify all positive finite-energy solutions of the equation (\ref{1.yamabe}). They showed that all such solutions are given by  (the so-called Talenti bubbles)
\begin{equation}
\label{1.bubble}
u(z)= \left(\frac{\sqrt{N(N-2)}\lambda}{\lambda^2 + |z - a|^2}\right)^{\frac{N-2}{2}}, \quad \lambda>0, \quad a \in \mathbb{R}^N.
\end{equation}
In 1989, Caffarelli-Gidas-Spruck \cite{caffarelli1989asymptotic} extended this classification by removing the finite-energy assumption, obtaining the same result, and Chen-Li \cite{chen1991classification} gave a more simplified proof. The classification of solutions for polyharmonic equations with critical exponents was later established by the second author and Xu in \cite{wei1999classification}.

Concerning sign-changing solutions to (\ref{1.yamabe}), it is noteworthy that Ding \cite{ding1986conformally} was the first to employ variational methods to construct infinitely many conformally inequivalent sign-changing solutions with finite energy. Since then, the existence of sign-changing solutions to the Yamabe equation in the entire space has become an active research topic. Del Pino-Musso-Pacard-Pistoia \cite{del2011large} introduced a novel constructive approach that produces sign-changing solutions with large energy, where the energy densities concentrate along specific submanifolds of \( \mathbb{S}^N \). In particular, they constructed a solution resembling a positive bubble but crowned by \( m \) negative spikes arranged in a regular polygon with radius 1. Remarkably, this solution is invariant under both rotation and the Kelvin transformation.  Denote by \( \Sigma\) the set of nonzero finite-energy solutions to \eqref{1.yamabe}:
\begin{equation}
\label{1.yamabeset}
\Sigma:=\left\{Q\in\mathcal{D}^{1,2}(\mathbb{R}^N\setminus\{0\}):~\Delta Q+|Q|^{\frac4{N-2}}Q=0\right\}.
\end{equation}
In fact, one can see that equation \eqref{1.yamabe} is invariant under the four transformations: translation, dilation, orthogonal transformation and Kelvin transformation, we refer the readers to Section 2.3 for the explicit definition of these transformations. If $Q\in\Sigma$ we denote
$$L_Q\cdot=\Delta\cdot +\frac{N+2}{N-2}|Q|^{\frac{4}{N-2}}\cdot$$
 as the linearized operator around $Q$. Define the kernel space of $L_Q$:
\begin{equation}
\label{1.kernel}
\mathcal{Z}(Q)=\left\{f\in\mathcal{D}^{1,2}(\mathbb{R}^N):L_Qf=0\right\}.
\end{equation}
The elements in $\mathcal{Z}(Q)$ generated by the family of the aforementioned four transformations  define the following space
\begin{equation}
\label{1.kernel-s}
\widetilde{\mathcal{Z}}_Q=\mathrm{Span}\left\{
\begin{aligned}
&(2-N)z_jQ+|z|^2\partial_{z_j}Q-2z_jz\cdot\nabla Q,~\partial_{z_j}Q,~1\leq j\leq N,\\
&(z_j\partial_{z_\ell}-z_\ell\partial_{z_j})Q,~1\leq j<\ell\leq N,~\frac{N-2}{2}Q+z\cdot \nabla Q
\end{aligned}
\right\}.
\end{equation}
One can verify that the dimension of $\widetilde{\mathcal{Z}}_Q$ is at most $$ 2N + 1 + \frac{N(N-1)}{2}.$$
However, for the positive radial bubble solution $u$ given in \eqref{1.bubble}, the dimension of $\widetilde{\mathcal{Z}}_Q$ is found to be $N+1$. In general, the dimension of the solution may be smaller than $2N+1 +\frac{N(N-1)}{2}$.  Duyckaerts-Kenig-Merle introduced the following definition of non-degeneracy for a solution of equation \eqref{1.yamabe} in \cite{Duyckaerts2016focusing}: a solution \( Q \in \Sigma \) is said to be \textit{non-degenerate} if
\begin{equation}
\label{1.nondegenerate}
\mathcal{Z}(Q) = \tilde{\mathcal{Z}}(Q).
\end{equation}
In \cite{musso2015nondegeneracy}, Musso-Wei established that the solution constructed by Del Pino-Musso-Pacard-Pistoia is non-degenerate in the sense of Duyckaerts-Kenig-Merle. Using this non-degeneracy property of the crown solution, Musso-Wei \cite{musso2016sign} constructed a sign-changing solution to the Barhi-Coron problem, while Deng-Musso-Wei \cite{deng2019new} presented a novel solution to the scalar curvature problem, and Deng-Musso \cite{deng2021high} provided a new construction of sign-changing solutions to Coron's problem. The notion of non-degeneracy is not only significant in the study of elliptic partial differential equations but also serves as a fundamental component in proving soliton resolution for solutions to the energy-critical wave equation. This resolution relies on the compactness property established by Kenig and Merle \cite{kenig2006global,kenig2008global}.  For dimensions \( N = 3,4,5 \), and under the assumption of non-degeneracy, Duyckaerts et al. \cite{duyckaerts2012universality, duyckaerts2017soliton} demonstrated that any nontrivial solution to the energy-critical wave equation asymptotically decomposes into a finite sum of stationary solutions and solitary waves, which are Lorentz transforms of the former.

\subsection{Sketch of Proofs.} Going back to the proof of Theorem \ref{th1}, we remark that a key difficulty in dimension three is that the decay of solutions at infinity is not fast enough. Specifically, the standard bubble \eqref{1.bubble} decays only algebraically as \( 1/|z| \), and solving a Poisson equation with a source term of the form \( 1/|z| \) would lead to an ansatz that grows algebraically at infinity. This challenge is significantly more difficult in three dimensions than in higher dimensions. (Another reason we avoid using positive bubbles is more complex. See the explanation at the end for further details.)   Therefore, it is crucial to construct a suitable building block that is both sign-changing and exhibits a faster decay than the standard bubble.

Following this principle, a key observation in this paper is that we can apply an inversion transformation at the nodal set of the crown solution constructed by Del Pino-Musso-Pacard-Pistoia (see also \cite{yuan2022construction}). This transformation plays a central role in our strategy. Two essential properties ensure that this solution serves as an appropriate candidate. First, the crown solution remains invariant under Kelvin transformation, and Musso-Wei \cite{musso2015nondegeneracy} have proved its non-degeneracy in the sense of Duyckaerts-Kenig-Merle \cite{Duyckaerts2016focusing}. Second, as we will show in Section 2 (Theorem \ref{th2.qm-nodal}) by lengthy computations, the crown solution $q$ satisfies \( \nabla q(z) \neq 0 \) for each point $z$ in its nodal set, thus the nodal set forms a smooth compact Riemann surface.  See Figure \ref{fig:levelset}. These two factors ensure the robustness of the gluing and reduction procedure.

Once the appropriate building block is established, the next step is to determine where we put the configurations and construct a suitable approximate solution. Our approach is as follows: we seek solutions that are rotationally invariant in the first two variables and exhibit even symmetry in the third one. This symmetry assumption allows us to eliminate certain kernel elements from \( \tilde{\mathcal{Z}}(Q) \). To proceed, we divide the unit ball into \( K \) subdomains, where \( K \) is chosen to be a sufficiently large even number, serving as the perturbation parameter in our analysis. (The idea of using the number of bubbles as a parameter seems first due to Wei-Yan \cite{WY2010}). Each subdomain is a replica of the following region:
\begin{equation}
    \Sigma_K= \left\{ (r, \theta,z_3)\in B_1; 0<r<1, -\frac{\pi}{K} < \theta <\frac{\pi}{K} \right\}.
\end{equation}
See Figure \ref{fig:SigmaK} for illustration.
\begin{figure}[htp]
    \centering
    \includegraphics[width=0.27\textwidth]{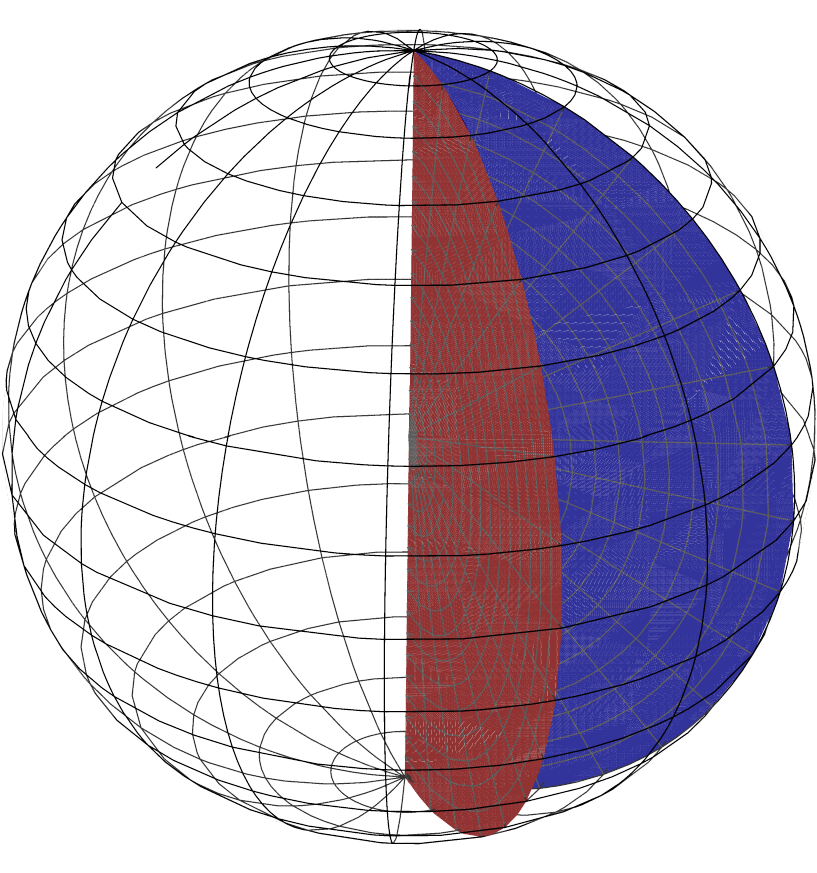}\hspace{2cm}
    \includegraphics[width=0.3\textwidth]{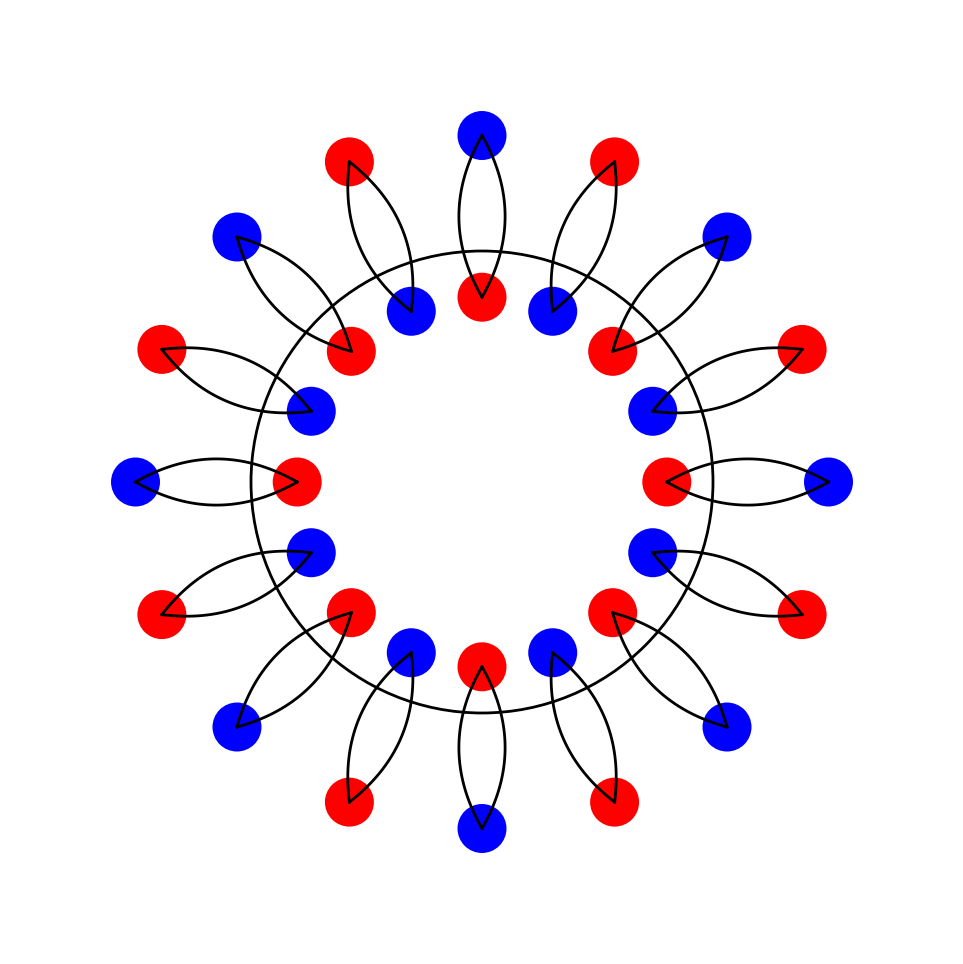}
    \caption{Left picture is the region $\Sigma_K$ and its  two boundaries inside $B_1$. Right one is the necklace solution.}
    \label{fig:SigmaK}
\end{figure}

 We will position a crown bubble (after applying the inversion transformation) in each sector in an alternating manner. In fact, they are arranged close to the boundary $\partial B_1$. These bubbles form {\bf necklace} pattern consisting alternating Del Pino-Musso-Pacard-Pistoia bubbles near the boundary. See Figure \ref{fig:SigmaK} for the right one.  Specifically, in the vicinity of the crown bubble located in \( \Sigma_K \), the neighboring two bubbles will have opposite signs. This setup reduces the problem to solving the following Dirichlet problem in \( \Sigma_K \):
\begin{equation}
\label{210}
\begin{cases}
\Delta u+\lambda u+ u^5=0 \ &\mbox{in} \quad \Sigma_K,\\
u=0 \ &\mbox{on} \quad \partial \Sigma_K.
\end{cases}
\end{equation}
With respect to \eqref{210}, we present the following existence result.
\begin{theorem}
\label{th1.sector}
    If $K$ is even and large enough, then \eqref{210} has a solution.
\end{theorem}
To prove Theorem \ref{th1.sector}, we must modify the alternating sum of crown bubbles (after inversion) to ensure that it satisfies the boundary conditions. As a result, the next-order correction terms \( T_A \), \( \varphi_A \), and \( \psi_A \) become essential, with their definitions provided in Section 3. In particular, \( T_A \) accounts for the influence of crown bubbles from the other sectors, while \( \varphi_A \) represents the effect of the mirror bubble of the one in \( \Sigma_K \) with respect to the boundary of the unit ball. Although the non-degeneracy of the solution has already been established by Musso-Wei in \cite{musso2015nondegeneracy}, analyzing the associated linearized problem remains challenging. The main difficulty arises from the necessity of positioning the center of the crown bubble close to the boundary. To tackle this, we decompose the perturbation into two components: one that resolves the problem in a small ball centered near the boundary and another that handles the remaining errors. By employing the classical inner-outer gluing method (see \cite{del2007concentration, kowalczyk2011giorgi}), we confirm that the resulting coupled system exhibits weak coupling, allowing us to reduce the original linearized problem to a standard form. Following the well-established procedure of analyzing the linearized operator and the nonlinear problem (see  Section 6), we further reduce the infinite-dimensional problem to a finite-dimensional one. For this finite-dimensional setting, we apply the localized energy method (see \cite{gui1999multiple}) to identify the critical points of the energy function with respect to the parameters.

We face two essential difficulties. First, due to the algebraic decay of the building blocks and the specific configuration of our approximation, the interaction between bubbles renders the energy computation particularly intricate. In particular, we must determine the asymptotic behavior of several types of series like
\begin{equation*}
\sum_{j=1}^{n-1}(-1)^{j}\left(x^2+\sin^2\tfrac{j\pi}{n}\right)^{-\frac{k}{2}}\quad\mbox{and}\quad
\sum_{j=1}^{n-1}\left(x^2+\sin^2\tfrac{j\pi}{n}\right)^{-\frac{k}{2}},\quad k\in\mathbb{N}.
\end{equation*}
(See Section 9.) It is fortunate that contour integration from complex analysis can be employed to derive an integral representation for the given series. Utilizing results on elliptic integrals, we establish a precise asymptotic behavior for these series under specific conditions on
$x$ and $n$. The detailed computations are provided in Section 9.

The second difficulty is that we have \textbf{six} parameter configuration spaces: they are respectively (1) scaling; (2) two translations; (3) one rotation; (4) two Kelvin transforms. These parameters have different scales, some of them may dominate others, so we have to find out the delicate small forces. After calculating the energy of the solution carefully, we analyze its leading order terms and identify the critical point through minimization, thereby proving the existence of a solution to \eqref{210}. The whole process is a careful balancing act, where we must figure out the proper order for each parameter. Roughly speaking, we find that the scaling parameter is of $O(K^{-3})$, the distance of the center of bubble to the boundary is of $O\left(K^{-1}\log K\right)$, the shift of the bubble in the $z_1$-axis is of $O(K^{-2}  \log K)$, and finally the rotation is found to be of $O(K^{-1} \log K)$. The Kelvin transform parameters are smaller in order.

Using the existence result from Theorem \ref{th1.sector}, we can get a solution to the Brezis-Nirenberg problem in $B_1$. Specifically, we extend $u$ by performing an odd reflection across each hyperplanes  $\theta=j\pi/K$, $j=1,3,5,\cdots,K-1$, consecutively. As a consequence, $u$ can be extended to a function in $B_1\setminus\{0\}$ and verifies \eqref{main-eqn} in $B_1\setminus\{0\}$. Since the solution we constructed is bounded, the singularity at the origin is removable. Consequently, we establish the existence of a nontrivial solution to \eqref{main-eqn}, thereby resolving Brezis's open problem.

\begin{remark}
The solution we constructed has a special shape, see Figure \ref{fig:SigmaK}, where blue and red colors represent the positive and negative bubbles, respectively.
\end{remark}

% \begin{figure}[ht]
%     \centering
%     \includegraphics[width=0.3\textwidth]{necklace.png}
%     \caption{Necklace solution}
%     \label{fig:necklace}
% \end{figure}

To end the introduction, we make a few comments. First, when
$\lambda$ is a sufficiently small positive number, an alternative type of solution can be constructed, where a single crown bubble is placed at the origin. We will report it elsewhere.  Second, for the Brezis-Nirenberg problem in higher dimensions, we can apply the same approach to construct infinitely many sign-changing solutions. Third, we believe that the procedure outlined in this paper can be adapted to other elliptic problems involving critical exponents. Finally, since the Del Pino-Musso-Pacard-Pistoia bubble has large Morse index,  the solutions we constructed carry very high Morse index. It remains open if one can give a positive answer to Brezis' Open Problem 1.1 by direct variational methods. The corresponding Brezis' problem in general three-dimensional domains is also challenging. By [Theorem 1.2, Brezis-Nirenberg \cite{brezis1983positive}], for strictly starshaped three dimensional domain $\Omega$, positive solution exists only if $ \lambda \geq \lambda_0 (\Omega)>0 $. (As a result and by a scaling argument, if we use positive bubble as building block in Theorem \ref{th1.sector}, then $ \lambda \geq c_0 K^2>0$.)   By Theorem \ref{th1}, we raise  the following generalized Brezis' Open Problem 1.1:

\medskip

\noindent
{\bf Generalized Brezis' Open Problem.} Let $\Omega $ be strictly star-shaped domain in ${\mathbb R}^3$. Then for every $\lambda>0$, problem (\ref{BN}) admits infinitely many solutions.

\medskip

The paper is organized as follows: In Section 2, we examine the crown solution constructed by Del Pino-Musso-Pacard-Pistoia \cite{del2011large} by analyzing its nodal set. In Sections 3 and 4 we introduce the approximate solution with correction terms and include the computation of the energy for the approximate solution. Section 5 presents the inner-outer gluing procedure, reducing the coupled system to the inner one. In Section 6, we consider the linear and nonlinear problems. In Section 7 we focus on reducing the infinite-dimensional problem to a finite-dimensional one, where we resolve the reduced problem by identifying the local minimum of the reduced energy and provide the proof of Theorem \ref{th1.sector}. Sections 8 and 9 focus on the delicate estimation of the correction terms and the analysis of certain series, respectively. Finally, some technical computations from Section 2 are included in the appendix.
%\vspace{1cm}

\section{On the Del Pino-Musso-Pacard-Pistoia solution}
In this section, we give a detailed study on the Del Pino-Musso-Pacard-Pistoia solutions to equation \eqref{1.yamabe}. Del Pino, Musso, Pacard, and Pistoia discovered an interesting sign-changing solution to (\ref{1.yamabe}) characterized by a large number $m$ of bubbles. This presents the first semi-explicit construction of its kind, utilizing an approach that naturally reveals spectral information related to the linearized problem. Because of its distinctive shape, they refer to it as the crown solution. Recently, this type of crown solution has played a significant role in understanding the long-term dynamics in the corresponding Schr\"odinger equation. (See \cite{yuan2022construction}.) Moreover, it serves as the foundation for the current work. In this section, we begin by describing this solution $q_m$ and providing a more refined analysis of its asymptotic behavior. Additionally, we examine the nodal set of this solution, demonstrating that it forms a smooth manifold, a property that proves crucial in addressing the reduction problem. Finally, we offer an interpretation for the non-degeneracy of the solution.

In \cite{del2011large}, the authors proved that there exists $m_0$ such that for all integer $m>m_0$, there exists a solution $q_m$ of \eqref{1.yamabe} that can be represented as follows
% \begin{equation}
% \label{2.qm}
% \begin{aligned}
% q_m(z)=~&U_*(z)+\phi(z)=U(z)-\sum_{j=1}^mU_j(z)
% +\phi(z)\\
% =~&U(z)-\sum_{j=1}^m\mu_m^{-\frac12}U\left(\frac{z-\xi_j}{\mu_m}\right)+\phi(z)\\
% =~&3^\frac14\left(\frac{1}{1+|z|^2}\right)^\frac12-3^\frac14\sum_{j=1}^m\mu_m^{-\frac12}\left(\frac{1}{1+\mu_m^{-2}|z-\xi_j|^2}\right)^\frac12+\phi(z),
% \end{aligned}
% \end{equation}
% where
% \begin{equation*}
% \phi(z)=\sum_{j=1}^m \tilde \phi_j(z)+\psi(z).
% \end{equation*}
\begin{align}\label{2.qm}
    q_m(z)=U_*(z)+\phi(z),
\end{align}
where
\begin{align*}
    U_*&=U(z)-\sum_{j=1}^mU_j(z)=U(z)-\sum_{j=1}^m\mu_m^{-\frac12}U\left(\frac{z-\xi_j}{\mu_m}\right)\\
    &=3^\frac14\left(\frac{1}{1+|z|^2}\right)^\frac12-3^\frac14\sum_{j=1}^m\mu_m^{-\frac12}\left(\frac{1}{1+\mu_m^{-2}|z-\xi_j|^2}\right)^\frac12,\\
    \phi(z)&=\sum_{j=1}^m \tilde \phi_j(z)+\psi(z).
\end{align*}
% In order to give a first description of this solution, we first introduce some notation. Let $m$ be an even number. For any integer $j=1,\cdots,m$, we define

For $U_*$, the parameters $\mu_m$ and  $\xi_j$, $j=1,\cdots,m$ are chosen as the following
\begin{equation}
\label{2.xij}
\xi_j=\sqrt{1-\mu_m^2}\left(\cos\frac{2(j-1)\pi}{m},\sin\frac{2(j-1)\pi}{m},0\right),
\end{equation}
and (for the choice of $d_m$, see \cite[Page 2590]{del2011large})
\begin{equation}
\label{2.mum}
\mu_m=\frac{d_m^2}{m^2(\log m)^2}\quad \mbox{and}\quad
%\begin{aligned}
d_m=\sqrt{2}\dfrac{m\log m}{\sum_{j=1}^{m-1}{\csc\frac{j\pi}{m}}}+O\left(\frac{1}{m\log m}\right).
%\end{aligned}
\end{equation}
% For the choice of $d_m$, see \cite[Page 2590]{del2011large}.
It follows from \cite[Theorem 7]{blagouchine2024finite} that
\[\sum_{j=1}^{n-1} \csc \left(\frac{j \pi}{m}\right)=\frac{2 m}{\pi}\left(\log \frac{2 m}{\pi}+\gamma_0\right)+O\left(\frac{1}{m}\right), \quad \text { as } n \rightarrow \infty.\]
where $\gamma_0=\lim_{n\to \infty}(\sum_{j=1}^n\frac{1}{j}-\log n)=0.577\cdots$ is the Euler's constant.
Thus
\begin{align}
    d_m=\frac{\sqrt{2}}{2}\pi +O\left(\frac{1}{\log m}\right)
\end{align}

For the error terms $\tilde\phi_j,~j=1,\cdots,m,$ and $\psi$, we have
\begin{equation}
\label{2.phi}
|\nabla^\ell \tilde\phi_j(z)|\leq \frac{C}{1+\mu_m^{-1}|z-\xi_j|^{\ell+1}},\quad \ell=0,1,2,\quad j=1,\cdots,m,
\end{equation}
and
\begin{equation}
\label{2.psi}
|\psi(z)|\leq \frac{C}{\log m}\quad \mbox{and}\quad |\nabla\psi(z)|+|\nabla^2\psi(z)|\leq C.
\end{equation}
%Though the estimates on the derivatives of $\tilde\phi_j$ and $\psi$ are not provided in \cite{del2011large}, one can derive them from the equation satisfied by these functions and standard elliptic estimates. From the estimation \eqref{2.phi} and \eqref{2.psi}, we can see that for the point away from the center of the negative bubbles, $\psi(z)$ is the dominated term, which is of order $O\left(\frac{1}{\log m}\right)$. Actually, the leading order of $\psi(z)$ will play an important role in determining the nodal set of $q_m$, and we need to figure out the exact value of $\psi$ of this order for $z$ on the circle of radius $\sqrt{1-\mu^2}$ lying on the $xy$ plane. In the following subsection, we shall compute the value.
Although \cite{del2011large} does not provide explicit estimates for the derivatives of $\tilde{\phi}_j$ and $\psi$, these can be obtained using the equations they satisfy along with standard elliptic estimates. From the estimates in \eqref{2.phi} and \eqref{2.psi}, it is obvious that, away from the centers of the negative bubbles, $\psi(z)$ is the dominant term, with an order of $O\left(1/{\log m}\right)$. In fact, the leading-order behavior of $\psi(z)$ plays a crucial role in determining the nodal set of $q_m$. Therefore, it is necessary to determine the precise value of $\psi$ in this order for the points $z$ lying on the circle of radius $\sqrt{1-\mu_m^2}$ on the $z_1z_2$ plane. In Subsection 2.1, we will carry out this computation.
\medskip

The following theorem provides precise estimates of the nodal set of the Del Pino-Musso-Pacard-Pistoia solution. This will be crucially used in later constructions.

\begin{theorem}
\label{th2.qm-nodal}
When $m$ is large enough, $q_m$ has a smooth embedding and compact nodal set $ \mathcal{N}(q_m)$ such that $q_m(z)=0$ and $\nabla q_m(z)\neq 0$ for any $z\in\mathcal{N}(q_m)$. Near the center of each bump $\xi_j$, $\mathcal{N}(q_m)\cap\{z_3=0\}\sim \{|z-\xi_j|\sim 1/m\}$. \footnote{Indeed, we can prove that near the center of each bump $\xi_j$, the distance between the zero point $z$ of $q_m$ and $\xi_j$ is on the order of $1/m$. However,
it is neither possible nor necessary for our proof to show that the nodal set resembles a ball in topology; it suffices to confirm that it is a smooth Riemannian manifold.}
\end{theorem}

\begin{remark}When analyzing the behavior of the approximate solution \( U_* \), the nodal set exhibits a structure resembling a torus, as illustrated in the left image of Figure 2. However, the strong influence of \( \psi \) also affects the behavior of \( q_m(z) \) for \( |z| = \sqrt{1 - \mu_m^2} \). Specifically, \( q_m(z) \) is positive at the point
\[
z = \left(\sqrt{1 - \mu_m^2} \cos\frac{\pi}{m}, \, \sqrt{1 - \mu_m^2} \sin\frac{\pi}{m}, \, 0 \right).
\]
Consequently, when considering the nodal set of \( q_m(z) \) in the \( z_1z_2 \)-plane, the numerical simulations reveal that it consists of \( m \) circles surrounding the center of each negative bubble.
\end{remark}

\begin{figure}[ht]
\centering
\includegraphics[width=0.35\textwidth]{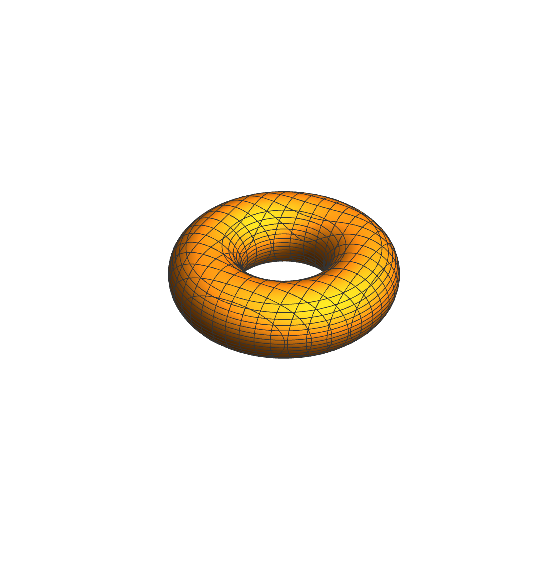}
\includegraphics[width=0.30\textwidth]{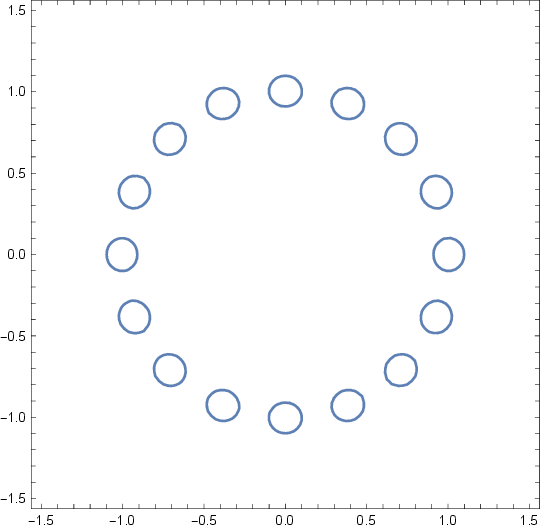}
\caption{The left one is the nodal sets of $U_*$ inside the cube $[-5/2,5/2]^3$ when $m=16$ and $d_m=\sqrt{2}m\log m\left(\sum_{j=1}^{m-1}\csc \frac{j\pi}{m}\right)^{-1}$. The right one is the intersection of nodal sets of $q_m$ with $z_1z_2$-plane.}
\label{fig:levelset}
\end{figure}

In Subsection 2.2 we shall give the proof of Theorem \ref{th2.qm-nodal} by analyzing the nodal set of $q_m(z)$ and its smoothness. First, we estimate the global term $\psi$ in the next subsection.

\subsection{Estimates of $\psi$}
Recall that $\psi$ satisfies (see \cite[(4.13)]{del2011large})
\begin{equation}
\label{2.psi-eq}
\begin{aligned}
\Delta \psi+5U^4\psi +V(z)\psi+5|U_*|^4\sum_j(1-\zeta_j)\tilde\phi_j+M(\psi)=0,
\end{aligned}
\end{equation}
where (we used the notations from there)
\begin{equation*}
V(z):=5(|U_*|^4-|U|^4)\left(1-\sum_{j=1}^m\zeta_j\right)-5U^4\sum_{j=1}^m\zeta_j,
\end{equation*}
\begin{equation*}
M(\psi):=\left(1-\sum_{j=1}^m\zeta_j\right)\left(E+N\left(\sum_{j=1}^m\tilde\phi_j+\psi\right)\right),
\end{equation*}
\begin{equation*}
E=\left(U-\sum_{j=1}^mU_j\right)^5-U^5-\sum_{j=1}^mU_j^5,
\end{equation*}
and
\begin{equation*}
\zeta_j(z)=\begin{cases}
\zeta(m\eta^{-1}|z|^{-2}|(z-\xi_j|z|)|),\quad &\mbox{if}\quad |z|>1,\\
\zeta(m\eta^{-1}|z-\xi_j|),\quad &\mbox{if}\quad |z|\leq1.
\end{cases}
\end{equation*}
Here $\zeta(s)$ is a smooth function such that $\zeta(s)=1$ for $s<1$ and $\zeta(s)=0$ for $s>2$, and $\eta$ is a small positive constant. It is not difficult to check that
\begin{equation}
\label{2.kelvin-f}
U(z)=\frac{1}{|z|}U\left(\frac{z}{|z|^2}\right),\quad U_j(z)=\frac{1}{|z|}U_j\left(\frac{z}{|z|^2}\right),~j=1,\cdots,m,
\end{equation}
and
\begin{equation}
\label{2.kelvin-f1}
\zeta_j(z)=\zeta_j\left(\frac{z}{|z|^2}\right),~j=1,\cdots,m.
\end{equation}
Together with the properties that the ansatz is $2\pi/m$-rotationally invariant in $z_1z_2$ plane and even symmetry with respect to the $z_3$ axis, one can solve $\psi$ from \eqref{2.psi-eq} uniquely in a suitable Sobolev space, which is the orthogonal complement of the space spanned by the kernels of the linear operator $\Delta+5U^4$. More precisely, the authors in \cite{del2011large} proved that
\begin{equation*}
\|\psi\|_{L_w^\infty}\leq C\left(\|V(z)\psi\|_{L_w^q}+5\left\||U_*|^4\sum_j\left(1-\zeta_j\right)\tilde\phi_j\right\|_{L_w^q}+\|M(\psi)\|_{L_w^q}\right),
\end{equation*}
where
\begin{equation*}
\|\psi\|_{L_w^\infty}:=\left\|(1+|z|)\psi\right\|_{L^\infty(\mathbb{R}^3)}\quad\mbox{and}\quad
\|h\|_{L_w^q}:=\left\|(1+|z|)^{5-\frac{6}{q}}h\right\|_{L^q(\mathbb{R}^3)}.
\end{equation*}
Particularly, for $q>\frac32$ it has been shown that
\begin{equation*}
\begin{aligned}
\left\|\left(1-\sum_{j=1}^m\zeta_j\right)
N\left(\sum_{j=1}^m\tilde\phi_j+\psi\right)\right\|_{L_w^q}
+\left\||U_*|^4\sum_j\left(1-\zeta_j\right)\tilde\phi_j\right\|_{L_w^q}+\|V(z)\psi\|_{L_w^q}\leq \frac{C}{(\log m)^2}.
\end{aligned}
\end{equation*}
In order to capture the leading behavior of $\psi$, we decompose $\psi=\psi_{d}+\psi_{s}$, where $\psi_d$ and $\psi_s$ are the solutions are used for solving the major error term $\left(1-\sum_{j=1}^m\zeta_j\right)E$ and the left terms in \eqref{2.psi-eq} respectively, i.e.,
\begin{equation}
\label{2.psi-m}
\Delta \psi_d+5U^4\psi_d+\left(1-\sum_{j=1}^m\zeta_j\right)E=0,
\end{equation}
and
\begin{equation}
\label{2.psi-s}
\Delta \psi_s+5U^4\psi_s+V(z)\psi+5|U_*|^4\sum_j(1-\zeta_j)\tilde\phi_j+M(\psi)-\left(1-\sum_{j=1}^m\zeta_j\right)E=0.
\end{equation}
It is straightforward to verify that both equations are solvable, since the source terms are orthogonal to the kernel of the linear operator $\Delta + 5U^4$. Moreover, if we impose the condition that both functions $\psi_d$ and $\psi_s$ lie in the orthogonal complement of the kernel of $\Delta + 5U^4$, it can be shown that
\begin{equation}
\label{2.psi-s-est}
\|\psi_{s}\|_{L^\infty(\mathbb{R}^3)}\leq\frac{C}{(\log m)^2}.
\end{equation}
In the following, we study $\psi_d$. By diving it into two parts $\psi_{d,1}$ and $\psi_{d,2}$
\begin{equation}
\label{2.psi-m1}
\Delta\psi_{d,1}+5U^4\psi_{d,1}-5U^4\sum_{j=1}^m\frac{3^\frac14\mu_m^\frac12}{|z-\xi_{j,0}|}=0,
\end{equation}
and
\begin{equation}
\label{2.psi-m2}
\Delta\psi_{d,2}+5U^4\psi_{d,2}+5U^4\sum_{j=1}^m\frac{3^\frac14\mu_m^\frac12}{|z-\xi_{j,0}|}+\left(1-\sum_{j=1}^m\zeta_j\right)E=0,
\end{equation}
where $\xi_{j,0}=\left(\cos\frac{2(j-1)\pi}{m},\sin\frac{2(j-1)\pi}{m},0\right)$. Note that $\xi_{j,0}$ is on the unit circle, while $\xi_j$ (see \eqref{2.xij}) is on the circle of radius $\sqrt{1-\mu^2}$. For \eqref{2.psi-m2}, one can check that
\begin{equation}
\label{2.psi-m2-0}
\left\|5U^4\sum_{j=1}^m\frac{3^\frac14\mu_m^\frac12}{|x-\xi_{j,0}|}+\left(1-\sum_{j=1}^m\zeta_j\right)E\right\|_{L_w^q}\leq \frac{C}{(\log m)^2}.
\end{equation}
The proof of (\ref{2.psi-m2-0}) is lengthy and thus left to Appendix \ref{appendix-A}.  From (\ref{2.psi-m2-0}) we see that
\begin{equation}
\label{2.psi-d2-est}
|\psi_{d,2}|_{L^\infty(\mathbb{R}^3)}\leq\frac{C}{(\log m)^2}.
\end{equation}
It remains to study $\psi_{d,1}$. In fact, we can find the explicit form of the solution for $z=\left(\cos\frac{\pi}{m},\sin\frac{\pi}{m},0\right)$. Indeed, we write
\begin{equation}
\psi_{d,1}=3^\frac14\mu_m^\frac12\sum_{j=1}^{\frac{m}{2}}\left(\psi_{d,1,j}+\frac{1}{|z-\xi_{j,0}|}+\frac{1}{|z-\xi_{j+\frac{m}{2},0}|}\right),
\end{equation}
where $\psi_{d,1,j}$ verifies
\begin{equation}
\Delta\psi_{d,1,j}+5U^4\psi_{d,1,j}=4\pi\left(\delta_{\xi_{j,0}}+\delta_{\xi_{j+\frac{m}{2},0}}\right),\quad j=1,\cdots,\frac{m}{2}.
\end{equation}
where $\delta_\xi$ is the Dirac function at $\xi$.
Particularly $\psi_{d,1,1}$ is given below
\begin{equation}
\psi_{d,1,1}(z)=-\frac{\sqrt{2}}{(1+|z|^2)^\frac12}\frac{8z_1^2-(1+|z|^2)^2}{(1+|z|^2)\sqrt{(1+|z|^2)^2-4z_1^2}},
\end{equation}
the other $\psi_{d,1,j},~j=2,\cdots,\frac{m}{2}$ can be derived from $\psi_{d,1,1}$ by a rotation of $\frac{2(j-1)\pi}{m}$. At $z=\hat\xi_0=\left(\cos\frac{\pi}{m},\sin\frac{\pi}{m},0\right)$, we can compute the value of $\psi_{d,1}$ and get that
\begin{equation}
\label{2.psim1-middle}
\begin{aligned}
\psi_{d,1}\left(\xi_0\right)
=~&{3^\frac14\mu_m^\frac12}\sum_{j=1}^{\frac{m}{2}}\left(
\frac{1-2\cos^2\left(\frac{2(j-1)\pi}{m}-\frac{\pi}{m}\right)}{\left|\sin\left(\frac{2(j-1)\pi}{m}-\frac{\pi}{m}\right)\right|}
+\frac{1}{\sin\left(\frac{(2j-1)\pi}{2m}\right)}\right)\\
=~&{3^\frac14\mu_m^\frac12}\sum_{j=1}^{\frac{m}{2}}\left(\frac{1}{\sin\left(\frac{(2j-1)\pi}{2m}\right)}
-\frac{1}{\left|\sin\left(\frac{2(j-1)\pi}{m}-\frac{\pi}{m}\right)\right|}\right)\\
&+{3^\frac14\mu_m^\frac12}\sum_{j=1}^{\frac{m}{2}}2\left|\sin\left(\frac{2(j-1)\pi}{m}-\frac{\pi}{m}\right)\right|.
\end{aligned}
\end{equation}
Together with the estimation \eqref{2.psi-s-est}, \eqref{2.psi-d2-est}, and \eqref{2.psi}, we have
\begin{equation}
\label{2.psi-middle}
\psi\left(\sqrt{1-\mu_m^2}\cos\frac{\pi}{m},\sqrt{1-\mu_m^2}\sin\frac{\pi}{m},0\right)=\psi_{d,1}(\xi_0)+O\left(\frac{1}{(\log m)^2}\right).
\end{equation}
We shall see that it plays an important role in determining the nodal set of the crown bubble solution $q_m$.

\subsection{On the nodal set of Del Pino-Musso-Pacard-Pistoia's solution}
In this subsection, we will study the nodal set of the solution by Del Pino-Musso-Pacard-Pistoia. To begin with, we present the following lemma to describe the location of the nodal set.
\begin{lemma}
\label{2.lemma-1}
When $m$ is large enough, $q_m$ has a compact nodal set $\mathcal{N}(q_m)$ such that for any  point $z$ in $\mathcal{N}(q_m)$, the following holds
$${\rm dist}\left(z,\left\{\xi_1,\cdots,\xi_m\right\}\right)\sim \frac{1}{m}.$$
In addition, there exists a small positive constant $c_0$ such that
\begin{equation}
\label{2.lem-eq}
\mathcal{N}(q_m)\cap \bigcup_{j=1}^{m} B_{c_0/m}\left(\sqrt{1-\mu_m^2}\cos\frac{2j-1}{m}\pi,\sqrt{1-\mu_m^2}\sin\frac{2j-1}{m}\pi,0\right)=\emptyset.
\end{equation}
\end{lemma}

\begin{proof}
Without loss of generality, we write
$z=(r\cos\theta,r\sin\theta,z_3)$ and assume that
$$|z-\xi_1|=\min\{|z-\xi_j|,~j=1,\cdots,m\},\quad\mbox{and}\quad \theta\in\left[0,\frac{\pi}{m}\right].$$
If $|z-\xi_1|\geq C$, then for sufficiently large $m$ we have
\begin{align*}
q_m(z)\geq 3^\frac14-3^\frac14\sum_{j=1}^m\frac{{\mu_m^\frac12}}{|z-\xi_j|}-O\left(\frac{1}{\log m}\right)>0.
\end{align*}
Hence, the zero point $z$ of $q_m(z)$ must be close to $\xi_1$. As a consequence, we may restrict our discussion for the point $z$ satisfying
\begin{equation}
\label{2.r-asy}
|z-\xi_1|=o_m(1)\quad\mbox{and}\quad |r-1|=o_m(1).
\end{equation}
In the following, we divide our discussion into three steps. First, we will prove that $q_m(z)$ is negative in the vicinity of $\xi_1$. The key point here is that when $z$ is extremely close to $\xi_1$, the effect of the other negative bubbles offsets the central positive bubble, making the negative bubble at $\xi_1$ the dominant term determining the sign of $q_m$. Next, we introduce a quantity $h$ (refer to \eqref{2.com-2}), which measures the distance from the point $z$ to the center of the negative bubble. We will show that when $h$ is not too small, the major contributions from all negative bubbles form a series. This series can be assessed using an integral representation and the asymptotic behavior of the elliptic integral, enabling us to figure out the leading term in the total sum of negative bubbles. Finally, we will analyze the behavior of $q_m(z)$ at the midpoint between $\xi_1$ and $\xi_2$, denoted by $\xi_0$. By using the estimate of $\psi$ at $\xi_0$ from \eqref{2.psi-middle}, we will compare $\psi$ to $U_* = U - \sum_{j=1}^m U_j$ and establish that $q_m$ is positive at $\xi_0$. Together with the evaluation of $\nabla q_m$, this analysis reveals a small region (a tiny ball) where $q_m$ remains positive.
\medskip

\noindent {\bf Step 1}. In this step, we shall prove that $q_m$ is negative within a small neighborhood of $\xi_1$,  i.e.,
\begin{equation}
\label{2.claim-0}
q_m(z)<0\quad \mbox{for}\quad |z-\xi_1|=\min_{j\in\{1,\cdots,m\}}|z-\xi_j|\leq \frac{c_1}{m},
\end{equation}
provided $c_1$ is sufficiently small. For $|z-\xi_1|\leq\frac{c_1}{m}$,  we have
$$|z-\xi_1|\leq \frac12|\xi_j-\xi_1|\quad\mathrm{and}\quad |z-\xi_j|\leq |\xi_j-\xi_1|+|z-\xi_1|,\quad j=2,\cdots,m.$$
Then using the simple inequality $\frac{1}{1+x}\geq1-x$ for $x\in(0,1)$ we have
\begin{equation}
\label{2.com-4}
\begin{aligned}
\left(\frac{1}{1+\mu_m^{-2}|z-\xi_j|^2}\right)^{\frac12}
&=\frac{\mu_m}{|z-\xi_j|}+O\left(\frac{1}{m^3(\log m)^6}\right)\\
&\geq\frac{\mu_m}{|\xi_j-\xi_1|+|z-\xi_1|}+O\left(\frac{1}{m^3(\log m)^6}\right)\\
&\geq \frac{\mu_m}{|\xi_j-\xi_1|}-\frac{\mu_m|z-\xi_1|}{|\xi_j-\xi_1|^2}+O\left(\frac{1}{m^3(\log m)^6}\right),~
\end{aligned}
\end{equation}
for any $j=2,\cdots,m$. Similarly, we can apply the inequality $\frac{1}{1-x}\leq 1+2x$ for $x\in\left(0,\frac12\right)$ to derive that
\begin{equation*}
\begin{aligned}
\frac{\mu_m}{|z-\xi_j|}\leq \frac{\mu_m}{|\xi_j-\xi_1|}+\frac{2\mu_m|z-\xi_1|}{|\xi_j-\xi_1|^2},\quad j=2,\cdots,m.
\end{aligned}
\end{equation*}
Together with \eqref{2.phi} and \eqref{2.psi}, we have
\begin{equation}
\label{2.com-5}
\begin{aligned}
q_m(z)\leq~&3^{\frac14}\left(\left(\frac{1}{1+|z|^2}\right)^\frac12
-\sum_{j=2}^m\frac{\mu_m^\frac12}{|\xi_1-\xi_j|}+\sum_{j=2}^m\frac{\mu_m^\frac12|z-\xi_1|}{|\xi_1-\xi_j|^2}\right)\\
&-3^\frac14\mu_m^{-\frac12}\left(\frac{1}{1+\mu_m^{-2}|z-\xi_1|^2}\right)^{\frac12}+O\left(\frac{1}{m(\log m)^5}\right)\\
&+O\left(\frac{1}{\log m}\right)+O\left(\frac{1}{1+\mu_m^{-1}|z-\xi_1|}\right)\\
&+O\left(\sum_{j=2}^m\frac{\mu_m}{|\xi_j-\xi_1|}+\sum_{j=1}^m\frac{\mu_m|z-\xi_1|}{|\xi_j-\xi_1|^2}\right).
\end{aligned}
\end{equation}
On the right hand side of \eqref{2.com-5}, using $|z-\xi_1|\leq\frac{c_1}{m} $we have
\begin{equation*}
\frac{1}{1+|z|^2}=\frac{1}{1+|\xi_1|^2}+O\left(\frac{1}{m}\right)
=\frac12+O\left(\frac{1}{m}\right),
\end{equation*}
the second and third terms in the bracket can be handled respectively as follows
$$\sum_{j=2}^m\frac{\mu_m^\frac12}{|\xi_1-\xi_j|}=\frac{\sqrt{2}}{2}+O\left(\frac{1}{m\log m}\right),$$
and
\begin{equation*}
\sum_{j=2}^m\frac{|z-\xi_1|}{m\log m|\xi_1-\xi_j|^2}\leq C\frac{m|z-\xi_1|}{\log m}.
\end{equation*}
Consider the fourth term on the right hand side of \eqref{2.com-5}, we have
\begin{equation*}
\mu_m^{-\frac12}\left(\frac{1}{1+\mu_m^{-2}|z-\xi_1|^2}\right)^{-\frac12}
\geq \min\left\{\frac{\sqrt{2}}{2}\mu_m^{-\frac12},
\frac{\sqrt{2}}{2}\frac{\mu_m^{\frac12}}{|z-\xi_1|}\right\}.
\end{equation*}
Putting all these information on the right hand side of \eqref{2.com-4} we get
\begin{equation}
\begin{aligned}
\label{2.com-6}
q_m(z)\leq &-3^\frac14\min\left\{\frac{\sqrt{2}}{2}\mu_m^{-\frac12},\frac{\sqrt{2}}{2}\frac{\mu_m^{\frac12}}{|z-\xi_1|}\right\}
+\frac{Cm|z-\xi_1|}{\log m}\\
&+O\left(\frac{1}{\log m}\right)+O\left(\min\left\{1,\frac{1}{\mu_m^{-1}|z-\xi_1|}\right\}\right).
\end{aligned}
\end{equation}
If $|z-\xi_1|\leq c_1/m$ with $c_1$ is small enough, we can see that the right-hand side of \eqref{2.com-6} is strictly negative for $m$ sufficiently large. Thus, we finish the proof of the claim \eqref{2.claim-0}.
\medskip

\noindent {\bf Step 2}. Let
\begin{equation}
\label{2.com-2}
h^2=\frac{(r-\sqrt{1-\mu_m^2})^2+z_3^2}{4r\sqrt{1-\mu_m^2}}.
\end{equation}
In this step we will show that
\begin{equation}
\label{2.claim-10}
q_m(z)~\mbox{is positive for}~h\geq \frac{C_0}{m}~\mbox{for}~C_0~\mbox{sufficiently large}.
\end{equation}
	To begin with, we will prove that
		\begin{equation}
			\label{2.claim-1}
			q_m(z)~\mbox{is positive for}~h\geq m^{-\alpha_1}~\mbox{with any}~\alpha_1<1.
		\end{equation}
		Based on the definition of $h$, we can express $|z-\xi_j|^2$ as follows
		\begin{equation}
			\label{2.com-1}
			\begin{aligned}
				|z-\xi_j|^2=4r\sqrt{1-\mu_m^2}\left(h^2+\sin^2\left(\frac{(j-1)\pi}{m}-\frac{\theta}{2}\right)\right).
			\end{aligned}
		\end{equation}
		Using \eqref{2.phi}, \eqref{2.com-4} and \eqref{2.com-1} we have
		\begin{equation}
			\label{2.com-7}
			\begin{aligned}
			q_m(z)=~&3^\frac14\left(\frac{1}{1+|z|^2}\right)^\frac12-
				\frac{3^\frac14\mu_m^{\frac12}}{2(r\sqrt{1-\mu_m^2})^\frac12}\sum_{j=1}^m\frac{1}{\left(h^2+\sin^2\left(\frac{(j-1)\pi}{m}-\frac{\theta}{2}\right)\right)^\frac12}\\
				&+\psi(z)+O\left(\frac{1}{m\log m}\right).
			\end{aligned}
		\end{equation}
		For the second term on the right hand side of \eqref{2.com-7}, we have
		\begin{equation}
		\label{2.com-7-1}
		\begin{aligned}
		\sum_{j=1}^m\frac{1}{\left(h^2+\sin^2\left(\frac{(j-1)\pi}{m}-\frac{\theta}{2}\right)\right)^\frac12}
		\leq~& \sum_{j=m_1+3}^{m-m_1}\frac{1}{\sin\left(\frac{(j-1)\pi}{m}-\frac{\theta}{2}\right)}+\frac{2m_1+2}{h}\\
		\leq~&\sum_{j=m_1+1}^{m-m_1-1}\frac{1}{\sin\frac {j\pi}{m}}+\frac{2m_1+2}{h},
		\end{aligned}
		\end{equation}
		where $m_1=\lfloor m^{1-\alpha_1}\rfloor.$ \footnote{$\lfloor x\rfloor$ is the greatest integer less than $x.$}
		Substituting  \eqref{2.mum} and \eqref{2.com-7-1} into \eqref{2.com-7} we have
		\begin{equation}
		\label{2.com-7-2}
		\begin{aligned}
		q_m(z)\geq~&3^\frac14\frac{\sqrt{2}}{2}-\frac{3^\frac14}{2}\mu_m^\frac12
		\sum_{j=m_1+1}^{m-m_1-1}\frac{1}{\sin\frac{j\pi}{m}}-3^\frac14\mu_m^\frac12\frac{m_1+1}{h}+o_m(1)\\
		\geq ~&3^\frac14\mu_m^\frac12\sum_{j=1}^{m_1}\frac{1}{\sin\frac{j\pi}{m}}
		+O\left(m\mu_m^\frac12\right)+o_m(1)\\
		\geq~&C\frac{\log m_1}{\log m}+o_m(1)>0,
		\end{aligned}
		\end{equation}
		where we also used that $|\psi|\leq \frac{C}{\log m}$ and
		$$\sum_{j=1}^{m_1}\frac{1}{j}\geq C\log m_1.$$		
		Next, we show that $q(z)$ is positive for $m^{-\alpha_1}\geq h\geq\frac{C_0}{m}$ for $C_0$ sufficiently large, where $m$ is large enough such that $m^{-\alpha_1}\geq\frac{C_0}{m}$. Then we have
		$$|z-\xi_1|\leq Cm^{-\alpha_1}.$$
		It implies
		$$|z|=1+O(m^{-\alpha_1})\quad\mbox{and}\quad r=1+O(m^{-\alpha_1}).$$
		Following almost the same argument as we did in \eqref{2.com-7-1}, we have
		\begin{equation}
				\label{2.com-12-red}
				\begin{aligned}
					q_m(z)
				%\geq~&3^\frac14\frac{\sqrt{2}}{2}-\frac{3^\frac14\mu_m^{\frac12}}{2(r\sqrt{1-\mu_m^2})^\frac12}\sum_{j=1}^m\frac{1}{\left(h^2+\sin^2\left(\frac{(j-1)\pi}{m}\right)\right)^\frac12}\\
					%&+\psi(z)+O\left(\frac{1}{m\log m}\right)\\
			\geq~&3^\frac14\frac{\sqrt{2}}{2}-3^\frac14\frac{\mu_m^{\frac12}}{2(r\sqrt{1-\mu_m^2})^\frac12}\sum_{j=m_2+1}^{m-m_2-1}\frac{1}{\sin\frac{j\pi}{m}}-3^\frac14\frac{\mu_m^{\frac12}(m_2+1)}{(r\sqrt{1-\mu_m^2})^\frac12h}\\
					&+\psi(z)+O\left(\frac{1}{m\log m}\right),
				\end{aligned}
			\end{equation}
			where $m_2$ is determined later. Using \eqref{2.mum}, we see that
			\begin{equation}
				\label{2.com-13-red}
				\begin{aligned}
					\frac{\sqrt{2}}{2}-\frac{\mu_m^{\frac12}}{2(r\sqrt{1-\mu_m^2})^\frac12}\sum_{j=m_2+1}^{m-m_2-1}\frac{1}{\sin\frac{j\pi}{m}}
					=\frac{d_m}{m\log m}\sum_{j=1}^{m_2}\frac{1}{\sin\frac{j\pi}{m}}
					+O\left(\frac{1}{m\log m}\right).
				\end{aligned}
			\end{equation}
			From \eqref{2.com-12-red} and \eqref{2.com-13-red} we have
			\begin{equation}
				\label{2.com-14-red}
				\begin{aligned}
					q_m(z)\geq~&3^\frac14\frac{d_m}{m\log m}\sum_{j=1}^{m_2}\frac{1}{\sin\frac{j\pi}{m}}
					-3^\frac14\frac{\mu_m^{\frac12}(m_2+1)}{(r\sqrt{1-\mu_m^2})^\frac12h}+\psi(z)+O\left(\frac{1}{m\log m}\right)\\
					\geq~&C\frac{\log m_2}{\log m}-C\frac{m_2 }{C_0\log m}+O\left(\frac{1}{\log m}\right),
				\end{aligned}
			\end{equation}
where we used $|\psi| \leq \frac{C}{\log m}$ and $h \geq \frac{C_0}{\log m}$. By first selecting a sufficiently large $m_2$ and then $C_0$ large enough, we can ensure that the right-hand side of \eqref{2.com-14-red} is positive. Together with \eqref{2.claim-1}, this allows us to establish \eqref{2.claim-10}, thereby completing the proof of this step.
\medskip

\noindent {\bf Step 3}. In this step, we shall show \eqref{2.lem-eq}. For convenience, we set $$\xi_0=\left(\sqrt{1-\mu_m^2}\cos\frac{\pi}{m},\sqrt{1-\mu_m^2}\sin\frac{\pi}{m},0\right).$$
%Using the estimation of $\psi$ at $\xi_0$, as shown in \eqref{2.psi-middle}, we will compare $\psi$ with $U_* = U - \sum_{j=1}^m U_j$ and prove that $q_m$ is positive at $\xi_0$. Combined with the estimation of $\nabla q_m$, this allows us to identify a small region (a tiny ball) where $q_m$ remains positive.
First, we prove that $q_m(\xi_0)$ is positive. After straightforward calculation we have
\begin{equation}
\label{2.qmiddle-1}
\begin{aligned}
q_m(\xi_0)=~&3^\frac14\frac{\sqrt{2}}{2}-3^\frac14\mu_m^\frac12\sum_{j=1}^m\frac{1}{|\xi_0-\xi_j|}+\psi(\xi_0)+O\left(\frac{1}{m\log m}\right)\\
=~&3^\frac14\frac{\sqrt{2}}{2}-3^\frac14\mu_m^\frac12\sum_{j=1}^\frac{m}{2}\frac{1}{\sin\frac{(2j-1)\pi}{2m}}+\psi(\xi_0)+O\left(\frac{1}{m\log m}\right).
\end{aligned}
\end{equation}
By \eqref{2.psi-middle} and
\begin{equation}
\label{2.right-1}
\begin{aligned}
\sum_{j=1}^{\frac{m}2}2\left|\sin\left(\frac{2(j-1)\pi}{m}-\frac{\pi}{m}\right)\right|
=2\sin\frac{\pi}{m}+\frac{1-\cos\frac{(m-2)\pi}{m}}{\sin\frac{\pi}{m}}
=\frac{2m}{\pi}+O\left(\frac{1}{m}\right),
\end{aligned}
\end{equation}
we have
\begin{equation}
\label{2.qmiddle-2}
\begin{aligned}
q_m(\xi_0)=~3^\frac14\frac{\sqrt{2}}{2}-3^\frac14\mu_m^\frac12
\sum_{j=1}^{\frac{m}{2}}\frac{1}{\left|\sin\left(\frac{2(j-1)\pi}{m}-\frac{\pi}{m}\right)\right|}+3^\frac14\mu_m^\frac12\frac{2m}{\pi}+O\left(\frac{1}{m\log m}\right).
\end{aligned}
\end{equation}
Recall that (see \eqref{2.mum})
\begin{equation*}
\mu_m^\frac12=\frac{d_m}{m\log m}=\sqrt{2}\left(\sum_{j=1}^{m-1}\frac{1}{\sin\frac{j\pi}{m}}\right)^{-1}+O\left(\frac{1}{m\log m}\right),
\end{equation*}
then
\begin{equation}
\label{2.qmiddle-3}
\begin{aligned}
&\frac{\sqrt{2}}{2}-\mu_m^\frac12\sum_{j=1}^{\frac{m}{2}}\frac{1}{\left|\sin\left(\frac{2(j-1)\pi}{m}-\frac{\pi}{m}\right)\right|}
				\\
&=\mu_m^\frac12\sum_{j=1}^{\frac{m}{2}}\frac{1}{\sin\frac{j\pi}{m}}
-\mu_m^\frac12\sum_{j=1}^{\frac{m}{2}}\frac{1}{\left|\sin\left(\frac{2(j-1)\pi}{m}-\frac{\pi}{m}\right)\right|}
+O\left(\frac{1}{m\log m}\right)\\
&\geq -\mu_m^\frac12\frac{1}{\sin\frac{\pi}{m}}+O\left(\frac{1}{m\log m}\right),
\end{aligned}
\end{equation}
where we used that
\begin{equation*}
\frac{1}{\sin\frac{2j\pi}{m}}
\geq \frac{1}{\sin\frac{\left(m-2j-1\right)\pi}{m}},\quad j=1,2,\cdots,\frac{m}{2}.
\end{equation*}
Substituting \eqref{2.qmiddle-3} into \eqref{2.qmiddle-2} we have
\begin{equation}
\label{2.qmxi0}
q_m(\xi_0)\geq 3^\frac14\mu_m^\frac12\frac{m}{\pi}+O\left(\mu_m\right)\geq \frac{3^\frac14d_m}{2\pi\log m}.
\end{equation}
Consider the derivative of $q_m(z)$. By direct computation we have
\begin{equation}
\label{2.qm-d}
\begin{aligned}
\nabla q_m(z)=- \frac{3^\frac14z}{(1+|z|^2)^\frac32}+\sum_{j=1}^m\frac{3^\frac14\mu_m^{-\frac52}(z-\xi_j)}{(1+\mu_m^{-2}|z-\xi_j|^2)^\frac32}+\nabla\psi(z)+\sum_{j=1}^m\nabla\tilde\phi_j(z).
\end{aligned}
\end{equation}
For $z\in B_{c_2m^{-1}}\left(\xi_0\right)$ with $c_2$ small, we write
$$\nabla q_m(z)=3^\frac14(q_{m,1}(z),q_{m,2}(z),q_{m,3}(z)),$$
where
\begin{equation}
\label{2.cal-1}
\begin{aligned}
q_{m,1}(z)=~&-\frac{\sqrt{2}z_1}{4}+\mu_m^{\frac12}\sum_{j=1}^m\frac{(r-\sqrt{1-\mu_m^2})\cos\theta}{(4r\sqrt{1-\mu_m^2})^\frac32\left(h^2+\sin^2\left(\frac{(j-1)\pi}{m}-\frac{\theta}{2}\right)\right)^\frac32}\\
&+\mu_m^{\frac12}\sum_{j=1}^m
\frac{2\sqrt{1-\mu_m^2}\sin\left(\frac{(j-1)\pi}{m}+\frac{\theta}{2}\right)\sin\left(\frac{(j-1)\pi}{m}-\frac{\theta}{2}\right)}{(4r\sqrt{1-\mu_m^2})^\frac32\left(h^2+\sin^2\left(\frac{(j-1)\pi}{m}-\frac{\theta}{2}\right)\right)^\frac32}\\
&+3^{-\frac14}\partial_{z_1}\psi+3^{-\frac14}\sum_{j=1}^m\partial_{z_1}\tilde\phi_j+O\left(\frac{1}{m(\log m)^5}\right),
\end{aligned}
\end{equation}
\begin{equation}
\label{2.cal-2}
\begin{aligned}
q_{m,2}(z)=~&-\frac{\sqrt{2}z_2}{4}+\mu_m^{\frac12}\sum_{j=1}^m\frac{(r-\sqrt{1-\mu_m^2})\sin\theta}{(4r\sqrt{1-\mu_m^2})^\frac32\left(h^2+\sin^2\left(\frac{(j-1)\pi}{m}-\frac{\theta}{2}\right)\right)^\frac32}\\
&-\mu_m^{\frac12}\sum_{j=1}^m
\frac{2\sqrt{1-\mu_m^2}\cos\left(\frac{(j-1)\pi}{m}+\frac{\theta}{2}\right)\sin\left(\frac{(j-1)\pi}{m}-\frac{\theta}{2}\right)}{(4r\sqrt{1-\mu_m^2})^\frac32\left(h^2+\sin^2\left(\frac{(j-1)\pi}{m}-\frac{\theta}{2}\right)\right)^\frac32}\\
&+3^{-\frac14}\partial_{z_2}\psi+3^{-\frac14}\sum_{j=1}^m\partial_{z_2}\tilde\phi_j+O\left(\frac{1}{m(\log m)^5}\right),
\end{aligned}
\end{equation}
and
\begin{equation}
\label{2.cal-3}
\begin{aligned}
q_{m,3}(z)=~&-\frac{\sqrt{2}z_3}{4}+\mu_m^{\frac12}\sum_{j=1}^m\frac{z_3}{(4r\sqrt{1-\mu_m^2})^\frac32\left(h^2+\sin^2\left(\frac{(j-1)\pi}{m}-\frac{\theta}{2}\right)\right)^\frac32}\\
&+3^{-\frac14}\partial_{z_3}\psi+3^{-\frac14}\sum_{j=1}^{m}\partial_{z_3}\tilde\phi_j+O\left(\frac{1}{m(\log m)^5}\right).
\end{aligned}
\end{equation}
For $z\in B_{c_2m^{-1}}\left(\xi_0\right)$ with $c_2$ small, we have
$h+|z_3|\leq Cm^{-1}$ . Furthermore, we can derive the following estimate for $q_{m,\ell},~\ell=1,2,3$ from \eqref{2.phi}-\eqref{2.psi} and \eqref{2.cal-1}-\eqref{2.cal-3}
\begin{equation}
\label{2.cal-nabla}
\sum_{\ell=1}^3|q_{m,\ell}|\leq C\frac{m}{\log m}.
\end{equation}
As a consequence of \eqref{2.qmxi0} and \eqref{2.cal-nabla} we derive that
there exists a small positive generic constant $c_0$ such that
$$q_m(z)>0\quad \mbox{for}~z\in B_{c_0m^{-1}}\left(\xi_0\right).$$
Hence we proved \eqref{2.lem-eq} and it finishes the whole proof.
\end{proof}

In addition to Lemma \ref{2.lemma-1}, we can further show that the nodal
set $\mathcal{N}(q_m)$ is smooth.

\begin{lemma}
\label{2.lemma-2}
The nodal set $\mathcal{N}(q_m)$ of $q_m$ is a smooth Riemann surface whenever $m$ is sufficiently large, i.e., for any point $z$ of $\mathcal{N}(q_m)$, $\nabla q_m(z)\neq 0$.
\end{lemma}

\begin{proof}
We first show that $\nabla q_m(z)\neq 0$ for any zero point $z=(r\cos\theta,r\sin\theta,z_3)$ with $|z|\geq1$. Based on Lemma \ref{2.lemma-1} we may assume that
$$|z-\xi_1|=\inf_{j=1,\cdots,m}|z-\xi_j|\geq c_1m^{-1},~\quad \theta \in\left[0,\frac{\pi}{m}\right],~\quad\mbox{and}\quad h\leq C_0m^{-1}.$$
%We shall divide our discussion into three steps according the position of the zero point $z$. First of all we prove that $\nabla q_m(z)\neq 0$ for $z_3\neq 0$, therefore we can focus our discussion for the point $z$ lying on $z_1z_2$ plane. In the second step, we shall prove the derivative of $q_m$ in the radial direction is strictly positive for $|z|\geq1$ and $h$ is not so small. While if $h$ is small, using Lemma \ref{2.lemma-1}, we have established that $q_m$ is positive within a small ball centered at $\xi_0$, whereas the series from the negative bubble exhibits rapid oscillations along the angular direction. We will prove that $\partial_\theta q_m(z) \neq 0$ for $z$ satisfying \eqref{2.cal-6}.
We will divide our argument into three steps based on the location of the zero point $z$. Initially, we will prove that $\nabla q_m(z) \neq 0$ when $z_3 \neq 0$, enabling us to focus our analysis on points $z$ within the $z_1z_2$ plane. In the second step, we will establish that the radial derivative of $q_m$ is strictly positive when $|z| \geq 1$ and $h$ is not too small. However, when $h$ is small, we observe that the series from the negative bubble exhibits rapid oscillations in the angular direction, provided $z$ is not close to $\xi_0$. Combined with the findings of Lemma \ref{2.lemma-1}, this allows us to prove that $\partial_\theta q_m(z)$ is nonzero for zero points in the $z_1z_2$ plane where $h$ is small.
\medskip

\noindent {\bf Step 1.} For any zero point $z$ of $q_m(z)$ with $z_3\neq 0$, $\partial_{z_3}q_m(z)\neq 0$. Using \eqref{2.cal-3} and the fact $\psi$ and $\tilde\phi_j$ is evenly symmetry with respect to $z_3$ we see that,
$$|\nabla \psi(z)|\leq C|z_3|\quad \mbox{and}\quad
\sum_{j=1}^m|\nabla\tilde\phi_j(z)|\leq C\frac{m}{(\log m)^2}|z_3|\quad \mbox{for}\quad |z-\xi_1|\geq \frac{c_0}{m}.$$
Hence
\begin{equation*}
\left|\frac{\partial q(z)}{\partial z_3}\right|\geq \left|\sum_{j=1}^m\frac{3^\frac14\mu_m^{\frac12}}{2\left(h^2+\sin^2\left(\frac{(j-1)\pi}{m}-\frac\theta2\right)\right)^\frac32}
	-C\frac{m}{(\log m)^2}\right||z_3|>0,
\end{equation*}
where we used that $\theta\in\left[0,\frac{\pi}{m}\right]$.
\medskip

\noindent {\bf Step 2.} For any zero point $z$ lying on the $z_1z_2$ plane, we show that $\nabla q_m(z)\neq 0$ with $h\geq Cm^{-\frac53}$. To be precise, if $z_3=0$ we have
	\begin{equation*}
			h^2=\frac{(r-\sqrt{1-\mu_m^2})^2}{4r\sqrt{1-\mu_m^2}}\quad\mbox{and}\quad r\geq1>\sqrt{1-\mu_m^2}.
		\end{equation*}
		Multiplying the above $\nabla q_m(z)$ by $z$ we have
		\begin{equation}
			\label{2.direct}
			\begin{aligned}
				\nabla q_m(z)\cdot z=-3^\frac14\frac{|z|^2}{(1+|z|^2)^\frac32}+3^\frac14\mu_m^{\frac12}\sum_{j=1}^m\frac{(z-\xi_j)\cdot z}{|z-\xi_j|^3}+O\left(1\right).
			\end{aligned}
		\end{equation}
		Let us compute $(z-\xi_j)\cdot z$,
		\begin{equation*}
			\begin{aligned}
				(z-\xi_j)\cdot z
				%=~&z_3^2+r^2-\left(\cos\left(\frac{2(j-1)\pi}{k}\right)+\cos\theta\right)z_1\\&-\left(\sin\left(\frac{2(j-1)\pi}{k}\right)+\sin\theta\right)z_2+(1-\mu^2)\cos\left(\frac{2(j-1)\pi}{k}-\theta\right)\\
				=r^2-r\sqrt{1-\mu_m^2}\cos\left(\frac{2(j-1)\pi}{m}-\theta\right)
				>0,
			\end{aligned}
		\end{equation*}
		where we used $r>\sqrt{1-\mu_m^2}$ due to that $r=|z|\geq1$. Particularly for $j=1$ we have
		\begin{equation}
			\label{2.cal-4}
			\begin{aligned}
				(z-\xi_1)\cdot z=r^2-r\sqrt{1-\mu_m^2}+2r\sqrt{1-\mu_m^2}\sin^2\frac{\theta}{2}\geq C(h+\theta^2).
			\end{aligned}
		\end{equation}
		Substituting \eqref{2.cal-4} into \eqref{2.direct}, we have
		\begin{equation}
			\label{2.cal-5}
			\begin{aligned}
				\nabla q_m(z)\cdot z&\geq -3^\frac14\frac{|z|^2}{(1+|z|^2)^\frac32}
				+\frac{C}{m\log m}\frac{h+\theta^2}{(h^2+\sin^2\theta)^\frac32}
				+O\left(1\right).
			\end{aligned}
		\end{equation}
		So if $h\geq  Cm^{-\frac{5}{3}}$,
		$$\frac{C}{m\log m}\frac{h+\theta^2}{(h^2+\sin^2\theta)^\frac32}\geq
		\min\left\{\frac{C}{h^2m\log m},C\frac{hm^2}{\log m}\right\}\gg1.$$
		Hence $\nabla q_m\neq 0$.
\medskip

\noindent {\bf Step 3.} It remains to study the zero point $z=(r\cos\theta,r\sin\theta,z_3)$ of $q_m(z)$ with\begin{equation}
			\label{2.cal-6}
			|z|=r\geq1,\quad  |z-\xi_1|\geq\frac{c_1}{m},\quad \theta\in\left(0,\frac{\pi}{m}\right],\quad \mbox{and}\quad 0\leq h\leq Cm^{-\frac53}.
		\end{equation}
By \eqref{2.lem-eq} we see that distances between all the zero points of $q_m(z)$ and $\xi_0$ is no less than $c_0/m$, i.e.,
$$\left|z-\xi_0\right|\geq\frac{c_0}{m}.$$
		For the zero point $z$ in \eqref{2.cal-6},  we have
		\begin{equation*}
			\begin{aligned}
				\frac{c_0}{m}\leq~& \left|z-\left(\sqrt{1-\mu_m^2}\cos\frac{\pi}{m},\sqrt{1-\mu_m^2}\sin\frac{\pi}{m},0\right)\right|\\
				\leq~& |z-\sqrt{1-\mu_m^2}\left(\cos\theta,\sin\theta,0\right)|
				+\sqrt{1-\mu_m^2}\left|\left(\cos\theta-\cos\frac{\pi}{m},\sin\theta-\sin\frac{\pi}{m}\right)\right|\\
				\leq~& Ch+C\left|\theta-\frac{\pi}{m}\right|.
			\end{aligned}
		\end{equation*}
		Therefore, the angle of the zero point must have a distance of the order $\frac{1}{m}$ to $\frac{\pi}{m}$. It implies that there exists a small positive constant $c_3$ such that $\left|\theta-\frac{\pi}{m}\right|\geq \frac{c_3\pi}{m}.$ Together with \eqref{2.cal-6} we can further restrict our discussion to
		\begin{equation}
			\label{2.cal-7}
			|z|=r\geq1,\quad  |z-\xi_1|\geq\frac{c_1}{m},\quad \theta\in\left(0,\frac{\pi(1-c_3)}{m}\right],\quad \mbox{and}\quad 0\leq h\leq Cm^{-\frac53}.
		\end{equation}
		Now we shall prove that for all the zero points $z$ satisfying \eqref{2.cal-7}, $\nabla q_m(z)\neq 0$. In fact, we compute $\partial_{\theta}q_m(z)=-\sin\theta\partial_{z_1}q_m(z)+\cos\theta \partial_{z_2}q_m(z)$. It is known that $$\partial_{\theta}\psi\left(\cos\frac{\pi}{m},\sin\frac{\pi}{m},0\right)=0$$ due to $\psi(z)$ being symmetric with respect to $\theta=\frac{\pi}{m}.$ Thus $\partial_\theta\psi(z)=O(\frac{1}{m})$ for all $z$ in \eqref{2.cal-7}. Using \eqref{2.cal-1} and \eqref{2.cal-2}, we have
		\begin{equation}
			\label{2.cal-8}
			\begin{aligned}
				\partial_\theta q_m(z)=&-
				3^\frac14\mu_m^{\frac12}\sum_{j=1}^m
				\frac{2\sqrt{1-\mu_m^2}\sin\left(\frac{(j-1)\pi}{m}-\frac{\theta}{2}\right)\cos\left(\frac{(j-1)\pi}{m}-\frac{\theta}{2}\right)}{(4r\sqrt{1-\mu_m^2})^\frac32\left(h^2+\sin^2\left(\frac{(j-1)\pi}{m}-\frac{\theta}{2}\right)\right)^\frac32}\\
                &+O\left(\frac{1}{(\log m)^2}\right).
\end{aligned}
\end{equation}
Using the fact that $h\leq Cm^{-\frac53}$ we have
\begin{equation}
\label{2.cal-9}
\begin{aligned}
\sum_{j=2}^m\frac{\sin\left(\frac{(j-1)\pi}{m}-\frac{\theta}{2}\right)\cos\left(\frac{(j-1)\pi}{m}-\frac{\theta}{2}\right)}{\left(h^2+\sin^2\left(\frac{(j-1)\pi}{m}-\frac{\theta}{2}\right)\right)^\frac32}=
\sum_{j=2}^m\frac{\cos\left(\frac{(j-1)\pi}{m}-\frac{\theta}{2}\right)}{\sin^2\left(\frac{(j-1)\pi}{m}-\frac{\theta}{2}\right)}+O\left({m^\frac23}\right).
\end{aligned}
\end{equation}
On the other hand, it is not difficult to check that
\begin{equation}
\label{2.cal-10}
-\frac{\cos\left(\frac{(m-j+1)\pi}{m}-\frac{\theta}{2}\right)}{\sin^2\left(\frac{(m-j+1)\pi}{m}-\frac{\theta}{2}\right)}
=\frac{\cos\left(\frac{(j-1)\pi}{m}+\frac{\theta}{2}\right)}{\sin^2\left(\frac{(j-1)\pi}{m}+\frac{\theta}{2}\right)}\geq \frac{\cos\left(\frac{j\pi}{m}-\frac{\theta}{2}\right)}{\sin^2\left(\frac{j\pi}{m}-\frac{\theta}{2}\right)},\quad j=2,\cdots,\frac{m}{2}.
\end{equation}
Substituting \eqref{2.cal-9} and \eqref{2.cal-10} into \eqref{2.cal-8}, we have
\begin{equation}
\label{2.cal-11}
\begin{aligned}
\partial_\theta q_m(z)
\geq~& {3^\frac14\mu_m^{\frac12}}\frac{2\sqrt{1-\mu_m^2}}{(4r\sqrt{1-\mu_m^2})^\frac32}\left(\frac{\cos\frac{\theta}{2}\sin\frac{\theta}{2}}{(h^2+\sin^2\frac{\theta}{2})^\frac32}-\frac{\cos\left(\frac{\pi}{m}-\frac{\theta}{2}\right)}{\sin^2\left(\frac{\pi}{m}-\frac{\theta}{2}\right)}\right)\\
&+O\left(\frac{1}{(\log m)^2}\right),
\end{aligned}
\end{equation}
If $\theta\in\left[\frac{\pi}{3m},\frac{\pi(1-c_3)}{m}\right]$, we have
\begin{equation}
\label{2.cal-12}
\begin{aligned}
\partial_\theta q_m(z)
\geq ~&{3^\frac14\mu_m^{\frac12}}\frac{2\sqrt{1-\mu_m^2}}{(4r\sqrt{1-\mu_m^2})^\frac32}\left(\frac{\cos\frac{\theta}{2}}{\sin^2\frac{\theta}{2}}-\frac{\cos\left(\frac{\pi}{m}-\frac{\theta}{2}\right)}{\sin^2\left(\frac{\pi}{m}-\frac{\theta}{2}\right)}\right)+O\left(\frac{1}{(\log m)^2}\right)\\
\geq~&{3^\frac14\mu_m^{\frac12}}\frac{2\sqrt{1-\mu_m^2}}{(4r\sqrt{1-\mu_m^2})^\frac32}\left(\frac{4}{\theta^2}-\frac{4}{\left(\frac{2\pi}{m}-\theta\right)^2}\right)+O\left(\frac{1}{(\log m)^2}\right)\\
\geq~&{3^\frac14\mu_m^{\frac12}}\frac{2\sqrt{1-\mu_m^2}}{(4r\sqrt{1-\mu_m^2})^\frac32}\left(\frac{16c_3m^2}{\pi^2(1-c_3^2)^2}\right)+O\left(\frac{1}{(\log m)^2}\right).
\end{aligned}
\end{equation}
While if $\theta\in(0,\frac{\pi}{3m})$, we have
\begin{equation}
\label{2.cal-13}
\begin{aligned}
\partial_\theta q_m(z)
\geq ~& {3^\frac14\mu_m^{\frac12}}\frac{2\sqrt{1-\mu_m^2}}{(4r\sqrt{1-\mu_m^2})^\frac32}\left(\min\left\{\frac{\theta}{4\sqrt{2}(h^2)^\frac32},
\frac{\sqrt{2}}{\theta^2}\right\}
-\frac{4}{\left(\frac{2\pi}{m}-\theta\right)^2}\right)\\
&+O\left(\frac{1}{(\log m)^2}\right)\\
\geq~&{3^\frac14\mu_m^{\frac12}}\frac{2\sqrt{1-\mu_m^2}}{(4r\sqrt{1-\mu_m^2})^\frac32}\left(\frac{9\sqrt{2}m^2}{\pi^2}-\frac{36m^2}{25\pi^2}\right)+O\left(\frac{1}{(\log m)^2}\right).
\end{aligned}
\end{equation}
As a consequence of \eqref{2.cal-11}-\eqref{2.cal-13}, for $z$ in \eqref{2.cal-7} we have
$$\partial_\theta q_m(z)=\frac{Cm}{ \log m}+O\left(\frac{1}{(\log m)^2}\right)<0.$$
Together with the conclusions in Step 1 and Step 2, we conclude that for all the zero point $|z|\geq 1$ of $q_m(z)$,  $|\nabla q_m(z)|\neq 0.$
\medskip

For the zero point $z$ of $q(z)$ with $|z|\leq 1$, due to the fact that the solution is invariant after the Kelvin transform, we see that $q_m(z^*)=0$ for $z^*=\frac{z}{|z|^2}$. We have already shown that $\nabla q_m(z^*)\neq 0$, using the fact $q_m(z)=\frac{1}{|z|}q_m\left(\frac{z}{|z|^2}\right)$ and $q_m(z^*)=q_m(z)=0$ for $z$ is zero point, we have
		\begin{equation*}
			\partial_{z_j^*}q_m(z^*)=\frac{1}{|z^*|}\partial_jq_m(z)\frac{1}{|z^*|^2}-2\frac{1}{|z^*|}\sum_{\ell=1}^3\partial_\ell q_m(z)\frac{z_j^*z_\ell^*}{|z^*|^4},\quad j=1,2,3.
		\end{equation*}
		As a consequence, if $\nabla q_m(z^*)\neq 0$ for $z^*=\frac{z}{|z|^2}$ is a zero point of $q_m(z)$, we must have $\nabla_m q(z)\neq 0$. Hence we finish the whole proof.
\end{proof}

\begin{proof}[Proof of Theorem \ref{th2.qm-nodal}.]
Theorem \ref{th2.qm-nodal} follows from Lemma \ref{2.lemma-1} and Lemma \ref{2.lemma-2}.
\end{proof}

%Summarizing the discussion in this subsection, we have the following theorem of describing the nodal set of $q_m(z)$

\subsection{Non-degeneracy}\label{sec:non-deg}
In this subsection, we shall interpret the non-degeneracy in the sense of Duyckaerts-Kenig-Merle for our demands. From now on, we choose $m$ large enough and fix a function $q=q_m$. For simplicity of notation, we drop the index $m$ when there is no ambiguity.
\begin{align}\label{Yamabe-eqn}
    \Delta q+q^5=0 \quad \text{in}\quad \R^3.
\end{align}
As we have pointed out, the following transformations keep the solution set of the Yamabe equation \eqref{Yamabe-eqn} invariant.
%Since $q$ is invariant under the Kelvin transformation, we have $q(x)=q(0)|x|^{-1}+O(|x|^{-2})$ as $x\to \infty$.
\begin{enumerate}
    \item the translation $T_y:z\to z+y$ where $y\in \mathbb{R}^3$,
    \item the dilation $D_\e:z\to \e z$ for $\e>0$,
    \item the rotation  $ R_\theta :z\to    R_\theta z$ for $R_\theta\in SO(3)$,
    \item the inversion $J:z\to \frac{z}{|z|^2}$,
    \item the translation under inversion $\psi_\xi=J\circ T_{\xi}\circ J$.
\end{enumerate}
For any set of parameters $A=(y,R_\theta,\e,\xi )\in \R^3\times SO(3)\times \R_+\times \R^3$, we define the transformation $\mathscr{T}_A=   T_{-y} \circ R_\theta \circ D_\e\circ \psi_\xi$. Then $ \Theta_A (x)=|\det(\mathscr{T}'_A(x))|^\frac{1}{6}q(\mathscr{T}_A(x))$ is also a solution to the Yamabe equation \eqref{Yamabe-eqn}. Using $Jq=q$, we have
\begin{align*}
     \Theta_A (z)= \frac{\varepsilon^{\frac12}}{|z-y|}q\left(\frac{\e R_\theta(z-y)}{|z-y|^2}+\xi\right).
\end{align*}
Choosing $A$ near to $(0,Id,1,0)$, it generates a family of solutions near $q$. Taking the derivatives on each parameter in $A$, we obtain
10 functions
\begin{align}
\begin{split}\label{10func}
   &  - z_j q+|z|^2 \partial_{z_j} q-2 z_j z \cdot \nabla q, \quad \partial_{z_j} q, \quad 1 \leq j \leq 3, \\
& \left(z_j \partial_{z_\ell}-z_\ell \partial_{z_j}\right) q, \quad 1 \leq j<\ell \leq 3, \quad \frac{1}{2} q+z \cdot \nabla q.
\end{split}
\end{align}
One can verify that $L_qf =0$ for any $f$ in the above,
where $L_q$ is linearized operator near $q$
\begin{align*}
    L_q=-\Delta -5q^4.
\end{align*}
Note that these 10 functions are generated consecutively by $\psi_\xi$, translation, rotation, and dilation. The linear combination of these 10 functions form a subspace $\widetilde{\mathcal{Z}}_q$ of $D^{1,2}(\R^3)$ such that $L_q f=0$ for any $f\in \widetilde{\mathcal{Z}}_q$. Some of these 10 functions may be equal, therefore ${\rm dim}~\widetilde{\mathcal{Z}}_q\leq 10$.

Define the kernel space of $L_q$ by $\mathcal{Z}_q=\{f\in D^{1,2}(\R^3):L_qf=0\}$. Musso-Wei \cite{musso2015nondegeneracy} proved that
$\widetilde{\mathcal{Z}}_q=\mathcal{Z}_{q}$. The results there indicate that $-z_3 q+|z|^2\partial_{z_3}q-2z_3z\cdot\nabla q=\partial_{z_3}q$ in \eqref{10func}. We re-state their result as the following proposition.
\begin{proposition}[Non-degeneracy]\label{prop:non-degeneracy}
    When $m$ is large enough, $q_m$ is non-degenerate and
    $${\rm dim}~ \mathcal{Z}_{q_m}=9,$$
    where $q_m$ is the one constructed in \cite{del2011large} in $\R^3$.
\end{proposition}

In the following, we will construct a solution to the Brezis-Nirenberg problem. Note that $q$ is even for $z_3$. We will use  Lyapunov-Schmidt reduction (or gluing method) to seek a solution that is also even for $z_3$. First, we need a family of bubbles that are even for $z_3$ and depend on finite-dimensional parameters. The construction will be similar to $\Theta_A$ in the above, but in this case we need to restrict the rotation and translation to the $z_1z_2$-plane and choose the translation in a subtle way.

We define
\begin{align}\label{def:Gamma}
    \Gamma= \mathcal{N}(q)\cap\{x\in \R^3:(z_1,z_2,0)\}
\end{align}
which is the intersection of the nodal set of $q$ with the $z_1z_2$-plane. For any set of parameters $A=(\e,\xi,a,b,\beta )\in \R\times \Gamma\times \R\times\R^2 \times \mathbb{S}$, we define
\begin{align}\label{def:Q_A}
    Q_A(z)=\frac{\varepsilon^{\frac12}}{|z-b|}q\left(\frac{\e R_\beta(z-b)}{|z-b|^2}+\xi+a\nu(\xi)\right),
\end{align}
where $\nu(\xi)=\frac{\nabla q(\xi)}{|\nabla q(\xi)|}$ is the unit normal to $\Sigma$.
Here $R_\beta$ is the rotation matrix in $z_1z_2$-plane by angle $\beta$
\begin{align}\label{def:rotation}
R_\beta=\begin{pmatrix}
\cos \beta & -\sin\beta & 0 \\
\sin \beta & \cos\beta & 0 \\
0 & 0 & 1
\end{pmatrix}.
\end{align}
In the notation of  \eqref{def:Q_A}, we also will treat a vector in the $z_1z_2$-plane as a vector in three-dimensional space with the last coordinate being 0. For example, $b=(b_1,b_2)$ also denotes $b=(b_1,b_2,0)$.

The family of $Q_A$ plays a vital role in this paper. Note that in the particular case $a=0$, $\e=1$ and $b=0$, $Q_A(z)=O(|z|^{-2})$ as $z\to \infty$ uniformly for $\xi\in \Gamma$ and $\beta\in \mathbb{S}$. This is a fast decay bubble. When we perturb $a$ a little bit, $Q_A(z)$ has leading order $a|z|^{-1}$ when $z\to \infty$. However, if $a$ is small enough, it is still fast decaying inside $B_1$.

%It is known that $Q_A$ satisfies the following equation
% \begin{equation*}
% \Delta Q_A+Q_A^5=0.
% \end{equation*}
% The equation is invariant under the parameter $A$. Hence, the derivative of $Q_A$ with respect to these parameters provide all the kernels of the corresponding linearized equation:

In the following, we will need the kernel of the $L_{Q_A}$ for $A=(1,\xi,0,0,\theta_*)$ where $\theta_*=\theta_*(\xi)$ is chosen such that $R_{\theta_*}^T\nabla q(\xi)=|\nabla q(\xi)|(1,0,0)^T$. The purpose of $\theta_{*}(\xi)$ is to fix the direction of $Q_A$ as $\xi$ runs in the nodal set of $q$. Since $\xi$ is the coordinates of the zero set of $q$, then $(q_{\nu(\xi)},q_\xi,q_{z_3})^T=|\nabla q(\xi)|(1,0,0)^T=R_{\theta_*}^T\nabla q(\xi)$. Thus
\begin{align*}
q_{\nu(\xi)}=\nabla q(\xi)\cdot(\cos\theta_*,\sin\theta_*,0)^T,~
q_\xi=\nabla q(\xi)\cdot(-\sin\theta_*,\cos\theta_*,0)^T,~
q_{z_3}=\nabla q(\xi)\cdot(0,0,1)^T.
\end{align*}
The associate kernels of the linearized kernels are given by
\begin{equation*}
Z_0(z)=\frac{\partial}{\partial \e}Q_{A}|_{\e=1,b=0,a=0,\beta=\theta_*}=\frac{1}{2|z|}q\left(R_{\theta_*}\frac{z}{|z|^2}+\xi\right)-\frac{1}{|z|^3}R_{\theta_*}^T\nabla q\left(R_{\theta_*}\frac{z}{|z|^2}+\xi\right)\cdot\left(\begin{matrix}
z_1\\ z_2\\0
\end{matrix}\right),
\end{equation*}
%where $\theta_*=\theta_*(\xi)$ is chosen such that $R_{\theta_*}^T\nabla q(\xi)=(1,0,0)^T$.
\begin{equation*}
Z_1(z)=\frac{\partial}{\partial a}Q_{A}|_{\e=1,b=0,a=0,\beta=\theta_*}=\frac{1}{|z|}q_{\nu(\xi)}\left(R_{\theta_*}\frac{z}{|z|^2}+\xi\right),
\end{equation*}
\begin{equation*}
Z_2(z)=\frac{\partial}{\partial \xi}Q_{A}|_{\e=1,b=0,a=0,\beta=\theta_*}=\frac{1}{|z|}q_{\xi}\left(R_{\theta_*}\frac{z}{|z|^2}+\xi\right),
\end{equation*}
\begin{equation*}
Z_3(z)=\frac{\partial}{\partial b_1}Q_{A}|_{\e=1,b=0,a=0,\beta=\theta_*}
=\frac{2z_1}{|z|^2}\frac{\partial Q_A}{\partial\e}-\frac{1}{|z|^2}\frac{\partial Q_A}{\partial a}=\frac{2z_1}{|z|^2}Z_0-\frac{1}{|z|^2}Z_1,
\end{equation*}
\begin{equation*}
Z_4(z)=\frac{\partial}{\partial b_2}Q_{A}|_{\e=1,b=0,a=0,\beta=\theta_*}
=\frac{2z_2}{|z|^2}\frac{\partial Q_A}{\partial\e}-\frac{1}{|z|^2}\frac{\partial Q_A}{\partial \xi}=\frac{2z_2}{|z|^2}Z_0-\frac{1}{|z|^2}Z_2,
\end{equation*}
\begin{equation*}
Z_5(x)=\frac{\partial}{\partial\beta}Q_A|_{\e=1,b=0,a=0,\beta=\theta_*}
=\frac{z_1}{|z|^2}\frac{\partial Q_A}{\partial \xi}-\frac{z_2}{|z|^2}\frac{\partial Q_A}{\partial a}=
\frac{z_1}{|z|^2}Z_2-\frac{z_2}{|z|^2}Z_1.
\end{equation*}
There are three functions we did not list here, which come from the differentiation of translation in $z_3$, rotation in $z_1z_3$-plane, and rotation in $z_2z_3$-plane. They are all odd in $z_3$ and play no role in our proof.

\section{Preliminary on Approximate Solutions}
In this section, we shall modify the family of bubbles $Q_A$, defined at \eqref{def:Q_A}, to satisfy a similar equation to the Brezis-Nirenberg problem in $\Sigma_K$ and Dirichlet boundary conditions. This is the first step in the gluing process. More precisely, we define the approximate solution $PQ_A$ (or the projection of $Q_A$) to be
\begin{equation}
\begin{cases}\label{301}
\Delta PQ_A +\lambda PQ_A + Q_A^5=0 \ &\mbox{in} \ \Sigma_K,\\ PQ_A=0 \ &\mbox{on} \ \partial \Sigma_K.
\end{cases}
\end{equation}
At the end of this section, we will prove $PQ_A$ is $Q_A$ summing other terms with good control.

%define the approximate solution based on the crown solution given by Del Pino-Musso-Pacard-Pistoia \cite{del2011large}. For the rest of this paper, denote

First, let us set up some notations which will be used for the rest of this paper.
% $$ \theta_0= \frac{\pi}{K},$$
% where $K$ is a large even number, which will be determined later:
% See Figure \ref{fig:SigmaK} for illustration.
%We will assemble  Del Pino-Musso-Pacard-Pistoia crown solution so that it satisfies the Dirichlet boundary condition on $\Sigma_K$.
Let $z=(z_1,z_2,z_3)^T=(z_1+iz_2,z_3)^T$, where $i=\sqrt{-1}$ is the imaginary number. We abuse the notation $ze^{i\alpha}=((z_1+iz_2)e^{i\alpha},z_3)^T$ and $\bar z=(z_1-iz_2,z_3)^T=(z_1,-z_2,z_3)^T$. Using this notation convention and definition of $R_\beta$ in \eqref{def:rotation}, we have $R_\beta z=ze^{i\beta}$. All vectors in $\R^3$ are column vectors without explicit mention.

We define
\begin{align}\label{def:SK}
    \Sigma_K= \{ (r, \theta,z_3)\in B_1; 0<r<1, -\theta_0 < \theta <\theta_0 \}\quad \text{where}\quad \theta_0= \frac{\pi}{K}.
\end{align}
For any function $u$ in $B_1$, we define an operation
\begin{equation}
\label{ext}
u^e (z)= u(z)- u(\bar{z} e^{2i \theta_0}) + u(z e^{4i\theta_0})-...+ u(z e^{ 2(K-2) \theta_0})- u(\bar{z} e^{2(K-1) i \theta_0}).
\end{equation}
In the sum notation, it reads that
\begin{align*}
    u^e(z)=\sum_{j=0}^{K/2-1}[u(ze^{4ji\theta_0})-u(\bar z e^{(4j+2)i\theta_0})].
\end{align*}
This operation creates a function which is odd under the reflection along each ray $re^{ji\theta_0}$, $j=1,3,5,\cdots,K-1$.
That is,
\[u^e(r e^{ji\theta_0-i\theta},z_3)=-u^e(r e^{ji\theta_0+i\theta},z_3),\quad j=1,3,5,\cdots,K-1,\]
%that is, $u$ is an odd function under reflection along each ray $re^{ji\theta_0}$,  $j=1,3,5,\cdots,K-1$.
In particular
$$ u^e(r e^{ji\theta_0},z_3)= 0\quad j=1,3,5,\cdots,K-1.$$

% Identify $b=(b,0,0)$.

% \begin{align*}
%     Q_A(z)=\frac{\varepsilon^{\frac12}}{|z-b|}q\left(\frac{\e R_{\theta(\xi)}(z-b)}{|z-b|^2}+\xi+a\nu(\xi)\right)+\alpha Z_*+\beta Z_{**}
% \end{align*}
% where
% \begin{align*}
%     Z_*=\frac{d}{d\alpha}|_{\alpha=0}\frac{\varepsilon^{\frac12}}{|z-be^{i\alpha}|}q\left(\frac{\e R_{\theta(\xi)}(z-be^{i\alpha})}{|z-be^{i\alpha}|^2}+\xi+a\nu(\xi)\right)
% \end{align*}
% \begin{align*}
%     Z_{**}=\frac{d}{d\beta}|_{\beta=0}\frac{\varepsilon^{\frac12}}{|z-b|}q\left(\frac{\e R_{\theta(\xi)+\beta}(z-b)}{|z-b|^2}+\xi+a\nu(\xi)\right)
% \end{align*}

We shall use the following family of bubbles described in \eqref{def:Q_A} in the subsection \ref{sec:non-deg},
\begin{align*}
    Q_A(z)=\frac{\varepsilon^{\frac12}}{|z-b|}q\left(\frac{\e R_{\beta}(z-b)}{|z-b|^2}+\xi+a\nu(\xi)\right),
\end{align*}
where $A=(\e,\xi,a,b,\beta)$ and $ \beta=\theta_*(\xi)+\hat \beta$ where $R_{\theta_*(\xi)}^T\nabla q(\xi)=|\nabla q(\xi)|(1,0,0)^{T}$. The purpose of $\theta_{*}(\xi)$ is to fix the ``direction" of $Q_A$ as $\xi$ runs in the nodal set of $q$.
% Since $\xi$ is the coordinates of the nodal set of $q$, then $(q_\xi,q_{\nu(\xi)},q_{z_3})=|\nabla q(\xi)|(0,1,0)$

We will make some constraints on the parameters of $A$.
 The rationale for selecting these constraints is to make sure some functional (see $\Psi(A)$ in \eqref{PsiA-re}) has an infimum achieved inside it. It will become clear throughout the computations in sections 3 and 4, culminating in the proof of Theorem \ref{th8.para}.
Denote $b=|b|e^{i\alpha_b}$ and $d=(1-|b|^2)/(2|b|)$. We make the following constraint.
\begin{align}
\begin{split}\label{constraint-m}
     &\e K^{3}\in [\delta,\delta^{-1}],\quad \xi\in \Gamma, \quad |a|\leq \delta^{-1}\e \log K,\quad d\in \left[ \tfrac{\log K-\log\log K}{K},\tfrac{\log K}{K}\right],\\
    & |\alpha_b|\leq \delta^{-1/2} K^{-2}\log K,\quad  |\hat \beta|\leq  \delta^{-1/2}K^{-1}\log K.
\end{split}
\end{align}
where $\delta\in(0,1)$ is some fixed number that will be determined later. We always choose $K$ that is sufficiently large and even. All the constants $C$ in this paper are independent of $K$ and $\delta$.
%The rationale for selecting these constraints will become clear throughout the computations in sections 3 and 4, culminating in the proof of Theorem \ref{th8.para}.

Under these constraints, we have
\begin{align*}
    |q(\hat \xi)|\leq C|a|\leq C \delta^{-1}\e\log K,\quad \mbox{dist}(b,\partial \Sigma_K)\geq \frac{1}{2K}.
\end{align*}
Hereafter, we adopt the following notations for the remainder of this paper.
\begin{align}\label{nt-wW}
    \hat \xi=\xi+a\nu(\xi),\quad \w1=R_\beta^T \nabla q(\hat \xi),\quad \W2=R_\beta^T\nabla^2q(\hat \xi)R_\beta.
\end{align}
Note that $|w|=|\nabla q(\hat \xi)|\leq C$ and $|W|=|\nabla^2q(\hat \xi)|\leq C$ for a constant $C$ just depend on $q$. Moreover, since $q$ is even for $z_3$, then $w_3=0$ and $w=(|\nabla q(\hat \xi)|e^{i\alpha_w},0)$. By our definition of $\theta_*(\xi)$, we have $\alpha_w=\hat \beta$.

Define $Q_A^e$ by the operation \eqref{ext},
\begin{align}\label{def:QAe}
    Q_A^e=\sum_{j=0}^{K/2-1}[Q_A(ze^{4ji\theta_0})-Q_A(\bar z e^{(4j+2)i\theta_0})]=:Q_A-T_A,
\end{align}
where $T_A$ is the summation of other bubbles (centered on other copies of $\Sigma_K$ in $B_1$)
\begin{align}\label{def:TA}
        T_A=Q_A(\bar z e^{2i\theta_0})-\sum_{j=1}^{K/2-1}[Q_A(z e^{4ji\theta_0})-Q_A(\bar ze^{(4j+2)i\theta_0})].
\end{align}

We define the following key function, $\gamma(z, p)$, which is used to bound the influence of other bubbles in $\Sigma_K$.
\begin{align}\label{def:gamma1-1}
    \gamma(z,p):=\frac{1}{|\bar z e^{2i\theta_0}-p|}-\sum_{j=1}^{K/2-1}\left(\frac{1}{|z e^{4ji\theta_0}-p|}-\frac{1}{|\bar z e^{(4j+2)i\theta_0}-p|}\right).
\end{align}
We can have a lower bound for each denominator.
\begin{lemma}Suppose $b$ satisfies \eqref{constraint-m}. Then for any $z\in \SK$,
\begin{align}
\begin{split}\label{lowerbd-|z-b|}
    |z e^{4ji\theta_0}-b|&\geq C\sin 2j\theta_0,\quad \text{for}\quad 1\leq j\leq K/2-1, \\
    |\bar z e^{(4j+2)\theta_0}-b|&\geq C \sin (2j+1)\theta_0,\quad \text{for}\quad 0\leq j\leq K/2-1.
    \end{split}
\end{align}
\end{lemma}
\begin{proof}
        Suppose $z=(z_1,z_2,z_3)=(re^{i\alpha_z},z_3)$ and $b=(b_1,b_2,0)=(|b|e^{i\alpha_b},0)$. Then
\begin{align}
\begin{split}\label{z-b-1}
    |ze^{4ji\theta_0}-b|^2&=|z|^2+|b|^2-2r |b|\cos(4j\theta_0+\alpha_z-\alpha_b)\\
    &=(r-|b|)^2+z_3^2+4r|b|\sin^2(2j\theta_0+\tfrac{1}{2}\alpha_z-\tfrac{1}{2}\alpha_b),
\end{split}
\end{align}
and
\begin{align}
    \begin{split}\label{z-b-2}
        |\bar ze^{(4j+2)i\theta_0}-b|^2
    &=(r-|b|)^2+z_3^2+4r|b|\sin^2((2j+1)\theta_0-\tfrac{1}{2}\alpha_z-\tfrac{1}{2}\alpha_b).
    \end{split}
\end{align}

If $r<\frac14$, then $|ze^{4ji\theta_0}-b|\geq\frac14$ for $1\leq j\leq K/2-1$ and $|\bar ze^{(4j+2)i\theta_0}-b|\geq\frac14$ for any $0\leq j\leq K/2-1$. Obviously, \eqref{lowerbd-|z-b|} holds.

If $r\geq \frac14$, then using \eqref{z-b-1} and \eqref{z-b-2} and $|b|\geq \tfrac12$, we obtain that
\begin{align*}
    |ze^{4ji\theta_0}-b|^2&
    \geq \tfrac12 \sin^2(2j\theta_0+\tfrac{1}{2}\alpha_z-\tfrac12\alpha_b),\quad 1\leq j\leq K/2-1,\\
    |\bar ze^{(4j+2)i\theta_0}-b|^2
    &\geq \tfrac12\sin^2((2j+1)\theta_0-\tfrac12\alpha_z-\tfrac12\alpha_b),\quad 0\leq j\leq K/2-1.
\end{align*}
Since $|\alpha_z\pm\alpha_b|\leq 2\theta_0$, then \eqref{lowerbd-|z-b|} also holds in this case.
\end{proof}

\begin{lemma}\label{lem:bd-TA}
    Under the constraint \eqref{constraint-m}, for any $z\in \SK$, one has
\begin{equation}
\label{TA-expn}
\begin{aligned}
T_A(z)=~&\e^{\frac12}q(\hat \xi)\gamma(z,b)+\e^{\frac32}\w1\cdot \nabla_p\gamma(z,b)+\frac16\e^{\frac52}\W2_{ij}\partial_{p_ip_j}^2\gamma(z,b)+O(\e^{\frac72}K^4).
\end{aligned}
\end{equation}
In particular, one has $|T_A|= O(\e^{\frac32}K^2)$. Here $\nabla_p\gamma(z,b)$ is the abbreviation of $\nabla_p\gamma(z,p)|_{p=b}$.
%     \begin{align}\label{TA-bound}
%     |T_A|= O(\e^{\frac32}K^2).
% \end{align}
\end{lemma}

\begin{proof}
Using Taylor expansion of $q$ near $\hat \xi$, we have the expansion of $Q(z)$ for any $z$ satisfying $|z-b|\geq \e$
\begin{align}
\begin{split}\label{QA-expn-m}
    Q_A (z)&=\frac{\varepsilon^{\frac12}}{|z-b|}\left[ q(\hat \xi)+\frac{\e w\cdot(z-b)}{|z-b|^2}+\frac{\e^2(z-b)^TW(z-b)}{2|z-b|^4}+O\left(\frac{\e^{3}}{|z-b|^3}\right)\right].
     \end{split}
\end{align}
It readily verifies that
\begin{align}
\label{nabla(z-p)}
        \nabla_{p_j}\left(\frac{1}{|z-p|}\right)=\frac{(z-p)_j}{|z-p|^3}, \quad
        \nabla_{p_jp_\ell}^2\left(\frac{1}{|z-p|}\right)=\frac{-\delta_{j\ell}}{|z-p|^{3}}+\frac{3(z-p)_j(z-p)_\ell}{|z-p|^{5}}.
\end{align}
% \[\nabla_{p_i}\left(\frac{1}{|z-p|}\right)=\frac{(z-p)_i}{|z-p|^3},\]
% \[\nabla_{p_ip_j}^2\left(\frac{1}{|z-p|}\right)=\frac{-\delta_{ij}}{|z-p|^{3}}+\frac{3(z-p)_i(z-p)_j}{|z-p|^{5}}.\]
Using these identities, one can rewrite \eqref{QA-expn-m} as
\begin{align*}
    Q_A(z)=&\ \frac{\e^{\frac12}q(\hat \xi)}{|z-b|}+\e^{\frac32} w\cdot \nabla_p\left(\frac{1}{|z-p|}\right)\Big|_{p=b}\\
    &+\frac16\e^{\frac52}W_{j\ell}\left[\partial_{p_jp_\ell}^2\left(\frac{1}{|z-p|}\right)+\frac{\delta_{j\ell}}{|z-p|^3}\right]\Big|_{p=b}+O\left(\frac{\e^{\frac{7}{2}}}{|z-b|^4}\right).
\end{align*}
When $b\in \Sigma_K$ satisfies \eqref{constraint-m}, then \eqref{lowerbd-|z-b|} implies $|ze^{4ji\theta_0}-b|\geq C/K\gg \e$ and $|\bar ze^{(4j+2)i\theta_0}-b|\geq C/K\gg \e$.
% \begin{align}
%     |ze^{4ji\theta_0}-b|\geq \frac{1}{2K}
% \end{align}$\e/|ze^{4ji\theta_0}-b|\leq  C\e K\ll 1$ for $j\geq 1$ and $\e/|\bar ze^{(4j+2)i\theta_0}-b| \ll 1$ for $j\geq 0$  and any $z\in \Sigma_K$.
Thus the corresponding expansion of $Q_A(ze^{4ji\theta_0})$ and $Q_A(\bar z e^{(4j+2)i\theta_0})$ from the above hold uniformly for $z\in \SK$. Recall the definition of $T_A$ in \eqref{def:TA}. We have
\begin{equation*}
\begin{aligned}
T_A=~&\e^{\frac12}q(\hat \xi)\gamma(z,b)+\e^{\frac32}\w1\cdot \nabla_p\gamma(z,p)|_{p=b}+\frac16\e^{\frac52}\W2_{ij}\partial_{p_ip_j}^2\gamma(z,p)|_{p=b}\\
&+O(\e^{\frac52}\gamma_3(z,b)|\Delta q(\hat\xi)|+\e^{\frac72}\gamma_4(z,b)),
\end{aligned}
\end{equation*}
where $\gamma_3(z,b)=\sum_{j=0}^{K/2-1}|\bar z e^{(4j+2)i\theta_0}-b|^{-3}+\sum_{j=1}^{K/2-1}|ze^{4ji\theta_0}-b|^{-3}$
% \begin{align}\label{def:gamma3-3}
%     \gamma_3(z,p):=\frac{1}{|\bar z e^{2i\theta_0}-p|^3}+\sum_{j=1}^{K/2-1}\left(\frac{1}{|z e^{4ji\theta_0}-p|^3}+\frac{1}{|\bar z e^{(4j+2)i\theta_0}-p|^3}\right),
% \end{align}
and $\gamma_4(z,p)$ is the corresponding one with the exponent replaced by $-4$. Again, \eqref{lowerbd-|z-b|} implies that
\begin{align*}
    \gamma_3(z,b)\leq C K^3\sum_{j=1}^{K/2-1}j^{-3}\leq CK^3,\quad \gamma_4(z,b)\leq CK^4.
\end{align*}
Thus $\e^{\frac52}\gamma_3(z,b)|\Delta q(\hat\xi)|+\e^{\frac72}\gamma_4(z,b)=O(\e^{\frac52}K^3|a|^5+\e^{\frac72}K^4)=O(\e^{\frac72}K^4)$. Therefore \eqref{TA-expn} is proved.

Finally, it follows from  Lemma \ref{lem:gamma-first} that $|\nabla^m\gamma(z,p)|_{p=b}|\leq CK^{m+1}$ for $m=0,1,2$. Applying these in \eqref{TA-expn} and using the constraint \eqref{constraint-m}, we get
\begin{align*}
   |T_A|\leq C(\e^{\frac12}|q(\hat \xi)|K+\e^{\frac32}K^2+\e^{\frac52}K^3+\e^{\frac72}K^4)\leq C \e^{\frac32}K^2.
\end{align*}
 \end{proof}

% Let $PQ_A$ be the projection
% \begin{equation}
% \begin{cases}\label{301}
% \Delta PQ_A +\lambda PQ_A + Q_A^5=0 \ &\mbox{in} \ \Sigma_K,\\ PQ_A=0 \ &\mbox{on} \ \partial \Sigma_K.
% \end{cases}
% \end{equation}
%Recall $G(z,p)=\frac{1}{4\pi}[|z-p|^{-1}-||z|p-z^*|^{-1}]$ is the Green's function for $B_1$ in $\R^3$.
% \begin{equation}
% \label{eq.green}
% \begin{cases}
% \Delta_z G(z,p)+\delta(z-p)=0,\quad &\mathrm{in}\quad B_1,\\
% G(z,p)=0 ~&\mathrm{on}\quad \partial B_1.
% \end{cases}
%\end{equation}
%where $\delta (z)$ denotes the Dirac mass at the origin.
Suppose that $G(z,p)$ is the Green's function for $B_1$ in $\R^3$.
We denote by $H_0(z,p)$ the regular part of the Green's function, namely
\[H_0(z,p)=4\pi \left(\frac{1}{4\pi |z-p|}-G(z,p)\right).\]
It is well-known that in the unit ball case
\begin{align}\label{H0-expn}
    H_0(z,p)=\left||z|p-\frac{z}{|z|}\right|^{-1}=\frac{1}{\sqrt{1-2z\cdot p+|z|^2|p|^2}}.
\end{align}
We apply operation \eqref{ext} to $H_0$ and denote it as $H_0^e(z,p)$
\begin{align}\label{def:H0e}
    H_0^e(z,p)=\sum_{j=0}^{K/2-1}[H_0(ze^{4ji\theta_0},p)-H_0(\bar ze^{(4j+2)i\theta_0},p)].
\end{align}
Such a function is used to control the influence of $H_0$ from other copies of $\Sigma_K$ in $B_1$.

To get $PQ_A$ in \eqref{301}, we first solve $\varphi_A$ from
\begin{equation}
\begin{cases}
\Delta \varphi_{A} =0  \quad &\mathrm{in}\quad B_1,\\
 \varphi_{A} = Q_A^e  ~&\mathrm{on}\quad \partial B_1,
\end{cases}
\end{equation}
where $Q_A^e$ is defined in \eqref{def:QAe}.

\begin{lemma}\label{lem:phiA-m}
 Under the constraint \eqref{constraint-m}, for any $z\in \SK$, one has
    \begin{align}\label{phiA-expn}
    \varphi_A(z)&=\e^{\frac12}  q(\hat \xi) H_0^e(z,b)+\e^{\frac{3}{2}} \w1\cdot\nabla_p H_0^e(z,b)+\frac16\e^{\frac52}\W2_{ij}\partial_{p_ip_j}^2H_0^e(z,b)+O(\e^{\frac72}K^4),
\end{align}
where $H_0^e(z,p)$ is defined in \eqref{def:H0e}. In particular, one has $|\varphi_A(z)|=O(\e^{\frac32}K^{2})$.
% \begin{align}\label{phiA-bound}
%     |\varphi_A(z)|=O(\e^{\frac32}K^{2}).
% \end{align}
\end{lemma}
\begin{proof}
    When $z\in \partial B_1$, one has $\e/|z-b|\leq C\e K\ll 1$. Using the Taylor expansion of $q$ near $\hat \xi$, we have the expansion of $Q_A(z)$ in \eqref{QA-expn-m}.
Using \eqref{H0-expn}, one knows that, $H_0(z,p)=|z-p|^{-1}$ for $z\in \partial B_1$. Thus \eqref{nabla(z-p)} implies
% \[H_0(z,b)=\frac{1}{|z-b|},\quad\nabla_{p_i}H_0(z,b)=\frac{(z-b)_i}{|z-b|^3},\]
\[\nabla_{p_j}H_0(z,b)=\frac{(z-b)_j}{|z-b|^3},\quad \nabla_{p_jp_\ell}^2H_0(z,b)=\frac{-\delta_{j\ell}}{|z-b|^{3}}+\frac{3(z-b)_j(z-b)_\ell}{|z-b|^{5}}.\]
Comparing the expansion of $Q_A(z)$ in \eqref{QA-expn-m}, we found that $Q_A$ on $\partial B_1$ can be approximated by $H_0(z,b)$ and its derivatives. Similar things also hold for $Q_A(z e^{4ji\theta_0})$ and $Q(\bar z e^{(4j+2)i\theta_0})$.
Thus if we let
\begin{align*}
    \varphi_A=\e^{\frac12}    q(\hat \xi) H_0^e(z,b)+\e^{\frac32} \w1\cdot\nabla_p H_0^e(z,b)+\frac16\e^{\frac52}\W2_{ij}\partial_{p_ip_j}^2H_0^e(z,b)+f_A(z),
\end{align*}
%$Tr(\nabla^2_{pp}H_0^e(z,b) \nabla^2q(\hat \xi))=$
then $f_A$ satisfies $\Delta f_A(z)=0$ in $\Sigma_K$ and
\begin{align*}
    f_A(z)&=\left[Q_A(z)-\frac{\e^{\frac12}q(\hat \xi)}{|z-b|}-\frac{\e^{\frac32}w\cdot(z-b)}{|z-b|^3}-\frac16\frac{\e^{\frac52}(z-b)^TW(z-b)}{|z-b|^5}-\frac16\frac{\e^{\frac52}\Delta q(\hat \xi)}{|z-b|^3}\right]^e
    %&+\frac{1}{3}\e^{\frac52}\Delta q(\hat \xi)\left(\frac{1}{|z-b|^3}\right)^e
\end{align*}
for $z\in \partial \SK$.
Note that on $\partial \SK$,
\begin{align*}
    |f_A(z)|&\leq C\left(\frac{\e^{\frac52}|\Delta q(\hat \xi)|}{|z-b|^3}+\frac{\e^{\frac{7}{2}}}{|z-b|^4}\right)^e\leq C\e^{\frac52}|q(\hat \xi)|^5K^3+C\e^{\frac72}K^4\leq C\e^{\frac72}K^4.
\end{align*}
% Note that $f_A(z)=O(\e^{\frac52}(1-|b|)^{-3})=O(\e^{\frac52}K^{\frac94})$ for $z\in \partial B_1$.
By maximum principle, we have $|f_A(z)|= O(\e^{\frac72}K^4)$ for $z\in B_1$. Finally, one can get $|\varphi_A|=O(\e^{\frac32}K^2)$ by using Lemma \ref{lem:H0e-d}, as did for $T_A$.
\end{proof}

Using the symmetry of $Q_A^e$ and the uniqueness of $\varphi_A$, $\varphi_A$ enjoys the same symmetry as $Q_A^e$. Thus  $\varphi_A(z)=0$ for $z\in \partial\Sigma_K\cap \{|z|<1\}$, i.e., the two portions of the boundary of $\Sigma_K$ inside $B_1$. Since $Q_A^e=0$ on these two portions of the boundary and $\varphi_A=Q_A^e$ on $\partial B_1$, then
\begin{equation}\label{eq:varphi-m}
\begin{cases}
\Delta \varphi_{A} =0  \quad &\mathrm{in}\quad \Sigma_K,\\
 \varphi_{A} = Q_A^e  ~&\mathrm{on}\quad \partial \Sigma_K.
\end{cases}
\end{equation}

Let $PQ_A=Q_A^e-\varphi_A-\psi_A$. Using the above equation and the one of $PQ_A$ in \eqref{301},
% \[PQ_A=Q_A^e-\varphi_A-\psi_A\]
$\psi_A$ must satisfy
\begin{align}\label{psi-F}
    \begin{cases}
        \Delta \psi_A+\lambda \psi_A=\la (Q_A^e-\varphi_A)+F&\text{in }\Sigma_K,\\
        \psi_A=0 &\text{on }\partial \Sigma_K,
    \end{cases}
\end{align}
where
\begin{align}\label{def:F}
    F(z)=Q_A(\bar z e^{2i\theta_0})^5-\sum_{j=1}^{K/2-1}\left[Q_A(ze^{4ji\theta_0})^5-Q_A(\bar ze^{(4j+2)i\theta_0})^5\right].
\end{align}
We want to derive the estimates of $\psi_A$. First, we consider $F$. Note that \eqref{QA-expn-m} leads to
    \begin{align}\label{QA-exterior}
        |Q_A(z)|\leq \frac{\e^{\frac12}|a|}{|z-b|}+ C\frac{\e^{\frac{3}{2}}}{|z-b|^2},\quad \text{for any }|z-b|\geq \e.
    \end{align}
    For any $z\in \Sigma_K$, this applies to $Q(\bar z e^{(4j+2)i\theta_0})$ and $Q(z e^{4ji\theta_0})$ since \eqref{lowerbd-|z-b|}.
    \begin{align}\label{F-bd}
        |F(z)|\leq C\e^{\frac52}|a|^5\gamma_5(z,b)+\e^{\frac{15}{2}}\gamma_{10}(z,b)\leq C\e^{\frac32}K^2.
    \end{align}

% Using $|T_A|+|\varphi_A|=O(\e^{\frac32}K^2)$ (see Lemma \ref{lem:bd-TA} and  \ref{lem:phiA-m}), we have
% \begin{align*}
%     F=\la Q_A+O(\e^{\frac32}K^2).
% \end{align*}

\begin{lemma}\label{lem:psiA-m}
 Under the constraint \eqref{constraint-m}, for any $z\in \Sigma_K$, one has $|\psi_A(z)|=O(\e^{\frac32}|\log \e|)$.
% \begin{align*}
% |\psi_A(z)|=O(\e^{\frac32}|\log \e|).
% \end{align*}
\end{lemma}
\begin{proof}
%The expansion of $\varphi_A$ still holds.

%Now consider $\psi_A$.

Let $y=(z-b)/\e$. Consider $\psi_A(z)=\e^{\frac32}\tilde \psi((z-b)/\e)$,
then $\Delta_z \psi_A+\la \psi_A=\e^{-\frac12}[\Delta_y\tilde \psi+\la \e^2\tilde \psi]$. Thus
%$\tilde\psi$ satisfies $\Delta_y \tilde \psi +\lambda \e^2 \tilde \psi=\tilde F(y)$ in $\tilde \Sigma_K$ and $ \tilde \psi_A=0$ on $\partial \tilde \Sigma_K$,
\begin{align*}
    \begin{cases}
        \Delta_y \tilde \psi +\lambda \e^2 \tilde \psi=\la \frac{1}{|y|}q\left(\frac{R_\beta y}{|y|^2}+\hat \xi\right)+O(\e^2K^2)&\text{in }\tilde \Sigma_K,\\
        \tilde \psi_A=0 &\text{on }\partial \tilde \Sigma_K,
    \end{cases}
\end{align*}
where $\tilde{\Sigma}_K=\{y=(z-b)/\e: z\in \SK\}$. Here we used $|T_A(z)|+|\varphi_A(z)|+|F(z)|=O(\e^{\frac32}K^2)$ for $z\in \SK$.
% and
% \begin{align*}
%     \tilde F(y)&=\la \frac{1}{|y|}q\left(\frac{R_\beta y}{|y|^2}+\hat \xi\right)+O(\e^2K^2).
% \end{align*}
Let
$$\tilde\psi_0(y)=\int_{\mathbb{R}^3}\frac{1}{4\pi|y-z|}\frac{1}{|z|}q\left(\frac{R_\beta z}{|z|^2}+\hat \xi\right)dz.$$
% \begin{align*}
%     \Delta_y \tilde \psi_{0}+\frac{1}{|y|}q\left(\frac{R_\beta y}{|y|^2}+\hat \xi\right)=0.
% \end{align*}
Then it is easy to see that $\tilde\psi_0(y)$ satisfies the equation $\Delta_y \tilde \psi_{0}+\frac{1}{|y|}q\left(\frac{R_\beta y}{|y|^2}+\hat \xi\right)=0$. One
can verify that $|\tilde \psi_0|\leq C(1+\log \sqrt{1+y^2})$. Now let $\tilde \psi=\la \tilde \psi_0+\tilde \psi_1$. Then $\tilde \psi_1$ satisfies
\begin{align*}
    \Delta_y \tilde \psi_1+\lambda \e^2\tilde \psi_1=O(\e^2K^2).
\end{align*}
Note that $B_{1/(2K\e)}(0)\subset \tilde \Sigma_K\subset \{|y_2|\leq \frac{2\pi}{K\e}\}$. Thus, on $\partial\tilde\Sigma_K$, one has    $\tilde \psi_1=O(|\log \e|)$. We construct a barrier function
\begin{align*}
    \bar\psi_1(y)=C|\log \e|[(3\pi)^2-(K \e y_2)^2].
\end{align*}
One can verify that
\begin{align*}
    \Delta_y \bar\psi_1+\la \e^2\bar\psi_1\leq C|\log \e|[-2(K\e)^2+\la \e^2 (3\pi)^2]\leq -C(K\e)^2,
\end{align*}
provided $K$ is large enough. Moreover, $\bar\psi_1\geq C|\log \e|$ on $\partial\tilde\Sigma_K$.

Note that $\tilde \Sigma_K\subset \{|y_2|\leq 2\pi/(K\e)\}$ implies that the first eigenvalue $\lambda_1(\tilde \Sigma_K)\geq (K\e/2\pi )^2$. Thus $\Delta_y+\la \e^2$ satisfies the maximum principle on $\tilde \Sigma_K$ when $K$ is large enough.
Therefore $|\tilde \psi_1(y)|\leq C|\bar\psi_1|=O(|\log \e|)$. Note that $|\tilde \psi_0|\leq C(1+\log \sqrt{1+y^2})\leq C|\log \e|$. The proof is complete.
\end{proof}
To summarize Lemma \ref{lem:bd-TA}, Lemma \ref{lem:phiA-m} and Lemma \ref{lem:psiA-m} in this section, we have the following proposition.
\begin{proposition}\label{prop:PQA=}
    Assume $A$ satisfies the constraint \eqref{constraint-m}. The solution of  \eqref{301} can be written $PQ_A=Q_A-\tilde T_A$ where $\tilde T_A= T_A-\varphi_A-\psi_A$ satisfying
%     \begin{align*}
%     PQ_A=Q_A-T_A-\varphi_A-\psi_A.
% \end{align*}
\begin{align}\label{bd-tTA}
|\tilde T_A|\leq C\e^{\frac32}K^2.
\end{align}

\end{proposition}

\begin{remark}\label{rmk:Ta-C1}
    By elliptic theory, it is not hard to show that the dependence of $\tilde T_A$ on the parameters in $A$ is at least $C^1$.
\end{remark}
% where $T_A$ is defined in \eqref{def:TA}. We write $\tilde T_A=T_A+\varphi_A+\psi_A$ for short. It follows from the bound in  that
% \begin{align}\label{bd-tTA}
% |\tilde T_A|\lesssim \e^{\frac32}K^2.
% \end{align}

\section{Energy expansion}
In this section, we shall compute the energy of the approximate solution and find its leading-order term with respect to the parameters.

Define the energy of $PQ_A$ as
\begin{align*}
    J(PQ_A)
   &=\frac{1}{2}\int_{\Sigma_K}(|\nabla PQ_A|^2-\la |PQ_A|^2)dz-\frac{1}{6}\int_{\Sigma_K}(PQ_A)^6dz\\
   &=\frac{1}{2}\int_{\Sigma_K}Q_A^5PQ_Adz-\frac{1}{6}\int_{\Sigma_K}(PQ_A)^6dz.
\end{align*}
For the first term on the right-hand side
\begin{align*}
    &\int_{\Sigma_K}Q_A^5PQ_A dz=\int_{\Sigma_K}Q_A^5(Q_A-\tilde T_A)dz=\int_{\Sigma_K}\left(Q_A^6-Q_A^5\tilde T_A\right)dz.
\end{align*}
For the second term
\begin{align*}
    &\int_{\Sigma_K}(PQ_A)^6dz=\int_{\Sigma_K}(Q_A-\tilde T_A)^6dz=\int_{\Sigma_K}\left(Q_A^6-6Q_A^5\tilde T_A\right)dz+I,
\end{align*}
where
\begin{align*}
    I&=\int_{\Sigma_K}\left(15Q_A^4(\tilde T_A)^2-20Q_A^3(\tilde T_A)^3 +15 Q_A^2(\tilde T_A)^4-6 Q_A(\tilde T_A)^5+(\tilde T_A)^6\right)dz.
\end{align*}

To estimate each term, we need the following lemma.
\begin{lemma}\label{lem:Qk-rough}
Assuming the constraints \eqref{constraint-m} are satisfied, we have
    % \begin{align}
    %     % \int_{B_1}Q_A=\e^{\frac12}q(\hat \xi)\gamma(b)+O(\e^{\frac32})\\
    %     \int_{\SK}|Q_A|^\ell=\begin{cases}\e^{\frac{3}{2}}K^{-1}\log K+\e^{\frac12}|a|K^{-1}& \text{if }\ell=1,\\\e^{3-\frac{\ell}{2}}&\text{if }\ell=2,\cdots,6.\end{cases}
    % \end{align}
    \begin{align}
    \int_{\SK}|Q_A|^\ell dz&\leq C\begin{cases}\e^{\frac{3}{2}}K^{-1}\log K& \text{if }\ell=1,\\\e^{3-\frac{\ell}{2}}&\text{if }\ell=2,3,4,5,\end{cases}\label{Q234}\\
    \int_{\Sigma_K}Q_A^6dz&=\int_{\R^3}q(z)^6dz+O(\e^{9}K^9),\label{Q6}\\
    \int_{\Sigma_K}Q_A^5dz&=\e^{\frac12}  4\pi q(\hat \xi)+O(\e^{\frac{15}{2}}K^7),\label{Q5-1}\\
    \int_{\SK}Q_A^5(z)(z-b)dz&=\e^{\frac32}  4\pi R_\beta^T\nabla q(\hat \xi)+O(\e^{\frac{15}{2}}K^6),\label{Q5-2}\\
    \int_{\SK}|Q_A|^5|z-b|^2dz&=O(\e^{\frac52}).\label{Q5-3}
\end{align}
\end{lemma}

\begin{proof}
To prove \eqref{Q234}. Since $q$ is Kelvin invariant, then $q(z)=O(|z|^{-1})$ as $z\to \infty$.
    When $|z-b|<\e$, one has
    \[\left|q\left(\frac{\e(z-b)}{|z-b|^2}+\hat \xi\right)\right|\leq C\frac{|z-b|}{\e},\]
    then  $|Q_A(z)|\leq \e^{-\frac12}$ and
    \[\int_{|z-b|<\e}|Q_A|^\ell dz\leq C\e^{3-\frac{\ell}{2}}.\]
    When $\e<|z-b|$, we have \eqref{QA-exterior}. One can integrate its right-hand side respectively.
    \begin{align*}
        \e^{\frac32\ell}\int_{\{z\in \Sigma_K:|z-b|>\e\}} \frac{dz}{|z-b|^{2\ell}}\leq \begin{cases}C\e^{\frac{3}{2}}K^{-1}\log K& \text{if }\ell=1,\\C\e^{3-\frac{\ell}{2}}&\text{if }\ell\geq 2.\end{cases}
    \end{align*}
    \begin{align*}
        \e^{\frac{\ell}{2}}|a|^\ell\int_{\{z\in \Sigma_K:|z-b|>\e\}}\frac{dz}{|z-b|^\ell}\leq \begin{cases}
          C\e^{\frac{1}{2}}|a| K^{-1} &\text{if }\ell=1,\\
          C\e|a|^2 K^{-1}\log K &\text{if }\ell=2,\\
          C\e^{\frac{3}{2}}|a|^3|\log \e|&\text{if }\ell=3,\\
          C\e^{3-\frac{\ell}{2}}|a|^\ell&\text{if }\ell=4,5.
        \end{cases}
    \end{align*}
    where we have used the estimates
     \begin{align}\label{def:rhob}
        \rho_2(b):=\frac{1}{4\pi}\int_{\Sigma_K}\frac{dz}{|z-b|^2}\leq C\frac{\log K}{K},\quad \rho_1(b):=\frac{1}{4\pi}\int_{\Sigma_K}\frac{dz}{|z-b|}\leq \frac{C}{K}.
    \end{align}
    %One can relax $\Sigma_K$ to be a thin cube center at $b$ and integrate it out explicitly. The proof is elementary and we omit it.
     The proof of \eqref{Q234} is complete by combining the above three equations.

To prove \eqref{Q6} and \eqref{Q5-1}, we split the integral
$$\int_{\SK}Q_A^\ell dz= \int_{\R^3 }Q_A^\ell dz-\int_{\R^3\setminus \SK}Q_A^\ell dz\quad \mbox{for}\quad \ell=5,6.$$
On one hand, we use \eqref{QA-exterior} to get
 \begin{align*}
    \int_{\R^3\setminus \SK} Q_A^\ell dz\leq C\e^{\frac{3\ell}{2}}K^{2\ell-3}+C \e^\ell |a|^\ell K^{\ell-3}\leq C\e^{\frac{3\ell}{2}}K^{2\ell-3},\quad \ell=5,6.
\end{align*}
% \begin{align*}
%     \int_{\SK}Q_A^6= \int_{\R^3 }Q_A^6-\int_{\R^3\setminus \SK}Q_A^6.
% \end{align*}
% On the one hand, we use \eqref{QA-exterior}
% \begin{align}\label{QA-exterior-2}
%    |Q_A(z)|\leq \frac{C\e^{\frac32}}{|z-b|^2}+\frac{C\e^{\frac12}|a|}{|z-b|},\quad x\in \R^3\setminus \SK,
% \end{align}
% then
% \[\int_{\R^3\setminus \SK}Q_A^6=\int_{\R^3\setminus B_2}Q_A^6+\int_{B_2\setminus \SK}Q_A^6\leq C\e^9K^9+C\e^3|a|^6K^3.\]
On the other hand, making a change of variables, $z=b+\e \frac{ R_\beta^T x}{|x|^2}$, one obtains that
\begin{align*}
    \int_{\R^3}Q_A^6 dz=\int_{\R^3}q( z+\hat \xi)^6dz=\int_{\R^3}q(z)^6dz.
\end{align*}
% Combining these two facts, we get the result.
% To prove \eqref{Q5-1}, we similarly split the integral. We apply \eqref{QA-exterior-2} to get $ \int_{\R^3\setminus \SK}|Q_A|^5=O(\e^{\frac{15}{2}}K^7+\e^{\frac52}|a|^5K^2)$
 % \begin{align*}
 %     \left|\int_{\R^3\setminus B_1}Q_A^5\right|=O(\e^{\frac{15}{2}})
 % \end{align*}
 and
    \begin{align*}
    \int_{\R^3} Q_A^5dz&=\int_{\R^3}\frac{\e^{\frac52}}{|z-b|^5}\left(q\left(\frac{\e  (z-b)}{|z-b|^2}+\hat \xi\right)\right)^5dz\\
    &=\e^{\frac12}\int_{\R^3}\frac{1}{|z|}[q( z+\hat \xi)]^5dz=\e^{\frac12}\int_{\R^3}\frac{1}{|z-\hat \xi|}q(z)^5dz=  \e^{\frac12}4\pi q(\hat \xi).
\end{align*}
 where we have used \[\int_{\R^3}\frac{1}{|z-\xi|}q(z)^5dz=  4\pi q(\xi)\] for any $\xi$ in the last step.
 %For the second integral, on $x\in \R^3\setminus B_1$,
This completes \eqref{Q6} and \eqref{Q5-1}.

The proof of \eqref{Q5-2} and \eqref{Q5-3} is similar to the previous proofs. We omit it.
\end{proof}

% For \eqref{Q5-2}, one can estimate similarly. First, $\int_{\R^3\setminus \SK}|Q_A|^5|z-b|=O(\e^{\frac{15}{2}}K^6+\e^{\frac52}|a|^5K)$. Second
% \begin{align*}
%     \int_{\R^3}Q_A^5(z)(z-b)dz
%     &=\e^{\frac32}\int_{\R^3}R_\beta^Tx|x|^{-3}[q(R_\beta x+\hat \xi)]^5dx
%     %&=\e^{\frac32}\int_{\R^3}R_{-\theta}z|z|^{-3}[q( z+\xi+a_1)]^5dz+O(\e_1^{\frac52})\\
%     =4\e^{\frac32}\pi  R_\beta^T\nabla q(\hat \xi).
% \end{align*}
% where we have used
% \[\int_{\R^3}\frac{x-\xi}{|x-\xi|^3}(q(x))^5dx= 4\pi\nabla q(\xi)\]
% for any $\xi$ in the last step.

% For \eqref{Q5-3}, its proof is similar to the approach in Lemma \ref{lem:Qk-rough}. We omit it here.

It follows from \eqref{bd-tTA} and \eqref{Q234} that $I=O(\e^4K^4)$.
% \begin{align*}
%     I=O(\e^4K^4).
% \end{align*}
Therefore
\begin{align}\label{JPQ-1m}
    J(PQ_A)=\frac{1}{3}\int_{\Sigma_K}Q_A^6dz+\frac{1}{2}\int_{\Sigma_K}Q_A^5(T_A+\varphi_A+\psi_A)dz+O(\e^{4}K^4).
\end{align}
We will compute the first two terms on the right-hand side.

\begin{lemma}\label{lem:QA5phiA}
Under the constraint \eqref{constraint-m}, we have
    \begin{align*}
        \frac{1}{4\pi}\int_{\Sigma_K}Q_A^5\varphi_Adz&=\e  [q(\hat \xi)]^2H_0^e(b,b)+\e^2  q(\hat \xi)\w1\cdot[\nabla_z H_0^e(b,b)+\nabla_p H_0^e(b,b)]\\
        &\quad +\e^3 \w1^T\nabla_{z,p}^2H_0^e(b,b)w+O(\e^4K^4).
    \end{align*}
\end{lemma}
\begin{proof}
Recall the expansion of $\varphi_A$ in \eqref{phiA-expn}. Then applying Lemma \ref{lem:Qk-rough} yields
\begin{align}
\begin{split}\label{Q5varphi-m}
    \int_{\Sigma_K}Q_A^5\varphi_Adz&
    %=\int_{\Sigma_K}Q_A^5[\e^{\frac12}  q(\hat \xi) H_0^e(z,b)dz+ \e^{\frac{3}{2}} \w1\cdot\nabla_p H_0^e(z,b) +\\
    =\e^{\frac12}  q(\hat \xi)\int_{\Sigma_K}Q_A^5H_0^e(z,b)dz+ \e^{\frac32}\w1\cdot \int_{\Sigma_K}Q_A^5\nabla_pH_0^e(z,b)dz\\
    &\quad +\frac16\e^{\frac52}\W2_{ij}\int_{\Sigma_K}Q_A^5 \partial_{p_ip_j}^2H_0^e(z,b)dz +O(\e^4K^4).
    \end{split}
\end{align}

Let us compute each term on the right-hand side. Using Lemma \ref{lem:H0e-d}, one has the Taylor expansion \[H_0^e(z,b)=H_0^e(b,b)+\nabla_zH_0^e(b,b)\cdot (z-b)+O(|z-b|^2K^{3}).\]
Applying \eqref{Q5-1}-\eqref{Q5-3}, we have
\begin{align*}
    \int_{\Sigma_K}Q_A^5H_0^e(z,b)dz&=H_0^e(b,b)\int_{\Sigma_K}Q_A^5dz+\nabla_z H_0^e(b,b)\cdot \int_{\Sigma_K}Q_A^5(z-b)dz+O(\e^{\frac52}K^{3})\\
    &=\e^{\frac12}4\pi q(\hat\xi)H_0^e(b,b)+\e^{\frac32} 4\pi w\cdot  \nabla_zH_0^e(b,b)+O(\e^{\frac52}K^{3}).
\end{align*}
Similarly, using $\nabla_pH_0^e(z,b)=\nabla_pH_0^e(b,b)+\nabla_{z,p}^2H_0^e(b,b)\cdot (z-b)+O(|z-b|^2K^{4})$,
\begin{align*}
    \int_{\Sigma_K}Q_A^5\nabla_pH_0^e(z,b)dz&=\nabla_pH_0^e(b,b)\int_{\Sigma_K}Q_A^5dz+\nabla_{z,p}^2H_0^e(b,b)\cdot \int_{\Sigma_K}Q_A^5 (z-b)dz+O(\e^{\frac52}K^4)\\
    & =\e^{\frac12}4\pi q(\hat \xi) \nabla_p H_0^e(b,b)+\e^{\frac32}4\pi \partial_{z,p}^2H_0^e(b,b)R_\beta^T\nabla q(\hat \xi)+O(\e^{\frac52}K^4)
\end{align*}
and
\begin{align*}
    \int_{\Sigma_K}Q_A^5\partial_{p_ip_j}^2H_0^e(z,b)dz&=\int_{\Sigma_K}Q_A^5\partial_{p_ip_j}^2H_0^e(b,b)dz+\nabla_z\partial_{p_ip_j}^2H_0^e(b,b)\cdot\int_{\Sigma_K}Q_A^5 (z-b)dz\\
    &\quad+O(\e^{\frac52}K^{5})\\
    &=O(\e^{\frac12}|q(\hat \xi)|K^2)+O(\e^{\frac32}K^4)+O(\e^{\frac52}K^{5})=O(\e^{\frac32}K^4).
\end{align*}
Inserting the above three estimates back to \eqref{Q5varphi-m}, we can get the conclusion.
\end{proof}

\begin{lemma}\label{lem:QA5TA}
Under the constraint \eqref{constraint-m}, we have
\begin{align*}
    \frac{1}{4\pi}\int_{\Sigma_K}Q_A^5T_Adz&=\e [q(\hat \xi)]^2\gamma(b,b)+\e^2 q(\hat \xi)w\cdot[\nabla_z\gamma(b,b)+\nabla_p\gamma(b,b)]\\
    &\quad +\e^3w^T\nabla_{z,p}^2\gamma(b,b)w+O(\e^4 K^4).
\end{align*}
\end{lemma}
\begin{proof}
The proof is almost identical to the previous one.
One should use $T_A$ in \eqref{TA-expn} and the estimates of $\gamma$ in Lemma \ref{lem:gamma-first}. We leave the details to the readers.
    % \begin{align*}
    %     \int_{\Sigma_K}Q_A^5T_A&=\e^{\frac12}q(\hat \xi)\int_{\Sigma_K}Q_A^5\gamma(z,b)+\e^{\frac32}\nabla q(\hat \xi)\cdot\int_{\Sigma_K}Q_A^5 R_\beta \nabla_p\gamma(z,b)\\
    %     &\quad +\frac13\e^{\frac52}(R_\beta^T\nabla^2q(\hat \xi)R_\beta)_{ij}\int_{\Sigma_K}Q_A^5\partial_{p_ip_j}^2\gamma(z,b)+O(\e^{3}K)]
    % \end{align*}
    % Let us compute each term on the right-hand side. First,
    % \begin{align*}
    %     \int_{\Sigma_K} Q_A^5(z)\gamma(z,b)&=\gamma(b,b)\int_{\Sigma_K}Q_A^5(z)+ \nabla_z\gamma(b,b)\cdot \int_{\Sigma_K}Q_A^5 (z-b)+O(\e^{\frac52}K^3)\\
    %     &=\e^{\frac12} 4\pi q(\hat \xi)\gamma(b,b)+\e^{\frac32}4\pi \nabla_z\gamma(b,b)\cdot R_\beta^T \nabla q(\hat \xi)+O(\e^{\frac32}).
    % \end{align*}
    % \begin{align*}
    %     \int_{\Sigma_K} Q_A^5(z)\nabla_p\gamma(z,b)&=\int_{\Sigma_K}Q_A^5(z)\nabla_p\gamma(b,b)+\int_{\Sigma_K}Q_A^5(z)\nabla_{zp}^2\gamma(b,b)(z-b)+O(\e^{\frac52}K^4)\\
    %     &=\e^{\frac12}4\pi q(\hat\xi)\nabla_p \gamma(b,b)+\e^{\frac32}4\pi\partial_{z,p}^2\gamma(b,b)R_\beta^T\nabla q(\hat \xi)+O(\e^{\frac32}K)
    % \end{align*}
    % \begin{align*}
    %     \int_{\Sigma_K}Q_A^5\partial_{p_ip_j}^2\gamma(z,b)&=\partial_{p_ip_j}^2\gamma(b,b)\int_{\Sigma_K}Q_A^5+\nabla_z\partial_{p_ip_j}^2\gamma(b,b)\cdot \int_{\Sigma_K}Q_A^5(z-b)+O(\e^{\frac52}K^5)\\
    %     &=O(\e^{\frac32}K^3).
    % \end{align*}
    % Inserting the above to ??, the proof is complete.
    % % \begin{align*}
    % %     \int_{\Sigma_K}Q_A^5T_A=\e [q(\hat \xi)]^2\gamma(b,b)+\e^2 q(\hat \xi)\nabla q(\hat \xi)\cdot R_\beta [\nabla_z\gamma(b,b)+\nabla_p\gamma(b,b)]+...
    % % \end{align*}
\end{proof}

\begin{lemma}\label{lem:QA5psiA}
Under the constraint \eqref{constraint-m}, we have
    \begin{align*}
        \int_{\Sigma_K}Q_A^5\psi_Adz=-\la \int_{\Sigma_K}Q_A^2dz+O(\e^3K\log K).
    \end{align*}
\end{lemma}
\begin{proof}
Using \eqref{301} and \eqref{psi-F}
\begin{align}
    \begin{split}\label{QA5=QA2}
    \int_{\Sigma_K}Q_A^5\psi_Adz&=-\int_{\Sigma_K}(\Delta PQ_A+\la PQ_A)\psi_Adz=-\int_{\Sigma_K}PQ_A(\Delta \psi_A+\la \psi_A)dz\\
    &=-\la\int_{\Sigma_K}(Q_A-\tilde T_A)(Q_A-T_A-\varphi_A)dz+\int_{\Sigma_K}PQ_A  Fdz\\
    &=-\la \int_{\Sigma_K}Q_A^2dz+\la \int_{\SK}Q_A(\tilde T_A+T_A+\varphi_A)dz+\int_{\Sigma_K}PQ_A  Fdz,
    \end{split}
\end{align}
where $F$ is defined in \eqref{def:F}.
% \begin{align*}
%     \tilde F(z)=Q_A(\bar z e^{2i\theta_0})^5-\sum_{j=1}^{K/2-1}[Q_A(z e^{4ji\theta_0})^5-Q_A(\bar z e^{(4j+2)\theta_0})^5]
% \end{align*}
% Since $|\bar z e^{(4j+2)i\theta_0}-b|\geq 1/K>\e$ for $j\geq 0$ and $|z e^{4ji\theta_0}-b|\geq 1/K>\e$ for $j\geq 1$, then by \eqref{QA-exterior-2},
% \begin{align*}
%     |\tilde F(z)|\leq&\  C\e^{\frac{15}{2}}\left(\sum_{j=0}^{K/2-1}\frac{1}{|\bar z e^{(4j+2)i\theta_0}-b|^{10}}+\sum_{j=1}^{K/2-1}\frac{1}{|z e^{4ji\theta_0}-b|^{10}}\right)\\
%     &+C\e^{\frac52}|a|^5\left(\sum_{j=0}^{K/2-1}\frac{1}{|\bar z e^{(4j+2)i\theta_0}-b|^{5}}+\sum_{j=1}^{K/2-1}\frac{1}{|z e^{4ji\theta_0}-b|^{5}}\right)\\
%     \leq &C\e^{\frac{15}{2}}K^{10}+C\e^{\frac52}|a|^5K^5.
% \end{align*}
Recall \eqref{bd-tTA}, Lemma \ref{lem:bd-TA}, Lemma \ref{lem:phiA-m} and \eqref{Q234},
\begin{align*}
    \left|\int_{\Sigma_K}Q_A(\tilde T_A+T_A+\varphi_A)dz\right|\leq C\e^{\frac32}K^2\int_{\Sigma_K}|Q_A|dz\leq C\e^3K\log K.
\end{align*}
Using \eqref{F-bd}, \eqref{Q234} and \eqref{bd-tTA}, one has
\begin{align*}
    \left|\int_{\Sigma_K}PQ_A Fdz\right|\leq C\e^{\frac{3}{2}}K^{2}\int_{\Sigma_K}|PQ_A|dz\leq C\e^3K\log K.
\end{align*}
Therefore, plugging in the above estimates back to \eqref{QA5=QA2},
\begin{align*}
    \int_{\Sigma_K}Q_A^5\psi_Adz=-\la \int_{\Sigma_K}Q_A^2dz+O(\e^3K\log K ).
\end{align*}

\end{proof}
Now, let us compute $\int_{\SK}Q_A^2dz$. We recall $\rho_2(b)$ in \eqref{def:rhob}.
% \begin{align}\label{def:rhob}
%     \rho_2(b)=\frac{1}{4\pi}\int_{\SK}\frac{1}{|z-b|^2}dz.
% \end{align}
%It is easy to see that $\rho_2(b)$ is a smooth and bounded function in $\Sigma_K$.
Taking the derivative to $b$ on both sides, we get
\begin{align*}
    \nabla\rho_2(b)=\frac{1}{4\pi}P.V.\int_{\SK}\frac{2(z-b)}{|z-b|^4}dz=-\frac{1}{4\pi}P.V.\int_{\R^3\setminus \SK}\frac{2(z-b)}{|z-b|^4}dz.
\end{align*}
It is not hard to show that $|\nabla \rho_2(b)|\leq C\log K$ in the constraint \eqref{constraint-m}.
%We have the following lemma about the expansion of certain integrals.

 \begin{lemma}\label{lem:int-QA-m}
Assume \eqref{constraint-m}, we have
   \begin{align*}
        \frac{1}{4\pi}\int_{\SK}Q_A^2dz&=
     \e^2P.V.\int_{\R^3}\frac{[q(z+ \xi)]^2}{4\pi |z|^4}dz+O(\e^3K(\log K)^2).
    \end{align*}
\end{lemma}
\begin{proof}
Note that
\begin{align*}
    \frac{1}{4\pi}\int_{\SK}Q_A^2dz=\frac{1}{4\pi}\int_{\SK}\frac{\e}{|z-b|^2}q\left(\frac{\e R_\beta(z-b)}{|z-b|^2}+\hat \xi\right)^2dz.
\end{align*}
Notice that $\frac{1}{4\pi}\int_{\SK}\frac{\e q(\hat \xi)^2}{|z-b|^2}dz= \e [q(\hat \xi)]^2\rho_2(b)=O(\delta^{-2}\e^3K^{-1}(\log K)^2)$.
It suffices to estimate
\begin{align*}
    I=\frac{1}{4\pi}\int_{\SK}\frac{\e}{|z-b|^2}\left(q\left(\frac{\e R_\beta(z-b)}{|z-b|^2}+\hat \xi\right)^2-q(\hat \xi)^2\right)dz.
\end{align*}
Again, we split the integral into two on $\R^3$ and $\R^3\setminus \SK$. First, using a change of variable,
\begin{align*}
    &~\int_{\R^3}\frac{\e}{|z-b|^2}\left[q(\tfrac{\e R_\beta(z-b)}{|z-b|^2}+\hat \xi)^2-q(\hat \xi)^2\right]dz=\e^2P.V.\int_{\R^3}\frac{[q(z+\hat \xi)]^2-[q(\hat \xi)]^2}{|z|^4}dz\\
    &=\e^2P.V.\int_{\R^3}\frac{[q(z+\xi)]^2}{|z|^4}dz+\e^2P.V.\int_{\R^3}\frac{[q(z+\hat \xi)]^2-[q(\hat \xi)]^2-[q(z+\xi)]^2}{|z|^4}dz\\
    &=\e^2P.V.\int_{\R^3}\frac{[q(z+\xi)]^2}{|z|^4}dz+O(\e^2|a|).
\end{align*}
% Note that
% \begin{align*}
%     &P.V.\int_{\R^3}\frac{[q(z+\hat \xi)]^2-[q(\hat \xi)]^2}{|z|^4}dz\\
%     &=P.V.\int_{\R^3}\frac{[q(z+\xi)]^2}{|z|^4}dz+P.V.\int_{\R^3}\frac{[q(z+\hat \xi)]^2-[q(\hat \xi)]^2-[q(z+\xi)]^2}{|z|^4}dz\\
%     &=P.V.\int_{\R^3}\frac{[q(z+\xi)]^2}{|z|^4}dz+O(|a|).
% \end{align*}
Second, since $\e/|z-b|\lesssim  1$ for $z\in \R^3\setminus \SK$, then
\begin{align*}
    q\left(\frac{\e R_\beta(z-b)}{|z-b|^2}+\hat \xi\right)^2-q(\hat \xi)^2=2q(\hat \xi)\nabla q(\hat \xi)\cdot \frac{\e R_\beta (z-b)}{|z-b|^2}+O(\e^2|z-b|^{-2})
\end{align*}
%$ \left|q(\frac{\e R_\beta(z-b)}{|z-b|^2}+\hat \xi)^2-q(\hat \xi)^2\right|\leq C \frac{|q(\hat \xi)| \e }{|z-b|}$.
% \begin{align*}
%     \left|q\left(\frac{\e R_\beta(z-b)}{|z-b|^2}+\hat \xi\right)^2-q(\hat \xi)^2\right|\leq C \frac{|q(\hat \xi)| \e }{|z-b|}
% \end{align*}
Therefore
\begin{align*}
    &\frac{1}{4\pi}\int_{\R^3\setminus\SK}\frac{\e}{|z-b|^2}\left(q\Big(\frac{\e R_\beta(z-b)}{|z-b|^2}+\hat \xi\Big)^2-q(\hat \xi)^2\right)dz\\
    &=\e^2 q(\hat \xi)\w1\cdot \frac{1}{2\pi }P.V.\int_{\R^3\setminus \SK}\frac{ (z-b)}{|z-b|^4}+O(\e^3)\\
    &=-\e^2 q(\hat \xi)\w1\cdot \nabla \rho_2(b)+O(\e^3)
    =O(\e^2|a|\log K).
\end{align*}
Combining the two results, the proof is complete.
\end{proof}

Inserting Lemma \ref{lem:QA5phiA}, Lemma \ref{lem:QA5TA}, Lemma \ref{lem:QA5psiA} and Lemma \ref{lem:int-QA-m} to \eqref{JPQ-1m}, we obtain
\begin{align}\label{JPQ-412}
    J(PQ_A)=\frac13\int_{\R^3}q^6dz+2\pi \Psi(A)+O(\e^3K(\log K)^2),
\end{align}
where
% \begin{align*}
%     \Psi(A)&=\e[q(\hat \xi)]^2\mathcal{H}(b,b)+\e^2 q(\hat \xi)w\cdot [\nabla_z\mathcal{H}(b,b)+\nabla_p\mathcal{H}(b,b)]+\e^3 w^T  \nabla_{z,p}^2\mathcal{H}(b,b)w\\
%     &\quad -\la \e [q(\hat \xi)]^2\rho_2(b)-\la\e^2[C_*(\xi)-q(\hat \xi)w\cdot \nabla \rho_2(b) ]
% \end{align*}
\begin{align}\label{PsiA-re}
\begin{split}
    \Psi(A)&=\e[q(\hat \xi)]^2\mathcal{H}(b,b)+\e^2 q(\hat \xi)w\cdot [\nabla_z\mathcal{H}(b,b)+\nabla_p\mathcal{H}(b,b)]\\
    &\quad +\e^3 w^T  \nabla_{z,p}^2\mathcal{H}(b,b)w-\la\e^2C_*(\xi).
    \end{split}
\end{align}
and $\mathcal{H}(z,p)=H_0^e(z,p)+\gamma(z,p)$ and
\begin{align*}
    C_*(\xi)=P.V.\int_{\R^3}\frac{[q(z+ \xi)]^2}{4\pi |z|^4}dz.
\end{align*}

\begin{remark}\label{rmk:JPQA-C1}
    Clearly, $\Psi(A)$ depends smoothly on the parameters of $A$. By elliptic theory and Remark \ref{rmk:Ta-C1}, it is not hard to show that the dependence of $J(PQ_A)$ on the parameters in $A$ is at least $C^1$.
\end{remark}

\section{Gluing Procedure}
In this section, we outline the gluing procedure by separating the perturbation into inner and outer components. Specifically, we will prove that the inner part is the dominant term.

Let $PQ_A$ denote the approximate solution defined in section 3
\begin{equation}
\label{g.sol}
PQ_A=Q_A-\tilde T_A=Q_A-T_A-\varphi_A-\psi_A.
\end{equation}
We introduce the following form of $PQ_A$ for convenience
\begin{align*}
PQ_{A'}=~&\e^\frac12Q_A-\e^\frac12T_A-\e^\frac12\varphi_A-\e^\frac12\psi_A\\
=~&Q_{A'}-\e^\frac12T_A-\e^\frac12\varphi_A-\e^\frac12\psi_A\\
=~&\frac{1}{|y-b_\e|}q\left(\frac{R_\beta(y-b_\e)}{|y-b_\e|^2}+\e^{-1}\xi+\e^{-1}a\nu(\xi)\right)-\e^{\frac12}\tilde T_A
% &-\e^\frac12T_A-\e^\frac12\varphi_A-\e^\frac12\psi_A,
\end{align*}
where $y=\frac{z}{\e}$ and $b_\e=\frac{b}{\e}.$ The Brezis-Nirenberg problem is equivalent to finding $\phi$ such that
\begin{equation}
\label{g.eq}
\Delta_y(PQ_{A'}+\phi)+\lambda\e^2(PQ_{A'}+\phi)+(PQ_{A'}+\phi)^5=0,
\end{equation}
which can be rewritten as
\begin{equation}
\label{g.eq-1}
\begin{cases}
\Delta_y\phi+\lambda\e^2\phi+5PQ_{A'}^4\phi=-E-N_\e(\phi)\quad &\mbox{in}\quad \Sigma_{K,\e},\\
\phi=0  &\mbox{on}\quad \partial\Sigma_{K,\e},
\end{cases}
\end{equation}
where $\Sigma_{K,\e}=\{y\mid \e y\in  \Sigma_k\}$,
\begin{align*}
E=&~\Delta PQ_{A'}+\lambda\e^2 PQ_{A'}+PQ_{A'}^5=PQ_{A'}^5-Q_{A'}^5\\
=&-5\e^\frac12Q_{A'}^4(T_A+\varphi_A+\psi_A)+O(\e^4K^4)Q_{A'}^3+O(\e^{6}K^6)Q_{A'}^2\\
&+O(\e^8K^8)Q_{A'}+O(\e^{10}K^{10}),
\end{align*}
and
\begin{align*}
N_\e(\phi)=(PQ_{A'}+\phi)^5-PQ_{A'}^5-5PQ_{A'}^4\phi.
\end{align*}
It is important to mention how the functions $PQ_{A'}$ and $Q_{A'}$ depend on the parameters $A'=(\Lambda,\e^{-1}a,\e^{-1}\xi,b_\e,\beta)$. Particularly, the dependence of $Q_{A'}$ on $\Lambda$ can be understood as follows:
\begin{equation*}
Q_{A'}=\left.\frac{\Lambda^\frac12}{|y-b_\e|}q\left(\frac{R_\beta\Lambda(y-b_\e)}{|y-b_\e|^2}+\e^{-1}\xi+\e^{-1}a\nu(\xi)\right)\right|_{\Lambda=1}.
\end{equation*}
Since we scale  $Q_A$ by $\e$ for the space variable, so the parameter $\Lambda$ does not appear in $Q_{A'}$ and we will not carry it in the expression of $Q_{A'}$ in the following argument.

We separate $\phi$ in the following form
\begin{equation}
\label{g.decom}
\phi(y)=\eta_{2r}^{\bar\e}(y)\phi_{\rm in}(y)+\phi_{\rm out}(y),
\end{equation}
where $\bar\e=\e K$ and
$$\eta_{2r}^{\bar \e}(y)=\begin{cases}
1,\quad &\mbox{if}\quad \bar\e|y-b_\e|\leq 2r,\\
\\
0,\quad &\mbox{if}\quad \bar\e|y-b_\e|\geq 4r.
\end{cases}$$
for some $r\in\left(0,\frac{1}{32}\right)$ and $\phi_{\rm in}(y)$, $\phi_{\rm out}(y)$ satisfy the following equation respectively
\begin{equation}
\label{g.lin-in}
\begin{cases}
\begin{aligned}
\Delta_y\phi_{\rm in}+\lambda\e^2\phi_{\rm in}+5PQ_{A'}^4\phi_{\rm in}
=&-\eta_r^{\bar\e}\left(N_\e(\eta_{2r}^{\bar\e}\phi_{\rm in}+\phi_{\rm out})+E\right)\\
&-5\eta_r^{\bar\e}PQ_{A'}^4\phi_{\rm out}
\end{aligned}
&\mbox{in}\quad B_{\frac{4r}{\bar \e}}(b_\e),\\
\\
\phi_{\rm in}=0~ &\mbox{on}\quad\partial B_{\frac{4r}{\bar \e}}(b_\e),
\end{cases}
\end{equation}
and
\begin{equation}
\label{g.lin-out}
\begin{cases}
\begin{aligned}
&\Delta_y\phi_{\rm out}+\lambda\e^2\phi_{\rm out}+5(1-\eta_r^{\bar\e})PQ_{A'}^4\phi_{\rm out}\\
&=-(1-\eta_{r}^{\bar\e})\left(E+N(\eta_{2r}^{\bar\e}\phi_{\rm in}+\phi_{\rm out})\right)+\Delta\eta_{2r}^{\bar \e}\phi_{\rm in}+2\nabla\eta_{2r}^{\bar\e}\nabla \phi_{\rm in}
\end{aligned}~&\mbox{in}\quad {\Sigma_{K,\e}},\\
\phi_{\rm out}=0 &\mbox{on}\quad \partial\Sigma_{K,\e}.
\end{cases}
\end{equation}
It is crucial to see that $|\nabla \eta_{2r}^{\bar\e}|\leq C\bar\e$ and $|\Delta\eta_{2r}^{\bar\e}|\leq C\bar\e^2$, which makes equations \eqref{g.lin-in} and \eqref{g.lin-out} weakly coupled.  Next we shall reduce this system to a single problem in the ball. To do this, we first pick out a small $\phi_{\rm in}$ and solve $\phi_{\rm out}$ from \eqref{g.lin-out}. We shall need the fact that the operator $\Delta_y+\lambda\e^2$ satisfies the Maximum principle in $\Sigma_{K,\e}$ provided $K$ is large enough and the following lemma
\begin{lemma}
\label{leg.1}
Let $\phi$ solve
\begin{equation*}
\begin{cases}
\Delta \phi+\lambda\e^2\phi+f=0\quad &\mbox{in}\quad \Sigma_{K,\e},\\
\phi=0  &\mbox{on}\quad \partial\Sigma_{K,\e}.
\end{cases}
\end{equation*}
Then we have
$$\|\phi\|_{L^\infty(\Sigma_{K,\e})}\leq \frac{C}{\bar\e^2}\|f\|_{L^\infty(\Sigma_{K,\e})}.$$
\end{lemma}

\begin{proof}
We construct a function
$$w=\frac{2\pi^2}{\bar\e^2}-r^2\sin^2\theta.$$
After straightforward computation, we check that
\begin{equation*}
\Delta w+\frac{\lambda\bar\e^2}{K^2} w=-2+\frac{\lambda\bar\e^2}{K^2}\left(\frac{2\pi^2}{\bar\e^2}-r^2\sin^2\theta\right)<-1
\end{equation*}
provided $K$ is large enough. Thus, by maximum principle, we derive that
$$\left(\frac{2\pi^2}{\bar\e^2}-r^2\sin^2\theta\right)\|f\|_{L^\infty(\Sigma_{K,\e})}$$ defines a super-solution and the conclusion follows easily.
\end{proof}

By Lemma \ref{leg.1}, we get that
\begin{equation}
\label{6.lin-out-2}
\begin{aligned}
\|\phi_{\rm out}\|_{L^\infty(\Sigma_{K,\e})}
\leq ~&\frac{C}{\bar\e^2}\left\|\left(1-\eta_r^{\bar\e}\right){PQ}_{A'}^4\phi_{\rm out}\right\|_{L^\infty(\Sigma_{K,\e})}\\
&+\frac{C}{\bar\e^2}\left\|\left(1-\eta_r^{\bar\e}\right)(E+N(\eta_{2r}^{\bar\e}\phi_{\rm in}+\phi_{\rm out}))\right\|_{L^\infty(\Sigma_{K,\e})}\\
&+\frac{C}{\bar\e^2}\left\|\nabla\eta_{2r}^{\bar\e}\nabla_y\phi_{\rm in}\right\|_{L^\infty(\Sigma_{K,\e})}+\frac{C}{\bar\e^2}\|\Delta \eta_{2r}^{\bar\e}\phi_{\rm in}\|_{L^\infty(\Sigma_{K,\e})}.
\end{aligned}
\end{equation}
If we presumably assume that
\begin{equation}
\label{g.assu}
\|\phi_{\rm in}\|_*:=\sup_{y\in B_{4r/\bar\e}}
\|\langle y-b_\e\rangle \phi_{\rm in}\|+\sup_{y\in B_{4r/\bar\e}}
\|\langle y-b_\e\rangle^2 \nabla\phi_{\rm in}\|~\mbox{is small}.
\end{equation}
Then we can apply the contraction mapping principle and the term $N_\e$ has a powerlike behavior with power greater than $1$ to deduce a unique (small) solution $\phi_{\rm out}$ with
\begin{equation}
\label{g.existence}
\|\phi_{\rm out}(\phi_{\rm in})\|_{L^\infty(\Sigma_{K,\e})}
\leq \frac{C}{\bar\e^2}\|\left(1-\eta_r^{\bar\e}\right)E\|_{L^\infty(\Sigma_{K,\e})}+C\bar\e\|\phi_{\rm in}\|_*.
\end{equation}
Given any two functions $\phi_{\rm in,1}$ and $\phi_{\rm in, 2}$, we see that the corresponding solution of \eqref{g.lin-out} satisfies a Lipschitz condition of the form
\begin{equation}
\label{g.lip}
\|\phi_{\rm out}(\phi_{\rm in,1})-\phi_{\rm out}(\phi_{\rm in,2})\|_{L^\infty(\Sigma_{K,\e})}
\leq C\bar\e\|\phi_{\rm in,1}-\phi_{\rm in,2}\|_*.
\end{equation}
In addition, one can follow the similar process to derive that
\begin{equation}
\label{g.der.existence}
\begin{aligned}
\|\nabla_{A'}\phi_{\rm out}(\phi_{\rm in})\|_{L^\infty(\Sigma_{K,\e})}
\leq ~&\frac{C}{\bar\e^2}\|\nabla_{A'}((1-\eta_r^{\bar\e})E)\|_{L^\infty(\Sigma_{K,\e})}+C\bar\e\|\phi_{\rm in}\|_*\\
&+C\bar\e\|\nabla_{A'}\phi_{\rm in}\|_*+C\bar\e^2\|\phi_{\rm out}(\phi_{\rm in})\|_{L^\infty(\Sigma_{K,\e})},
\end{aligned}
\end{equation}
and
\begin{equation}
\label{g.der.lip}
\begin{aligned}
\|\nabla_{A'}(\phi_{\rm out}(\phi_{\rm in,1})-\phi_{\rm out}(\phi_{\rm in,2}))\|_{L^\infty(\Sigma_{K,\e})}
\leq ~&C\bar\e\|\phi_{\rm in,1}-\phi_{\rm in,2}\|_*\\
&+C\bar\e\|\nabla_{A'}(\phi_{\rm in,1}-\phi_{\rm in,2})\|_*.
\end{aligned}
\end{equation}
Then the full problem can be reduced to solving the nonlocal problem in the ball $B_{4r/\bar\e}(b_\e)$,
\begin{equation}
\label{g.final-lin}
\begin{cases}
\begin{aligned}
\Delta\phi_{\rm in}+\lambda\e^2\phi_{\rm in}
+5PQ_{A'}^4\phi_{\rm in}=&-\eta_r^{\bar\e}N_\e(\phi_{\rm in}+\phi_{\rm out}(\phi_{\rm in}))\\
&-\eta_r^{\bar\e}(E+5PQ_{A'}^4\phi_{\rm out}(\phi_{\rm in}))
\end{aligned}
~&\mbox{in}\quad B_{4r/\bar\e}(b_\e),\\
\phi_{\rm in}=0  \quad &\mbox{on}\quad \partial B_{4r/\bar\e}(b_\e).
\end{cases}
\end{equation}
Next section, we shall study \eqref{g.final-lin} and investigate the related linear problem.

\section{The linear and nonlinear problem}
%In this section, we shall study the linearized problem of \eqref{g.final-lin} and the nonlinear problem. Since this part is now very classical and standard, we shall list the main result and provide a sketch proof. For the details we refer the readers to \cite{del2003two,musso2016sign,wei2005arbitrary}.

In this section, we will study the linearized problem associated with \eqref{g.final-lin} as well as the nonlinear problem. Given that this material is now well-established and standard, we will state the main result and provide a concise outline of its proof.  For the details we refer the readers to \cite{del2003two,musso2016sign,wei2005arbitrary}.

We first study the following linearized problem
\begin{align}
\label{l.lin}
\begin{cases}
\Delta \phi+\la\e^2 \phi+5PQ_{A'}^4\phi=h+\sum\limits_{j=0}^5 c_j  PQ_{A'}^4\hat Z_j(y),&\text{in }B_{4r/\bar\e}(b_\e),\\
\phi=0&\text{on }\partial B_{4r/\bar\e}(b_\e),\\
\int_{B_{4r/\bar\e}(b_\e)} \phi PQ_{A'}^4 \hat Z_j(y)dy=0,\quad j=0,1,\cdots,5,
\end{cases}
\end{align}
where $A'=(\Lambda,\e^{-1}a,\e^{-1}\xi,b_\e,\beta)$ and $\hat Z_j(y)=\chi(y)Z_j(y)$. Here $Z_j(y)$ is the kernel function introduced in section 2.3 \footnote{The parameters $(1,\xi,0,0,\theta_*)$ are replaced by $(\Lambda,\e^{-1}\xi,\e^{-1}a,b_\e,\beta)$.} and $\chi(y)$ is the characteristic function such that
\begin{equation*}
\chi(y)=\begin{cases}
1,\quad y\in B_{2r/\bar\e}(b_\e),\\
0,\quad y\in B_{4r/\bar\e}(b_\e)^c.
\end{cases}
\end{equation*}
To state the main result concerning the linearized problem, we introduce the following weighted function space:
\begin{align}
\label{l.h}
\|h\|_{**}=\sup_{y\in B_{4r/\bar\e}(b_\e)}\left|\langle y-b_\e\rangle^{3+2\sigma}h(y)\right|,
\end{align}
\begin{align}
\label{l.phi}
\|\phi\|_{*}=\sup_{y\in B_{4r/\bar\e}(b_\e)}|\langle y-b_\e\rangle \phi(y)|+\sup_{y\in B_{4r/\bar\e}(b_\e)}|\langle y-b_\e\rangle ^{2}\nabla\phi(y)|,
\end{align}
where $\sigma$ is a sufficiently small positive number and $\langle y\rangle=\sqrt{1+|y|^2}$. The first result of this section is the following.

\begin{proposition}
\label{prl.1}
Suppose that the parameters $A$ and $\e$ satisfy the relation in \eqref{constraint-m}. Then there exists $K_0$ large enough such that for all $K>K_0$ and all $h\in C^\alpha (B_{4r/\bar\e}(b_\e))$ which is even in $z_3$, the problem \eqref{l.lin} has a unique solution $\phi\equiv L_{\e}(h)$ which is even in $z_3$, and
    \begin{align*}
        \|\phi\|_*\leq C\|h\|_{**},\quad |c_{j}|\leq C\|h\|_{**},
    \end{align*}
    and
    \begin{equation*}
    \|\nabla_{A'}\phi\|_*\leq C\|h\|_{**}.
    \end{equation*}
    %\textcolor{red}{Here
   % $$\partial_\Lambda\phi=\frac{1}{2|y-b_\e|}\phi+\frac{1}{|y-b_\e|^3}R_{\theta^*}^T\nabla\phi\cdot(y-b_\e)^T.$$}
\end{proposition}

\begin{proof}
We shall prove the proposition by contradiction, we mainly follow \cite[Proposition 4.2]{musso2016sign} to sketch the major steps. Suppose that there exists a sequence $K=K_n\to\infty$ such that there are functions $\phi_n$ and $h_n$ with $\|\phi_n\|_*=1$ and $\|h_n\|_{**}=o(1)$ such that
\begin{equation}
\label{l.linphi}
\begin{cases}
\Delta\phi_n+\lambda\e^2\phi_n+5PQ_{A'}^4\phi_n=h_n+\sum\limits_{j=0}^5c_{j}PQ_{A'}\hat Z_j\quad &\mbox{in}~B_{4r/\bar\e}(b_\e),\\
\phi_n=0\quad &\mbox{on}~\partial B_{4r/\bar\e}(b_\e),\\
\int_{B_{4r/\bar\e}(b_\e)}\phi_nPQ_{A'}^4\hat Z_jdx=0\quad\mbox{for}\quad j=0,\cdots,5,
\end{cases}
\end{equation}
for certain constants $c_j$, we shall prove that $\|\phi_n\|_*\to0$ to derive a contradiction.
\medskip

\noindent {\bf Step 1}. We first establish the following:
\begin{align*}
\|\phi_{n}\|_\mu:=\sup_{y\in B_{4r/\bar\e}(b_\e)}\left|\langle y-b_\e\rangle^{1-2\mu}\phi_{n}(y)\right|+\sup_{y\in B_{4r/\bar\e}(b_\e)}\left|\langle y-b_\e\rangle^{2-2\mu}D\phi_{n}(y)\right|\to0
\end{align*}
with $\mu>0$ being a small fixed number. We shall also prove this statement by contradiction. Without loss of generality, we may take $\|\phi_{n}\|_\mu=1$. Multiplying the equation against $\hat Z_\ell$, integrating by parts twice, we get that
\begin{equation*}
\begin{aligned}
\sum_{j=0}^5c_j\int_{B_{\frac{4r}{\bar\e}}(b_\e)}PQ_{A'}^4\hat Z_j\hat Z_\ell dy=\int_{B_{\frac{4r}{\bar\e}}(b_\e)}(\Delta \hat Z_\ell+\lambda \e_n^2\hat Z_\ell+5PQ_{A'}\hat Z_\ell)\phi_{n}dy-\int_{B_{\frac{4r}{\bar\e}}(b_\e)}h_{n}\hat Z_\ell dy.
\end{aligned}
\end{equation*}
% \begin{equation}
% \label{5.equ}
% \begin{aligned}
% \sum_{j=0}^5c_j\int_{B_{\frac{4r}{\bar\e}}(b_\e)}PQ_{A'}^4\hat Z_j\hat Z_\ell dy=~&\int_{B_{\frac{4r}{\bar\e}}(b_\e)}(\Delta \hat Z_\ell+\lambda \e_n^2\hat Z_\ell+5PQ_{A'}\hat Z_\ell)\phi_{n}dy\\
% &-\int_{B_{\frac{4r}{\bar\e}}(b_\e)}h_{n}\hat Z_\ell dy.
% \end{aligned}
% \end{equation}
By straightforward calculation one can verify that $\big(\int_{B_{{4r}/{\bar\e}}(b_\e)}PQ_{A'}^4\hat Z_j\hat Z_ldy\big)_{j\ell}$ is an invertible matrix. While, one can easily prove that the right-hand side of the above equation is bounded by
\begin{equation*}
o(1)\|\phi_n\|_\mu+C\|h_n\|_{**}.
\end{equation*}
Thus, we conclude that
\begin{equation}
|c_j|\leq C(\|h_n\|_{**}+o(1)\|\phi_n\|_\mu),\quad j=0,1,\cdots,5,
\end{equation}
so that $c_{j}=o(1).$ Then by Green's representation formula we can derive the following estimation
$$|\phi_n(y)|+\langle y-b_\e\rangle|D\phi_n(y)|\leq C(\|\phi_n\|_\mu+\|h_n\|_{**})\langle y-b_\e\rangle^{-1}.$$
In particular,
\begin{equation*}
\langle y-b_\e\rangle^{1-2\mu}|\phi_n(y)|\leq C\langle y-b_\e\rangle^{-2\mu}.
\end{equation*}
Since $\|\phi_n\|_\mu=1$, we assume that $\|\phi_n\|_{L^\infty(B_R(0))}>\gamma$ for certain $R>0$ and $\gamma>0$ independently of $\e$. Then local elliptic estimates and the bounds above yield that, up to a subsequence $\phi_n(y+b_\e)$ converges uniformly over compact sets of $\mathbb{R}^3$ to a nontrivial solution $\phi$ of
\begin{equation}
\label{5.equ-m}
\Delta\phi+5Q^4\phi=0,\quad |\phi(y)|\leq C|y|^{-1}.
\end{equation}
Due to the non-degenerate result in Proposition \ref{prop:non-degeneracy} and $h$ is even in $z_3$, we have that $\phi$ is a linear combination of the functions $Z_j,~j=0,\cdots,5$, defined in Section 2.3. On the other hand, by the dominated convergence theorem, we see that (after passing to a subsequence if necessary) the limit function $\phi$ is perpendicular to these kernels $Z_j,~j=0,1,\cdots,5$.
Hence the only possibility is that $\phi\equiv 0$, which is a contradiction. This yields the proof of $\|\phi_n\|_\mu\to0$. Moreover, we have
$$\|\phi_n\|_*\leq C(\|h_n\|_{**}+\|\phi_n\|_\mu),$$
hence $\|\phi_n\|_*\to0.$
\medskip

\noindent {\bf Step 2}.
In this step, we shall prove the existence of $\phi$ to \eqref{l.linphi} in the following function space
$$H=\left\{\phi\in H_0^1\left(B_{{4r}/{\bar\e}}(b_\e)\right)\left|\int_{B_{{4r}/{\bar\e}}(b_\e)}\phi PQ_{A'}^4\hat Z_j\phi dy=0,~\forall\, j=0,\cdots,5\right.\right\}$$
endowed with the usual inner product
$$\langle \phi,\psi\rangle= \int_{B_{{4r}/{\bar\e}}(b_\e)}\nabla\phi\nabla\psi dy.$$
Problem \eqref{l.linphi} expressed in weak form is equivalent to that of finding a $\phi\in H$ such that
$$\langle \phi,\psi\rangle =\int_{B_{{4r}/{\bar\e}}(b_\e)}(\lambda\e^2\phi+5PQ_{A'}^4\phi-h)\psi dy,\quad\forall \psi\in H.$$
Using the Riesz's representation theorem, the above equation could be rewritten in $H$ as the following form
\begin{equation}
\label{5.equ-riesz}
\phi=T_\e(\phi)+\tilde h
\end{equation}
with certain $\tilde h\in H$ which depends linearly in $h$ and where $T_\e$ is a compact operator in $H$. As a consequence of Fredholm's alternative principle and the fact that $\|\phi_n\|_*\to 0$ as $\|h_n\|_{**}\to0$, we conclude that for each $h$, problem \eqref{l.linphi} admits a unique solution, and the estimate follows easily.
\medskip

\noindent {\bf Step 3}. In this step we shall study the dependence of $\phi$ on $A'=(\Lambda,\e^{-1}a,\e^{-1}\xi,\e^{-1
}b,\beta)$ with $(\e,a,\xi,b,\beta)\in\mathbb{R}\times\mathbb{R}\times\Gamma\times\mathbb{R}^2\times\mathbb{R}$. Let us define $A'=(A'_0,\cdots,A'_5)$ the components of $A'$. We differentiate $\phi$ with respect to $A'_\ell$ and  set formally $Z_{\phi,\ell}=\frac{\partial}{\partial A'_{\ell}}\phi,~\ell=0,\cdots,5$.

We define $c_{\ell,j}$ so that
\begin{equation}
\label{5.rel}
\int_{B_{{4r}/{\bar\e}}(b_\e)}PQ_{A'}^4\hat Z_m\big(Z_{\phi,\ell}-\sum_{j=0}^5c_{\ell,j}\hat Z_j\big)dy=0,\quad \forall \ell, j.
\end{equation}
This amounts to solving a linear system in the constants $c_{\ell,j}$,
\begin{equation}
\label{5.coefficient}
\sum_{j}c_{\ell,j}\int_{B_{{4r}/{\bar\e}}(b_\e)}PQ_{A'}^4\hat Z_m\hat Z_j dy=\int_{B_{{4r}/{\bar\e}}(b_\e)}\phi\frac{\partial}{\partial A_\ell'}(PQ_{A'}^4\hat Z_m) dy,
\end{equation}
which follows by a direct differentiation with respect to $A_\ell'$ of the orthogonal conditions $\int_{B_{4r/\bar \e}(b_\e)}\phi PQ_{A'}^4\hat Z_mdx=0$. Arguing as before, we see that the above equation is uniquely solvable and that
$$c_{\ell,j}=O(\|\phi\|_*)$$
uniformly for parameters $A'$ in the considered region. Thus
\begin{equation*}\eta:=Z_{\phi,\ell}-\sum_{j=0}^5c_{\ell,j}\hat Z_j \in H_0^1\left(B_{{4r}/{\bar\e}}(b_\e)\right)\quad\mbox{and}\quad
\int_{B_{{4r}/{\bar\e}}(b_\e)}\eta PQ_{A'}^4\hat Z_j dx=0,\quad \forall j=0,\cdots,5.
\end{equation*}
On the other hand, a direct  but tedious computation shows that
\begin{equation}
\label{5.eta}
\Delta\eta+\lambda\e^2\eta+5PQ_{A'}^4\eta=f+\sum_{j=0}^5d_jPQ_{A'}^4\hat Z_j\quad\mbox{in}\quad B_{\frac{4r}{\bar\e}}(b_\e),
\end{equation}
where $d_j=\frac{\partial}{\partial A_\ell'}c_j$ and
\begin{equation}
\label{5.eta-f}
\begin{aligned}
f=-\sum_{j=0}^5c_{\ell,j}(\Delta +\lambda_n\e_n^2+5PQ_{A'}^4)\hat Z_j+\sum_{j=0}^5c_j\partial_{A_\ell'}(PQ_{A'}^4\hat Z_j)-5(\partial_{A_\ell'}PQ_{A'}^4)\phi.
\end{aligned}
\end{equation}
Thus, we have $\eta=L_\e(f)$. Consider the function $f$, we have
$$\left\|\phi\partial_{A_\ell'}(PQ_{A'}^4)\right\|_{**}\leq C\|\phi\|_*,\quad \left|\partial_{A_\ell'}(PQ_{A'}^4\hat Z_j)\right|\leq C(1+|y|^2)^{-3},$$
hence
$$\|c_j\partial_{A_\ell'}PQ_{A'}^4\hat Z_j\|_{**}\leq C\|h\|_{**},$$
which dues to that $c_i=O(\|h\|_{**})$. Lastly, by direct computation we have
\begin{align*}
\left|(\Delta +\lambda\e^2+5PQ_{A'}^4)\hat Z_j\right|
\leq C\e^2(1+|y|^2)^{-\frac12}\leq C\bar\e^{1-2\sigma}(1+|y|^2)^{-\frac32-\sigma},
\end{align*}
which leads to
$$\left\|\sum_{j=0}^5c_{\ell,j}(\Delta +\lambda\e^2+5PQ_{A'}^4)\hat Z_j\right\|_{**}\leq C\|h\|_{**}.$$
Thus we conclude that
$$\|f\|_{**}\leq C\|h\|_{**}.$$
Next, we define
$$Z_{\phi,\ell}=L_\e(f)+\sum_jc_{\ell,j}PQ_{A'}^4\hat Z_j,$$
with $c_{\ell,j}$ given by relations \eqref{5.coefficient} and $f$ given by \eqref{5.eta-f}, we check that indeed $Z_{\phi,\ell}=\partial_{A_\ell’}\phi$. In fact, $Z$ depends continuously on the parameters $A$ and $h$ with respect to the norm $\|\cdot\|_*$ and $\|Z_{\phi,\ell}\|_*\leq C\|h\|_{**}$ for parameters in the considered region. Hence, we finish the proof.
%In conclusion, we proved that $A'\to L_\e$ is of class $C^1$ from $\|\cdot\|_*$ to $\|\cdot\|_{**}$, and
%\begin{equation}
%\label{l.diff-def}
%(D_{A_\ell}L_\e)(h)=L_\e(f)+\sum_{j}c_{\ell,j}\hat Z_j.
%\end{equation}
\end{proof}

Next, we shall solve the nonlinear problem
\begin{equation}
\label{n.non-lin}
\begin{cases}
\begin{aligned}
\left(\Delta+\lambda\e^2
+5PQ_{A'}^4\right)\phi_{\rm in}=-\eta_r^{\bar\e}(\bar N_\e(\phi_{\rm in})+E)+\sum_{j=0}^5c_jPQ_{A'}^4\bar Z_j
\end{aligned}
~&\mbox{in}~ B_{\frac{4r}{\bar\e}}(b_\e),\\
\phi_{\rm in}=0  ~&\mbox{on}~ \partial  B_{\frac{4r}{\bar\e}}(b_\e),
\end{cases}
\end{equation}
where
\begin{equation*}
\bar N_\e(\phi_{\rm in})=N_\e(\phi_{\rm in}+\phi_{\rm out}(\phi_{\rm in}))+5PQ_{A'}^4\phi_{\rm out}(\phi_{\rm in}).
\end{equation*}

\begin{proposition}
\label{prop.der}
Suppose that the parameters $A$ and $\e$ satisfy the relation in \eqref{constraint-m}. Then there exists $K_0$ large enough such that for all $K\geq K_0$, there is a unique solution $\phi_{\rm in}(A')$ to problem \eqref{n.non-lin} with
\begin{equation}
\label{n.est-1}
\|\phi_{\rm in}\|_*\leq C\e^2K^2,\quad \mbox{and}\quad \|\nabla_{A'}\phi_{\rm in}\|_*\leq C\e^2K^2.
\end{equation}
\end{proposition}

\begin{proof}
Here we shall give the details for the estimation on $\phi$, For the derivation of the estimation for the derivative of $\phi$ with respect to the parameters $A'$, we refer the readers to \cite{del2003two,musso2016sign,wei2005arbitrary} for the details.

Consider the right-hand side of \eqref{n.non-lin}, we
observe that
\begin{align*}
E=Q_{A'}^5-PQ_{A'}^5=~&5\e^\frac12PQ_{A'}^4(T_A+\varphi_A+\psi_A)
+O(\e^4K^4)Q_{A'}^3+O(\e^{6}K^6)Q_{A'}^2\\
&+O(\e^8K^8)Q_{A'}+O(\e^{10}K^{10}).
\end{align*}
Then we have
$$\left\|\eta_r^{\bar\e}E\right\|_{**}\leq C\e^2K^2.$$
To estimate $\bar N_\e(\phi_{\rm in})$, we first see that
\begin{equation*}
\begin{aligned}
\left|\eta_r^{\bar\e}\bar N_\e(\phi_{\rm in})\right|\leq ~&C|PQ_{A'}|^4|\phi_{\rm out}|+C|PQ_{A'}|^3(|\phi_{\rm in}|^2+|\phi_{\rm out}|^2)\\
&+C|PQ_{A'}|^2(|\phi_{\rm in}|^3+|\phi_{\rm out}|^3)+C|PQ_{A'}|(|\phi_{\rm in}|^4+|\phi_{\rm out}|^4)\\
&+C(|\phi_{\rm in}|^5+|\phi_{\rm out}|^5).
\end{aligned}
\end{equation*}
Hence
$$\left\|\eta_r^{\e}\bar N_\e(\phi_{\rm in})\right\|_{**}\leq C\bar\e^4+C\bar\e^{2-\sigma}(\|\phi_{\rm in}\|_{*}^2+\|\phi_{\rm in}\|_{*}^3+\|\phi_{\rm in}\|_{*}^4+\|\phi_{\rm in}\|_{*}^5).$$

Next we shall show that equation \eqref{n.non-lin} has a unique solution $\phi_{\rm in}=\phi_{\rm in,0}+\phi_{\rm in,n}$ with
\begin{equation}
\label{n.def-in0}
\phi_{\rm in,0}=T_\e(\eta_r^{\bar\e}E),
\end{equation}
as the required properties. Here $T_\e$ denotes the linear operator defined in Proposition \ref{prl.1}. That is $T_\e(\eta_r^{\bar\e}E)=\phi_{\rm in,0}$. As a consequence, $\phi_{\rm in}=\phi_{\rm in,0}+\phi_{\rm in,n}$ is a solution to \eqref{n.non-lin} if and only if
\begin{equation}
\label{n.fixed}
\phi_{\rm in,n}=-T_\e(\bar N_\e(\phi_{\rm in,0}+\phi_{\rm in,n}))=\mathcal{M}_\e(\phi_{\rm in,n}).
\end{equation}
Then we need to show that the operator $\mathcal{M}_\e$ defined above is a contraction inside a property chosen region. Since $\|\eta_r^{\bar\e}E\|_{**}\leq C\e^{2}K^2$, the result of Proposition \ref{prl.1} gives that
$$\|T_\e(\eta_r^{\bar\e}E)\|_*\leq C\e^{2}K^2,$$
and
\begin{equation}
\label{n.higher}
%\|\phi_{\rm in,0}\|_*\leq C\e^2K^2\quad\mbox{and}\quad
\|\eta_r^{\e}\bar N_\e(\phi_{\rm in,0}+\phi_{\rm in,n})\|_{**}\leq C(\bar\e\e^2K^2+\bar\e\|\phi_{\rm in,n}\|_*+\bar\e^{2-\sigma}\|\phi_{\rm in,n}\|_*^2).
\end{equation}
We shall study the problem \eqref{n.fixed} in the following function space
\begin{equation}
\label{n.setf}
\mathscr{F}=\left\{f\in H_0^1(\Sigma_{K,\e}):\|f\|_*\leq C\e^2K^2\right\}.
\end{equation}
%Since $\|\eta_r^{\bar\e}E\|_{**}\leq C\e^2K^2$, the result of Proposition \ref{prl.1} gives that

From Proposition \ref{prl.1} and \eqref{n.higher} we conclude that, for $K$ sufficiently large and any $\bar\phi\in \mathscr{F}$ we have
$$\|\mathcal{M}_\e(\bar\phi)\|_*\leq C\e^2K^2.$$
If we choose $R$ large enough in the definition of $\mathscr{F}$, see \eqref{n.setf}, then we get that $\mathcal{M}_\e$ maps $\mathscr{F}$ into itself. Now we shall show that the map $\mathcal{M}_\e$ is a contraction, for any $K$ sufficiently large, and it will imply that $\mathcal{M}_\e$ has a unique fixed point in $\mathscr{F}$ and hence problem \eqref{n.non-lin} has a unique solution. For any $\bar\phi_1,~\bar\phi_2$ in $\mathscr{F}$ we have
$$\|\mathcal{M}_\e(\bar\phi_1)-\mathcal{M}_\e(\bar\phi_2)\|_*\leq C\| \bar N_\e(\phi_{\rm in,0}+\bar\phi_1)-N_\e(\phi_{\rm in,0}+\bar\phi_2)\|_{**},$$
now we just need to check that $\bar N_\e$ is a contraction in its corresponding norms. By definition of $\bar N_\e$, we have
\begin{equation*}
\begin{aligned}
&\left|N_\e(\phi_{\rm in,0}+\bar\phi_1+\phi_{\rm out}(\phi_{\rm in,0}+\bar\phi_1))-N_\e(\phi_{\rm in,0}+\bar\phi_2+\phi_{\rm out}(\phi_{\rm in,0}+\bar\phi_2))
\right|\\
&\leq C PQ_{A'}^3|\bar\phi||\bar\phi_1-\bar\phi_2|
\end{aligned}
\end{equation*}
for some $\bar\phi$ in the segment joining $\phi_{\rm in,0}+\bar\phi_1+\phi_{\rm out}(\phi_{\rm in,0}+\bar\phi_1)$ and $\phi_{\rm in,0}+\bar\phi_2+\phi_{\rm out}(\phi_{\rm in,0}+\bar\phi_2)$, and
\begin{equation*}
\begin{aligned}
\left|PQ_{A'}^4\phi_{\rm out}(\phi_{\rm in, 0}+\bar\phi_1)
-PQ_{A'}^4\phi_{\rm out}(\phi_{\rm in, 0}+\bar\phi_1)
\right|\leq C\bar\e PQ_{A'}^4|\bar\phi_1-\bar\phi_2|.
\end{aligned}
\end{equation*}
Therefore, we can conclude that there exists $c\in(0,1)$  such that
\begin{equation}
\label{n.con}
\|\bar N_{\e}(\phi_{\rm in,0}+\bar\phi_1)-\bar N_{\e}(\phi_{\rm in,0}+\bar\phi_2)\|_{**}\leq c\|\bar\phi_1-\bar\phi_2\|_*.
\end{equation}
This concludes the proof of existence of $\phi$ solution to \eqref{n.non-lin}, and the first estimate in \eqref{n.est-1}.
\end{proof}

%We have
%\[\int_{B_1}|\nabla \phi_A|^2-\la \phi_A^2=\int_{B_1}[(PQ_A+\phi_A)^5-Q_A^5]\phi_A\]

%We have
%\begin{align*}
%    J(PQ_A+\phi_A)&=J(PQ_A)+\int_{B_1}\nabla PQ_A\nabla \phi_A-\lambda PQ_A\phi_A+\frac{1}{2}\int_{B_1}|\nabla \phi_A|^2-\la \phi_A^2\\
%    &\quad -\frac{1}{6}\int_{B_1}(PQ_A+\phi_A)^6-(PQ_A)^6
%\end{align*}

%\begin{align*}
%    J(PQ_A+\phi_A)=J(PQ_A)-\frac{1}{6}\int_{B_1}(PQ_A+\phi_A)^6-(PQ_A)^6-6Q_A^5\phi_A+\frac{1}{2}\int_{B_1}|\nabla \phi_A|^2-\la \phi_A^2
%\end{align*}
%\begin{align*}
%    &=J(PQ_A)-\frac{1}{6}\int_{B_1}(PQ_A+\phi_A)^6-(PQ_A)^6-6Q_A^5\phi_A+\frac{1}{2}\int_{B_1}[(PQ_A+\phi_A)^5-Q_A^5]\phi_A\\
%    &=J(PQ_A)-\frac{1}{6}\int_{B_1}(PQ_A+\phi_A)^6-(PQ_A)^6-6(PQ_A)^5\phi_A\\
%    &\quad -\int_{B_1}((PQ_A)^5-Q_A^5)\phi_A+\frac{1}{2}\int_{B_1}[(PQ_A+\phi_A)^5-Q_A^5]\phi_A
%\end{align*}
%\vspace{1cm}
% \section{The finite-dimensional reduction}
% \begin{align}
% \begin{cases}
%     \Delta(PQ_A+\phi_A)+\la(PQ_A+\phi_A)+(PQ_A+\phi_A)^5=\sum_{j=1}^6 c_j PQ_A^4Z_j\\
%     \phi_A=0&\text{on }B_j\\
%     \int_{B_1} \phi_APQ_A^4 Z_j=0
% \end{cases}
% \end{align}
% \[\Delta \phi_A+\lambda \phi_A+(PQ_A+\phi_A)^5-Q_A^5=\sum_{j=1}^6 c_j PQ_A^4Z_j\]
% \begin{align*}
%     \Delta \phi_A+\lambda \phi_A+5PQ_A^4\phi_A=[Q_A^5-PQ_A^5]-[(PQ_A+\phi_A)^5-PQ_A^5-5PQ_A^4\phi_A]+\sum_{j=1}^6 c_j PQ_A^4Z_j
% \end{align*}
% \begin{align*}
%     E(PQ_A)=[Q_A^5-PQ_A^5]
% \end{align*}

\section{The finite-dimensional reduction and the critical point}
In this section, we will first set up the reduction that transforms the original infinite-dimensional problem into a finite-dimensional one. Then, we will determine the critical points of the energy with respect to the parameters, using these to establish the proof of Theorem \ref{th1.sector}.

% Let $u_\e(y)=PQ_{A'}(y)+\phi_{A'}(y)$ be the solution to \eqref{g.eq}. Returning to the original problem (before scaling), we write it as $\e^{-\frac12}u_\e=PQ_A+\phi_A(z)$. The first conclusion is the following reduction lemma.

 Suppose that $\phi_{A'}$ is the solution of \eqref{n.non-lin}. Let $u_\e(y)=PQ_{A'}(y)+\phi_{A'}(y)$, then
\begin{align*}
    \Delta_y u_\e (y)+\la \e^2 u_\e{y}+u_{\e}(y)^5=\sum_{j=0}^{5}c_jPQ_{A'}^4\bar Z_j.
\end{align*}
Note that $u_\e$ will satisfy \eqref{g.eq} if the Lagrange multiplier $c_j=0$ for $j=0,1,\cdots,5$. The following reduction Lemma says that this is equivalent to the criticality of $A'$ in a finite-dimensional space. Returning to the original variable $z$ and $A$ (before scaling), denoting $PQ_A(z)=\e^{-1/2}PQ_{A'}(z/\e)$ and $\phi_A(z)=\e^{-1/2}\phi_{A'}(z/\e)$, then $PQ_A(z)+\phi_A(z)$ will be a solution to the Brezis-Nirenberg problem under the criticality of $A$.
\begin{lemma}
\label{le.finite-p}
$u_\e=PQ_{A'}+\phi_{A'}(y)$ is a solution of problem \eqref{g.eq} if and only if $A'$ is a critical point of the energy
\begin{equation*}
\begin{aligned}
J_\e(u_\e):=\frac{1}2\int_{\Sigma_{K,\e}}|\nabla_yu_\e|^2dy-\frac{\lambda\e^2}{2}\int_{\Sigma_{K,\e}}u_\e^2dy-\frac16\int_{\Sigma_{K,\e}}u_\e^6dy.
\end{aligned}
\end{equation*}
Equivalently, $PQ_A(z)+\phi_A(z)$ is a solution to \eqref{210} if and only if $A$ is a critical point of the energy $J(PQ_A+\phi_A)=J_\e(u_\e)$ where
\begin{align*}
    J(PQ_A+\phi_A):=&~\frac{1}2\int_{\Sigma_{K}}|\nabla_z(PQ_{A}+\phi_A)|^2dz
-\frac\lambda2\int_{\Sigma_K}(PQ_A+\phi_A)^2dz\\
&-\frac16\int_{\Sigma_K}(PQ_A+\phi_A)^6dz.
\end{align*}
\end{lemma}
% \begin{lemma}
% \label{le.finite-p}
% $u_\e=PQ_{A'}+\phi_{A'}(y)$ is a solution of problem \eqref{g.eq} if and only if $A'$ is a critical point of the energy
% \begin{equation*}
% \begin{aligned}
% I=~&\frac{1}2\int_{\Sigma_{K}}|\nabla_z(PQ_{A}+\phi_A)|^2dz
% -\frac\lambda2\int_{\Sigma_K}(PQ_A+\phi_A)^2dz-\frac16\int_{\Sigma_K}(PQ_A+\phi_A)^6dz\\ =~&J_\e(u_\e)=\frac{1}2\int_{\Sigma_{K,\e}}|\nabla_yu_\e|^2dy-\frac{\lambda\e^2}{2}\int_{\Sigma_{K,\e}}u_\e^2dy-\frac16\int_{\Sigma_{K,\e}}u_\e^6dy.
% \end{aligned}
% \end{equation*}
% \end{lemma}

\begin{proof}
Let $A_\ell',~\ell=0,\cdots,5$ be the elements of $A'$. Considering the derivative of $I$ with respect to $A_\ell'$, we see that $\frac{\partial I}{\partial A_\ell'}=0$ is equivalent to say that $\frac{\partial J_\e(PQ_{A'}+\phi_{A'})}{\partial A_\ell'}=0$. Next we compute
$$\frac{\partial }{\partial A_\ell'}J_\e(PQ_{A'}+\phi_{A'})=DJ_\e(PQ_{A'}+\phi_{A'})\left[\frac{\partial}{\partial A_\ell'}PQ_{A'}+\frac{\partial }{\partial A_\ell'}\phi_{A'}\right].$$
On the other hand, one sees that
$$\frac{\partial PQ_{A'}}{\partial A_\ell'}=\hat Z_2(y)+o(1).$$
Using Proposition \ref{prop.der} and estimates \eqref{g.existence}-\eqref{g.der.lip}, one can show
$$\left\|\frac{\partial \phi_{A'}}{\partial A_\ell'}\right\|_{**}=O(\e^2 K^2)\quad\mbox{as}\quad K\to\infty,$$
then we get that $\frac{\partial I}{\partial A_\ell'}=0$ is equivalent to say $DJ_\e(PQ_{A'}+\phi_{A'})[\hat Z_{\ell}+o(1)]=0$. As a consequence, we can draw a conclusion that $\nabla_A I=0$ is equivalent to
\begin{equation}
\label{reduce.eq-1}
DJ_\e(PQ_{A'}+\phi_{A'})\left[\hat Z_\ell+o(1)\right]=0,\quad \forall \ell=0,\cdots,5.
\end{equation}
From the fact that $DJ_\e(PQ_{A'}+\phi_{A'})[g]=0$ for all functions such that $\int_{\Sigma_{K,\e}}PQ_{A'}^4\hat Z_\ell gdy=0$, we can see that \eqref{reduce.eq-1} can be written as
\begin{equation}
\label{reduce.eq-2}
DJ_\e(PQ_{A'}+\phi_{A'})\left[\hat Z_\ell+o(1)\Xi\right]=0,\quad \forall \ell=0,\cdots,5,
\end{equation}
where $\Xi$ is a uniformly bounded function, that belongs to the vector space generated by the functions $\hat Z_\ell.$ Then $\nabla I(A')=0$ is equivalent to
$$DJ_\e(PQ_{A'}+\phi_{A'})[\hat Z_\ell]=0,\quad \forall \ell=0,\cdots,5.$$
By definition of $c_\ell$ in \eqref{n.non-lin}, then we readily derive that this is equivalent to $c_\ell=0$ for all $\ell$ and it finishes the proof.
\end{proof}

%\begin{align}
%\begin{cases}
%    \Delta(PQ_A+\phi_A)+\la(PQ_A+\phi_A)+(PQ_A+\phi_A)^5=\sum_{j=1}^6 c_j PQ_A^4Z_j\\
%    \phi_A=0&\text{on }\Sigma_K\\
%    \int_{B_1} \phi_APQ_A^4 Z_j=0
%\end{cases}
%\end{align}
%\[\Delta \phi_A+\lambda \phi_A+(PQ_A+\phi_A)^5-Q_A^5=\sum_{j=1}^6 c_j PQ_A^4Z_j\]
%\begin{align*}
%    \Delta \phi_A+\lambda \phi_A+5Q_A^4\phi_A=-[(PQ_A+\phi_A)^5-Q_A^5-5Q_A^4\phi_A]+\sum_{j=1}^6 c_j PQ_A^4Z_j
%\end{align*}
%Recall that $PQ_A=Q_A-T_A-\varphi_A^e-\psi_A$,
%\begin{align*}
%    [(PQ_A+\phi_A)^5-Q_A^5-5Q_A^4\phi_A]=-5Q_A^4(T_A+\varphi_A^e+\psi_A)
%\end{align*}

%\section{The reduced energy and its critical point}

In the following calculation, we shall see the major in the expansion of $J(PQ_A+\phi_A)$ is $J(PQ_A)$.
\begin{align*}
    J(PQ_A+\phi_A)=~&J_\e(PQ_{A'}+\phi_{A'})\\
    =~&\frac{1}2\int_{\Sigma_{K,\e}}|\nabla_y(PQ_{A'}+\phi_{A'})|^2dy
-\frac{\lambda\e^2}2\int_{\Sigma_{K,\e}}(PQ_{A'}+\phi_{A'})^2dy\\
&-\frac16\int_{\Sigma_{K,\e}}(PQ_{A'}+\phi_{A'})^6dy.
\end{align*}
By expanding all terms and grouping them, we get
\begin{equation}
\label{6.1}
\begin{aligned}
J_\e(PQ_{A'}+\phi_{A'})=~&J_\e(PQ_{A'})-\int_{\Sigma_{K,\e}}(\Delta PQ_{A'}+\lambda\e^2PQ_{A'}+PQ_{A'}^5)\phi_{A'}dy\\
&-\frac{1}{2}\int_{\Sigma_{K,\e}}(\Delta\phi_{A'}+\lambda\e^2\phi_{A'}+5PQ_{A'}^4\phi_{A'})\phi_{A'}dy \\
&-\frac{1}{6}\int_{\Sigma_{K,\e}}(20PQ_{A'}^3\phi_{A'}^3+15PQ_{A'}^2\phi_{A'}^4+6PQ_{A'}\phi_{A'}^5+\phi_{A'}^6)dy .
\end{aligned}
\end{equation}
It is known that $\|\phi_A\|_*\leq C\e^2K^2$, one can easily show that
\begin{equation*}
\left|\int_{\Sigma_{K,\e}}(20PQ_{A'}^3\phi_{A'}^3+15PQ_{A'}^2\phi_{A'}^4+6PQ_{A'}\phi_{A'}^5+\phi_{A'}^6) dy\right|\leq C\e^{6}K^6.
\end{equation*}
Consider the second term on the right-hand side of \eqref{6.1}, we have
\begin{equation}
\begin{aligned}
\Delta PQ_{A'}+\lambda\e^2 PQ_{A'}+PQ_{A'}^5=E(PQ_{A'}),
\end{aligned}
\end{equation}
where
\begin{align*}
    E(PQ_{A'})=PQ_{A'}^5-Q_{A'}^5=&-5\e^\frac12Q_{A'}^4(T_A+\varphi_A+\psi_A)+O(\e^4K^4)Q_{A'}^3+O(\e^6K^6)Q_{A'}^2\\
    &+O(\e^8K^8)Q_{A'}+O(\e^{10}K^{10}).
\end{align*}
Then
\begin{equation}
\begin{aligned}
&\int_{\Sigma_{K,\e}}(\Delta PQ_{A'}+\lambda\e^2 PQ_{A'}+PQ_{A'}^5)\phi_{A'}dy\\
&=-5\int_{\Sigma_{K,\e}}\e^\frac12Q_{A'}^4(T_{A}+\varphi_A+\psi_A)\phi_{A'}dy+O(\e^6K^6)\\
&=O(\e^4K^4)\int_{\Sigma_{K,\e}}Q_{A'}^4dy +O(\e^6K^6)=O(\e^4K^4).
\end{aligned}
\end{equation}
While for the third term on the right hand side of \eqref{6.1},
\begin{equation}
\begin{aligned}
\Delta \phi_{A'}+\lambda\e^2\phi_{A'}+5PQ_{A'}^4\phi_{A'}=E(PQ_{A'})+N(\phi_{A'})+\sum_{j=0}^5c_jPQ_{A'}^4\hat Z_j.
\end{aligned}
\end{equation}
As the above computation,
\begin{equation}
\begin{aligned}
\int_{\Sigma_{K,\e}}E(PQ_{A'})\phi_{A'}dy=O(\e^4K^4).
\end{aligned}
\end{equation}
While for the higher order term $N(\phi_{A'})$, it is easy to see that
\begin{equation}
\int_{\Sigma_{K,\e}}N(\phi_{A'})\phi_{A'}dy=O(\e^6K^6).
\end{equation}
The multiplicative of $c_jPQ_A^4\hat Z_j$ and $\phi_{A'}$ is obvious zero due to the setting of $\phi_{A'}$. Therefore, we conclude that
\begin{equation}
\int_{\Sigma_{K,\e}}(\Delta \phi_{A'}+\lambda\e^2\phi_{A'}+5PQ_{A'}^4\phi_{A'})\phi_{A'}dy=O(\e^4K^4).
\end{equation}
Thus we conclude that
\begin{equation}
J(PQ_A+\phi_A)=J_\e(PQ_{A'}+\phi_{A'})=J_\e(PQ_{A'})+O(\e^4K^4)
=J(PQ_A)+O(\e^4K^4).
\end{equation}
\smallskip

\begin{theorem}
\label{th8.para}There exists $\delta$ small such that for $K$ large enough,
    the $\inf_{A\in \mathscr{C}}J(PQ_A+\phi_A)$ is achieved in the interior of the set $\mathscr{C}$ defined by \eqref{constraint-m}, i.e.
    \begin{align}
\begin{split}\label{constraint-C-8}
    &\e K^{3}\in [\delta,\delta^{-1}],\quad \xi\in \Gamma, \quad |a|\leq \delta^{-1}\e \log K,\quad d\in \left[ \tfrac{\log K-\log\log K}{K},\tfrac{\log K}{K}\right],\\
    & |\alpha_b|\leq \delta^{-1/2} K^{-2}\log K,\quad  |\hat\beta|\leq  \delta^{-1/2}K^{-1}\log K.
\end{split}
\end{align}
\end{theorem}
\begin{proof}
     %Since the constraint set $\mathscr{C}$ is closed, the infimum of $J(PQ_A+\phi_A)$ over $A_t \in \mathscr{C}$ is attained at some point
     %$A_t = (\e, \xi, a, d,\alpha_b, \hat\beta) \in \mathscr{C}$. It should be noted that $A_t$ and $A=(\e,\xi,a,b,\hat\beta)$ uniquely determine each other.
     Since the constraint set $\mathscr{C}$ is closed, the infimum of $J(PQ_A+\phi_A)$ over $A \in \mathscr{C}$ is attained at some point
     $A = (\e, \xi, a, d,\alpha_b, \hat\beta) \in \mathscr{C}$. Note that since we denote $b=|b|e^{i\alpha_b}$ and $d=(1-|b|^2)/b$,  then it is equivalent to write $A=(\e, \xi, a, d,\alpha_b, \hat\beta)$ and $A=(\e,\xi,a,b,\hat\beta)$. Recall that our notation convention \eqref{nt-wW} implies $\alpha_w=\hat \beta$.

     We will prove that for each point on the boundary $\partial \mathscr{C}$ there is another interior point whose value is strictly smaller than that. Thus, the infimum must be achieved in the interior.
    First, let us recall that
    \begin{align}\label{J-full}
        J(PQ_A+\phi_A)=J(PQ_A)+O(\e^4K^4)=\frac13\int_{\R^3}q^6dz+2\pi \Psi(A)+O(\e^3K
        (\log K)^2).
    \end{align}
    Notice that $\frac13\int_{\R^3}q^6$ is a constant and does not depend on $A$. We need to study $\Psi(A)$  where
    \begin{align*}
    \Psi(A)=~&\e [q(\hat \xi)]^2\mathcal{H}(b,b)+\e^2 q(\hat \xi)w\cdot [\nabla_z\mathcal{H}(b,b)+\nabla_p\mathcal{H}(b,b)]\\
    &+\e^3w^T\nabla^2_{z,p}\mathcal{H}(b,b)w-\lambda \e^2C_*(\xi).
    \end{align*}
It is essential to determine the exact order of each term involving $\mathcal{H}$. To maintain the flow of the proof, we defer this analysis to the following sections and list the results here. First, it follows from Lemma \ref{lem:gamma(b,b)} and Lemma \ref{lem:H0e(b,b)} that
    \begin{align*}
        \mathcal{H}(b,b)=\gamma(b,b)+H_0^e(b,b)=\frac{1}{2|b|}[\hat S_1(K)+S_1(K,d)+O(K^3\alpha_b^2)].
    \end{align*}
    For later references, we denote the ``leading term" as $\mathcal{C}_0:=\hat S_1(K)+S_1(K,d)$. The order of $\mathcal{C}_0$ can be obtained by Lemma \ref{lem:csc3csc5} and Lemma \ref{lem:alt.sum-d} in Section \ref{sec:series}.
    \begin{align*}
    \mathcal{C}_0(K,d)=\hat S_1(K)+S_1(K,d)=\frac{K}{\pi}\log 4+O((\log K)^{1/2}).
\end{align*}
    Second, by Lemma \ref{lem:w_dot_z} and Lemma \ref{lem:w_dot_zH}, we have
    \begin{align*}
    &w\cdot \nabla_z\mathcal{H}(b,b)=w\cdot \nabla_z\gamma(b,b)+w\cdot \nabla_zH_0^e(b,b)=O(K(\log K)^{1/2}).
    %&=\frac{|w|}{4|b|^2}[-\hat S_1(K)-S_1(K,d)+d\sqrt{1+d^2}S_3(K,d)+O(Kd^{-1}\alpha_w^2+Kd^{-3}\alpha_b^2)].
    \end{align*}
    Third, using Lemma \ref{lem:wTn2w-g} and Lemma \ref{lem:wTn2w-h},
    \begin{align*}
        &w^T\nabla^2_{z,p}\mathcal{H}(b,b)=\frac{|w|^2}{8|b|^3}\left[\mathcal{C}_2(K,d)+(\alpha_w,\alpha_b)\mathcal{A}^\gamma (\alpha_w,\alpha_b)^T+O(Kd^{-2}\alpha_w^2+Kd^{-4}\alpha_b^2)\right],
    \end{align*}
    where
    \begin{align*}
        \mathcal{C}_2(K,d)=\hat S_1(K)+\hat S_3(K)+S_1(K,d)-(d+\sqrt{1+d^2})^2S_3(K,d)+3(d^2+d^4)S_5(K,d),
    \end{align*}
    \begin{align*}
    \mathcal{A}^\gamma
    &=\begin{pmatrix}
        S_3^o(K)-2S_1^o(K)+3\hat S_3^e(K)& -3\hat S_3(K)+3S_1^o(K)\\
        -3\hat S_3(K)+3S_1^o(K)& \frac32[4S_5^o(K)+S_3^o(K)-3S_1^o(K)]+3\hat S_3^e(K)
    \end{pmatrix}.
\end{align*}
    Again, Lemma \ref{lem:csc3csc5}, Lemma \ref{lem:rough-Sk} and Lemma \ref{lem:alt.sum-d} imply that
    \begin{equation}
    \label{mathC2-2}
    \begin{aligned}
        \mathcal{C}_2(K,d)&=\frac{3\zeta(3)}{2\pi^3}K^3+O(K\log K)+\sqrt{\frac{8}{\pi}}K^3(Kd)^{-\frac12}e^{-Kd}(1+O((Kd)^{-1})\\
        &=\frac{3\zeta(3)}{2\pi^3}K^3+O(K^2(\log K)^{1/2}),
    \end{aligned}
    \end{equation}
    \begin{align}\label{Ag-2}
        \mathcal{A}^\gamma=\begin{pmatrix}
        \frac{5\zeta(3)}{2\pi^3}K^3+O(K\log K)& -\frac{6\zeta(3)}{\pi^3}K^{3}+O(K\log K)\\
        -\frac{6\zeta(3)}{\pi^3}K^{3}+O(K\log K)& \frac{93\zeta(5)}{8\pi^5}K^{5}+O(K^3)
    \end{pmatrix},
    \end{align}
where $\zeta(\cdot)$ is the Riemann zeta function.

Now, we decompose
\[\Psi(A)=\Psi(A_0)+\Psi(A)-\Psi(A_0)\]
where $A_0=(\e,\xi,a,b_0,0)$ and $b_0=(|b|,0,0)$.
% Collecting all the leading terms involving $\mathcal{H}$ and its derivatives, we find that they form $\Psi(A_0)$, where $A_0=(\e,\xi,a,b_0,0)$ where $b_0=(|b|,0,0)$. More precisely,
% \begin{align}\label{PsiA0-all}
%     \Psi(A_0)=\e [q(\hat \xi)]^2\frac{\mathcal{C}_0(K,d)}{2|b|}+\e^2 q(\hat \xi)\frac{|w|}{2|b|^2}\mathcal{C}_1(K,d)+\e^3\frac{|w|^2}{8|b|^3}\mathcal{C}_2(K,d)-\la \e^2 C_*(\xi).
% \end{align}
\begin{align}\label{PsiA0}
    \Psi(A_0)=\e [q(\hat \xi)]^2\frac{\mathcal{C}_0(K,d)}{2|b|}+\frac{|w|^2}{8 |b|^3}\e^3\mathcal{C}_2(K,d)-\lambda \e^2 C_*(\xi)+O(\delta^{-1}\e^3K (\log K)^2),
\end{align}
\begin{align}
\begin{split}\label{PsiA-PsiA0-all}
    &\Psi(A)-\Psi(A_0)=\e^3 (\alpha_w,\alpha_b)\mathcal{A}^\gamma (\alpha_w,\alpha_b)^T\\
    &\quad +O\left(\e [q(\hat\xi)]^2K^3 \alpha_b^2+\e^2 q(\hat \xi)Kd^{-1}(\alpha_w^2+d^{-2}\alpha_b^2))+\e^3Kd^{-2}(\alpha_w^2+d^{-2}\alpha_b^2)\right).
    \end{split}
\end{align}
%It is convex on $\alpha_w,\alpha_b$.
Using the explicit orders of $\mathcal{A}^\gamma$ and the bounds in $\mathscr{C}$, we have rough estimates
% \begin{align}\label{PsiA0}
%     \Psi(A_0)=\e [q(\hat \xi)]^2\frac{\mathcal{C}_0(K,d)}{2|b|}+\frac{|w|^2}{8 |b|^3}\e^3\mathcal{C}_2(K,d)-\lambda \e^2 C_*(\xi)+O(\delta^{-1}\e^3K \log K),
% \end{align}
\begin{align}\label{PsiA-PsiA0}
     |\Psi(A)-\Psi(A_0)|\leq C\delta^{-1}\e^3K(\log K)^2.
\end{align}

In the following, we will prove that if $\delta$ is sufficiently small and $K$ sufficiently large then any point on the boundary $\partial \mathscr{C}$ is strictly greater than some interior point. Since $\Gamma$ is a closed curve, $\xi$ is always in the interior. It remains to consider the variation of $J(PA_A+\phi_A)$ for $\e, a,d,\alpha_b,\alpha_w$. We denote $C_q=\frac13\int_{\R^3}q^6dz$ for short.

(1) Consider the variation of $\e$, and  fix all the other variables in $A$. Thus one can think of $J(PQ_A+\phi_A)=\mathcal{J}_1(\e)$ as a function of $\e$ on $[\delta K^{-3},\delta^{-1}K^{-3}]$. If $\e =\delta K^{-3}$, then using \eqref{J-full}, \eqref{PsiA0}, \eqref{PsiA-PsiA0}, and the definition of $\mathscr{C}$ in \eqref{constraint-C-8},
    \begin{align*}
        \mathcal{J}_1(\delta K^{-3})&\geq C_q -2\pi \lambda \delta^2K^{-6} C_*(\xi)+O(\delta^2 K^{-8}(\log K)^2).
    \end{align*}

    If $\e= \delta^{-1}K^{-3}$, then the same estimates yields
    \begin{align*}
        \mathcal{J}_1(\delta^{-1}K^{-3})\geq C_q+2\pi\left(\frac{|w|^2}{8 |b|^3}\delta^{-3}K^{-6}\mathcal{C}_2(K,d)-\lambda \delta^{-2}K^{-6}C_*(\xi)\right)+O(\delta^{-4} K^{-8}(\log K)^2).
    \end{align*}
    However, one can choose another $\e_*=\delta_*K^{-3}$ with
    $\delta_*={16 |b|^3\lambda C_*(\xi)}K^3/(3|w|^2\mathcal{C}_2(K,d))$ and compute its energy. Notice that $\e_*[q(\hat \xi)]^2\mathcal{C}_0(K,d)/(2|b|)\leq C\delta^{-2}(\e_*)^3K(\log K)^2$, then
    \begin{align*}
        \mathcal{J}_1(\delta_*K^{-3})&\leq C_q+2\pi\left(\frac23\lambda(\e_*)^2C_*(\xi)-\lambda (\e_*)^2 C_*(\xi)\right)+O(\delta^{-4}K^{-8}(\log K)^2)\\
        &\leq C_q-\frac{\pi}{2}\lambda (\e_*)^2C_*(\xi)=C_q-\frac{\pi}2\lambda \delta_*^2C_*(\xi)K^{-6}.
    \end{align*}
    Comparing the order $K^{-6}$ of three cases, one can choose a sufficiently small $\delta$ satisfying $\delta<\delta_*<\delta^{-1}$ and sufficiently large $K$ such that
    \[\mathcal{J}_1(\delta_*K^{-3})<\min\{\mathcal{J}_1(\delta K^{-3}),\mathcal{J}_1(\delta^{-1}K^{-3})\}.\]

    (2) Consider the variation of $a$, and  fix all the other variables in $A$. Thus one can think of $J(PQ_A+\phi_A)=\mathcal{J}_2(|a|)$ as a function of $|a|$ on $[0,\delta^{-1}\e \log K]$.  If $|a|=\delta^{-1}\e \log K$. Plugging in $q(\hat \xi)=q(\xi+a\nu(\xi))=a|\nabla q(\xi)|+O(a^2)$ and
    \begin{align*}
        %q(\hat \xi)&=q(\xi+a\nu(\xi))=a|\nabla q(\xi)|+O(a^2)\\
        |w|^2&=|\nabla q(\hat \xi)|^2=|\nabla q(\xi)|^2+2a |\nabla q(\xi)|\nu(\xi)^T (\nabla^2 q(\xi))\nu(\xi)+O(a^2)
    \end{align*}
    to \eqref{PsiA0} yields
    \begin{align*}
        \mathcal{J}_2(\delta^{-1}\e(\log K)^{1/2})=&\  C_q+2\pi\left(\frac{|\nabla q(\xi)|^2\e^3 K\log K\log4}{2\pi|b|\delta^2} +\frac{|\nabla q(\xi)|^2}{8|b|^3}\e^3\mathcal{C}_2(K,d)-\lambda \e^2 C_*(\xi)\right)\\
        &+O(\delta^{-1}\e^3K(\log K)^2).
    \end{align*}
    However, if we choose $a=0$ then
    \begin{align*}
        \mathcal{J}_2(0)\leq C_q+2\pi\left(\frac{|\nabla q(\xi)|^2}{8|b|^3}\e^3\mathcal{C}_2(K,d)-\lambda \e^2 C_*(\xi)\right)+O(\delta^{-1}\e^3K(\log K)^2).
    \end{align*}
    Taking $K$ large enough and $\delta$ small enough, we have \[\mathcal{J}_2(0)<\mathcal{J}_2(\delta^{-1}\e \log K).\]

    (3) Consider the variation of $d$, and  fix all the other variables in $A$. Thus one can think of $J(PQ_A+\phi_A)=\mathcal{J}_3(d)$ as a function of $d$ on $[\frac{\log K-\log\log K}{K},\frac{\log K}{K}]$. In this case, we use \eqref{PsiA0}, \eqref{PsiA-PsiA0}, and the estimate of $\mathcal{C}_0(K,d)$ to get
    \begin{align}\label{PsiA0-3}
        \Psi(A_0)=\e^3\frac{|w|^2}{8}\frac{\mathcal{C}_2(K,d)}{|b|^3}-\lambda \e^2C_*(\xi)+O(\delta^{-2}\e^3K(\log K)^2).
    \end{align}
    Note that $|b|=\sqrt{1+d^2}-d=1-d+O(d^2)$. Using the definition of $\mathcal{C}_2(K,d)$ and the estimate of $S_i$ and $S_i(K,d)$ in the Lemma \ref{lem:csc3csc5} and Lemma \ref{lem:alt.sum-d}, we have
\begin{align*}
    \frac{\mathcal{C}_2(K,d)}{|b|^3}&=K^3\left[\frac{3\zeta(3)}{2\pi^3}+(Kd)^{-\frac12}e^{-Kd}+O((Kd)^{-\frac32}e^{-Kd})\right](1+3d+O(d^2))\\
   &=K^3\left[\frac{3\zeta(3)}{2\pi^3}+\frac{9\zeta(3)}{2\pi^3}d+(Kd)^{-\frac12}e^{-Kd}+O(d^2+(Kd)^{-\frac32}e^{-Kd})\right].
   %& =\frac{3\zeta(2)}{2\pi^3}K^3[1+3\frac{\log K-t\log\log K}{K}+(\log K-t\log\log K)^{-\frac12}\frac{(\log K)^t}{K}+O(\frac{(\log K)^{t-\frac32}}{K})]
\end{align*}
% \begin{align*}
%      =K^3\left[\frac{3\zeta(2)}{2\pi^3}+\frac{9\zeta(2)}{2\pi^3}\frac{\log K-t\log\log K}{K}+(\log K-t\log\log K)^{-\frac12}\frac{(\log K)^t}{K}+O(\frac{(\log K)^{t-\frac32}}{K})\right]
% \end{align*}
If $d=\frac{\log K}{K}$, then
\begin{align*}
   \frac{\mathcal{C}_2(K,d)}{|b|^3}=K^3\left[\frac{3\zeta(3)}{2\pi^3}+\frac{9\zeta(3)}{2\pi^3}\frac{\log K}{K}+\frac{(\log K)^{-\frac12}}{K}+O\left(\frac{(\log K)^{-\frac32}}{K}\right)\right] .
\end{align*}
If $d=\frac{\log K-\log \log K}{K}$, then
\begin{align*}
 \frac{\mathcal{C}_2(K,d)}{|b|^3}=K^3\left[\frac{3\zeta(3)}{2\pi^3}+\frac{9\zeta(3)}{2\pi^3}\frac{\log K}{K}+\frac{(\log K)^{\frac12}}{K}+O\left(\frac{\log\log K}{K}\right)\right] .
\end{align*}
If $d=\frac{1}{K}[\log K-\frac12\log\log K]$, then
\begin{align*}
 \frac{\mathcal{C}_2(K,d)}{|b|^3}=K^3\left[\frac{3\zeta(3)}{2\pi^3}+\frac{9\zeta(3)}{2\pi^3}\frac{\log K-\frac12\log\log K}{K}+O\left(\frac{1}{K
 }\right)\right].
\end{align*}
Obviously, the third one is strictly less than the first two at least by a term of the order $\e^3K^2\log\log K$ when $K$ is large enough. Consequently, \eqref{PsiA0-3} implies that
% \begin{align*}
%  \frac{\mathcal{C}_2(K,d)}{|b|^3}<\min\{\}
% \end{align*}
\begin{align*}
    \mathcal{J}_3(\tfrac{\log K-\frac12 \log\log K}{K})<\min\{\mathcal{J}_3(\tfrac{\log K}{K}),\mathcal{J}_3(\tfrac{\log K-\log\log K}{K})\}.
\end{align*}

(4) Consider the variation of $\alpha_b$ and $\alpha_w$, and fix all the other variables in $A$.  Recall that our definition of $\theta_*(\xi)$ makes $\alpha_w=\hat \beta$. Denote $\hat \alpha_b=K \alpha_b$. Thus the constraint $\mathscr{C}$ indicates that $|\alpha_w|\leq \delta^{-1/2}K^{-1}\log K$ and $|\hat \alpha_b|\leq \delta^{-1/2}K^{-1}\log K$. In this case, $\Psi(A_0)$ is fixed. We think of $J(PQ_A+\phi_A)=\mathcal{J}_4(\alpha_w,\hat \alpha_b)$ as a function of $\alpha_w,\hat \alpha_b$ on the square
\[\left[-\frac{\log K}{K\delta^{1/2}},\frac{\log K}{K\delta^{1/2}}\right]\times \left[-\frac{\log K}{K\delta^{1/2}},\frac{\log K}{K\delta^{1/2}}\right].\]
Then \eqref{PsiA-PsiA0-all} can be simplified
\begin{align}
\begin{split}
    &\Psi(A)-\Psi(A_0)=\e^3 (\alpha_w,\hat \alpha_b)\mathcal{\hat A}^\gamma (\alpha_w,\hat\alpha_b)^T+O\left(\frac{\e^3 K^3}{\log K}(\alpha_w^2+\hat\alpha_b^2)\right),
    %&\quad O\left(\e [q(\hat\xi)]^2K \hat \alpha_b^2+\e^2 q(\hat \xi)Kd^{-1}(\alpha_w^2+\hat \alpha_b^2))+\e^3Kd^{-2}(\alpha_w^2+\hat \alpha_b^2)\right).
    \end{split}
\end{align}
where
\begin{align*}
        \mathcal{\hat A}^\gamma=\begin{pmatrix}
        \frac{5\zeta(3)}{2\pi^3}K^3+O(K\log K)& -\frac{6\zeta(3)}{\pi^3}K^{2}+O(\log K)\\
        -\frac{6\zeta(3)}{\pi^3}K^{2}+O(\log K)& \frac{93\zeta(5)}{8\pi^5}K^{3}+O(K)
    \end{pmatrix}.
    \end{align*}
It is easy to see that $\mathcal{\hat A}^\gamma$ is positive-definite with lowest eigenvalue $$\lambda_\gamma\geq \frac12\min\left\{\frac{5\zeta(3)}{2\pi^3},\frac{93\zeta(5)}{8\pi^5}\right\}K^3$$
if $K$ is large enough. If $|\alpha_w|=\delta^{-1/2}K^{-1}\log K$ or $|\hat \alpha_b|=\delta^{-1/2}K^{-1}\log K$, then
\begin{align*}
\Psi(A)-\Psi(A_0)\geq ~&\e^3 (\alpha_w,\hat \alpha_b)\mathcal{\hat A}^\gamma (\alpha_w,\hat \alpha_b)^T+O\left(\frac{\e^3K^3}{\log K}(\alpha_w^2+\hat\alpha_b^2)\right)\\
\geq~ &\lambda_{\gamma}\e^3(\alpha_w^2+\hat \alpha_b^2)+O\left(\frac{\e^3K^3}{\log K}(\alpha_w^2+\hat\alpha_b^2)\right)\\
\geq~ &\frac14\delta^{-1} \min\left\{\frac{5\zeta(3)}{2\pi^3},\frac{93\zeta(5)}{8\pi^5}\right\}\e^3 K(\log K)^2.
%&\geq \frac12\lambda_\gamma \e^2 \delta^{-2}K^{-1}=\frac12\lambda_\gamma  \delta^{-2}K^2
\end{align*}
Note that $J(PQ_A+\phi_A)=C_q+\Psi(A_0)+\Psi(A)-\Psi(A_0)+O(\e^3K(\log K)^2)$. Choosing $\delta$ small and $K$ large, the above analysis implies that
\begin{align*}
    \mathcal{J}_4(0,0)<\mathcal{J}_4(\alpha_w,\hat \alpha_b),\quad  \forall \,|\alpha_w| \text{ or }|\hat \alpha_b|=\delta^{-1/2}K^{-1}\log K.
\end{align*}

Now combining the previous (1)-(4) parts, we know that the infimum of $J(PQ_A+\phi_A)$ must be achieved when the parameters $\e,\xi,a,d,\alpha_b,\alpha_w$ are in the interior of the constraint set. Hence we finish the whole proof.
% \begin{align*}
% J(PQ_A+\phi_A)=\Psi(A)+O(\e^3K^2)\geq \Psi(A_0)+\frac12\lambda_\gamma  \delta^{-2}K^2+O(\e^3K^2)>J(PQ_{A_0}+\phi_{A_0}).
% \end{align*}
% by taking $\delta$ small enough.
\end{proof}

\begin{proof}[Proof of Theorem \ref{th1.sector}.]
By Theorem \ref{th8.para}, it follows that $\inf_{A\in \mathscr{C}}J(PQ_A+\phi_A)$ is attained in the interior of $\mathscr{C}$. By Remark \ref{rmk:JPQA-C1} and Proposition \ref{prop.der}, $J(PQ_A+\phi_A)$ is at least $C^1$ on the parameters of $A$. Then the partial derivatives of $J(PQ_A+\phi_A)$ with respect to the six parameters of $A$ are zero at a minimum point in the interior of $\mathscr{C}$. Then using Lemma \ref{le.finite-p}, we find a nontrivial solution to \eqref{210} and it proves Theorem \ref{th1.sector}.
\end{proof}

% \begin{align*}

\section{Critical computations}
In this section, we present essential estimates for $\gamma$ and $H_0^e$, which are necessary for the analyses discussed in the preceding sections. These two functions exhibit certain computational similarities. We provide a detailed computation for $\gamma$ and outline the corresponding steps for $H_0^e$.

\subsection{Estimates on Type I}
We estimate $\gamma(z,p)$ in this subsection. Recall
\begin{align}\label{def:gamma1}
    \gamma(z,p):=\frac{1}{|\bar z e^{2i\theta_0}-p|}-\sum_{j=1}^{K/2-1}\left(\frac{1}{|z e^{4ji\theta_0}-p|}-\frac{1}{|\bar z e^{(4j+2)i\theta_0}-p|}\right).
\end{align}
% \begin{align}
%     \gamma_2(z,b)&:=\frac{1}{|\bar z e^{2i\theta_0}-b|^2}+\sum_{j=1}^{K/2-1}\left(\frac{1}{|z e^{4ji\theta_0}-b|^2}+\frac{1}{|\bar z e^{(4j+2)i\theta_0}-b|^2}\right)
% \end{align}
The following lemma gives a rough bound of $\gamma$ and its derivatives.
\begin{lemma}\label{lem:gamma-first} For any $z,b\in\Sigma_K$ with $b=(|b|e^{i\alpha_b},0)^T$ and $|b|>\frac12$, one has
\begin{align*}
    |\nabla^m\gamma(z,p)|_{p=b}|\leq CK^{m+1},\quad m=0,1,2,3,4.
\end{align*}
% \[|\gamma(z,b)|\leq CK,\]
% \[|\nabla_z \gamma(z,b)|\leq CK^2,\quad |\nabla_p \gamma(z,b)|\leq CK^2,\]
% \[|\nabla_{z}^2\gamma(z,b)|\leq CK^3,\quad |\nabla_z^2\nabla_p\gamma(z,b)|\leq CK^4.\]
% \begin{align*}
%     |\gamma(z,b)|&\lesssim K\\
%         |\nabla_z \gamma(z,b)|&\lesssim K^2,\quad |\nabla_p \gamma(z,b)|\lesssim K^2\\
%         |\nabla_{z}^2\gamma(z,b)&|\lesssim K^3,\quad |\nabla_z^2\nabla_p\gamma(z,b)|\lesssim K^4
%     % \gamma_2(z,b)&\lesssim K^2\\
%     % \gamma_3(z,b)&\lesssim K^3
% \end{align*}
% Moreover,
% \begin{align*}
%     \gamma(b,b)&\gtrsim K\\
%     |\nabla_x \gamma(z,b)|&\leq \gamma_2(z,b),\quad |\nabla_p \gamma(z,b)|\leq \gamma_2(z,b)\\
% \end{align*}
\end{lemma}
\begin{proof}

Let us first prove the case $m=0$.
    Suppose $z=(z_1,z_2,z_3)=(re^{i\alpha_z},z_3)$. We recall \eqref{z-b-1} and \eqref{z-b-2}.
% \begin{align}
% \begin{split}\label{z-b-1}
%     |ze^{4ji\theta_0}-b|^2&=|z|^2+|b|^2-2r |b|\cos(4j\theta_0+\alpha_z-\alpha_b)\\
%     &=(|z|-|b|)^2+z_3^2+4r|b|\sin^2(2j\theta_0+\tfrac{1}{2}\alpha_z-\tfrac{1}{2}\alpha_b),
% \end{split}
% \end{align}
% and
% \begin{align}
%     \begin{split}\label{z-b-2}
%         |\bar ze^{(4j+2)i\theta_0}-b|^2&=|z|^2+|b|^2-2r |b|\cos((4j+2)\theta_0-\alpha_z-\alpha_b)\\
%     &=(|z|-|b|)^2+z_3^2+4r|b|\sin^2((2j+1)\theta_0-\tfrac{1}{2}\alpha_z-\tfrac{1}{2}\alpha_b).
%     \end{split}
% \end{align}
Since $|\alpha_z\pm \alpha_b|<2\theta_0$, one can see that
\[2j\theta_0+\tfrac12\alpha_z-\tfrac12\alpha_b\in [(2j-1)\theta_0,(2j+1)\theta_0],\]
\[(2j+1)\theta_0-\tfrac12\alpha_z-\tfrac12\alpha_b\in [2j\theta_0,(2j+2)\theta_0].\]
Moreover, they both belong to $[0,\tfrac{\pi}{2}]$ when $1\leq j\leq  j_0=\lfloor K/4-1\rfloor$, the largest integer less than or equal to $K/4-1$. Thus, \eqref{z-b-1}, \eqref{z-b-2} and the monotonicity of Sine function on $[0,\tfrac{\pi}{2}]$ imply that
\begin{align*}
    |\bar ze^{2i\theta_0}-b|<|ze^{4i\theta_0}-b|<|\bar ze^{6i\theta_0}-b|<\cdots<|\bar ze^{(4j_0+2)i\theta_0}-b|.
\end{align*}
Therefore
\begin{align}\label{j0-1}
   0< \sum_{j=1}^{j_0}\left(\frac{1}{|z e^{4ji\theta_0}-b|}-\frac{1}{|\bar z e^{(4j+2)i\theta_0}-b|}\right)<\frac{1}{|ze^{4i\theta_0}-b|}-\frac{1}{|\bar z e^{(4j_0+2)i\theta_0}-b|}.
\end{align}

When $j\geq j_1=\lceil K/4 +1/2\rceil$, the least integer greater than or equal to $K/4+1/2$. In this case, we have
$2j\theta_0+\tfrac12\alpha_z-\tfrac12\alpha_b$ and $(2j+1)\theta_0-\tfrac12\alpha_z-\tfrac12\alpha_b$ belong to $[\tfrac{\pi}{2},\pi]$. Thus \eqref{z-b-1} and \eqref{z-b-2} and the monotonicity of Sine function on $[\tfrac{\pi}{2},\pi]$ imply that
\begin{align*}
    |z e^{4j_1i\theta_0}-b|>|\bar z e^{(4j_1+2)i\theta_0}-b|>\cdots>|\bar z e^{2(K-1)i\theta_0}-b|.
\end{align*}
Therefore
\begin{align}\label{j1-1}
    \frac{1}{|ze^{4j_1i\theta_0}-b|}-\frac{1}{|\bar z e^{2(K-1)i\theta_0}-b|}< \sum_{j=j_1}^{K/2-1}\left(\frac{1}{|z e^{4ji\theta_0}-b|}-\frac{1}{|\bar z e^{(4j+2)i\theta_0}-b|}\right)< 0.
\end{align}
For $j_0\leq j\leq j_1$, one has that $ze^{4ji\theta_0}$ and $\bar z e^{(4j+2)i\theta_0}$ are far from $b$. More precisely, $|ze^{4ji\theta_0}-b|\geq |b|$, and $ |ze^{(4j+2)i\theta_0}-b|\geq |b|$.
% \[|ze^{4ji\theta_0}-b|\geq |b|, \quad |ze^{(4j+2)i\theta_0}-b|\geq |b|.\]
Since $j_1-j_0\leq 3$, then
\[\left|\sum_{j=j_0+1}^{j_1-1}\left(\frac{1}{|z e^{4ji\theta_0}-b|}-\frac{1}{|\bar z e^{(4j+2)i\theta_0}-b|}\right)\right|\leq \frac{6}{|b|}.\]
Combining the above with \eqref{j0-1} and \eqref{j1-1}, we get
\begin{align}
\begin{split}\label{gamma1<}
    \gamma(z,b)&\leq \frac{1}{|\bar z e^{2i\theta_0}-b|}+\frac{1}{|\bar z e^{2(K-1)i\theta_0}-b|}-\frac{1}{|ze^{4j_1i\theta_0}-b|}+\frac{6}{|b|},\\
    \gamma(z,b)&\geq \frac{1}{|\bar z e^{2i\theta_0}-b|}-\frac{1}{|ze^{4i\theta_0}-b|}+\frac{1}{|\bar z e^{(4j_0+2)i\theta_0}-b|}-\frac{6}{|b|}.
\end{split}
\end{align}
Note that for any $z,b\in \Sigma_K$, one has $1/(2K)\leq |ze^{4ji\theta_0}-b|\leq 2$ for any $j\geq 1$. This also holds for $|\bar z e^{2(K-1)i\theta_0}-b|$ and $|z e^{4i\theta_0}-b|$. Thus, the above inequalities imply that
\[|\gamma(z,b)|\leq CK.\]

For the derivatives of $\gamma$,  using the estimates like \eqref{nabla(z-p)}, it is easy to show that for $m\geq 1$
\begin{align*}
    |\nabla^m \gamma(z,b)|\leq C\sum_{j=0}^{K/2-1}\frac{1}{|\bar z e^{(4j+2)i\theta_0}-b|^{m+1}}+C\sum_{j=1}^{K/2-1}\frac{1}{|z e^{4ji\theta_0}-b|^{m+1}}.
\end{align*}
% Recall that $z=(z_1,z_2,z_3)=(re^{i\alpha_z},z_3)$. If $r<\frac14$, then $|ze^{4ji\theta_0}-b|\geq\frac14$ and $|\bar ze^{(4j+2)i\theta_0}-b|\geq\frac14$ for any $j\geq 1$. Thus \[ |\nabla_z \gamma(z,b)|\leq CK.\]
% If $r>\frac14$, then using \eqref{z-b-1} and \eqref{z-b-2} and $|b|\geq \tfrac12$, we obtain that for $1\leq j\leq K/2-1$,
% \begin{align*}
%     |ze^{4ji\theta_0}-b|^2&
%     \geq \tfrac12 \sin^2(2j\theta_0+\tfrac{1}{2}\alpha_z-\tfrac12\alpha_b),\\
%     |\bar ze^{(4j+2)i\theta_0}-b|^2
%     &\geq \tfrac12\sin^2((2j+1)\theta_0-\tfrac12\alpha_z-\tfrac12\alpha_b).
% \end{align*}
% \begin{align*}
%     |ze^{4ji\theta_0}-b|^2&=|z|^2+|b|^2-2r |b|\cos(4j\theta_0+\alpha_z-\alpha_b)\\
%     &\geq 2r|b|(1-\cos(4j\theta_0+\alpha_z-\alpha_b))\gtrsim \min\{(j\theta_0)^2, (K/2-j)^2\theta_0^2\}
% \end{align*}
% where we used $1-\cos x=2(\sin (\frac{x}{2}))^2\geq 2(\frac{x}{\pi})^2=\frac{2}{\pi^2}x^2$ for any $x\in [-\pi,\pi]$.  Similarly
% \begin{align*}
%     |\bar ze^{(4j+2)i\theta_0}-b|^2&=|z|^2+|b|^2-2r |b|\cos((4j+2)\theta_0-\alpha_z-\alpha_b)\\
%     &\geq 2r|b|(1-\cos((4j+2)\theta_0+\alpha_z-\alpha_b))\gtrsim \min\{(j\theta_0)^2,(K/2-j)^2\theta_0^2\}.
% \end{align*}
% For $j_1\leq j\leq K/2-1$,
% \begin{align*}
%     |ze^{4ji\theta_0}-b|^2&=2r|b|(1-\cos (4j\theta_0+\alpha_z-\alpha_b-2K\theta_0))\\
%     &\geq \frac{4}{\pi^2}r|b|[4((K-1)/2-j)\theta_0]^2\gtrsim [((K-1)/2-j)\theta_0]^2
% \end{align*}
% \begin{align*}
%     |\bar z e^{(4j+2)\theta_0}-b|^2\gtrsim []
% \end{align*}
Invoking \eqref{lowerbd-|z-b|}, one has
$$|\nabla^m\gamma(z,b)|\leq C \sum_{j=0}^{K-1}\csc^{m+1} j\theta_0\leq CK^{m+1}.$$
Hence we finish the proof.
\end{proof}
% Combining the above two cases, $|\alpha_z\pm\alpha_b|\leq 2\theta_0$ and $|\alpha_b|\leq \frac12\theta_0$, inserting them to $|\nabla_z\gamma(z,b)|$, one has
% \begin{align*}
%        |\nabla_z\gamma(z,b)|\leq CK^2+C\sum_{j=1}^{K/2-1}\frac{1}{j^2\theta_0^2}\leq CK^2.
% \end{align*}

 In this particular case $z=b$, we have a more precise estimate.
 \begin{lemma}\label{lem:gamma(b,b)} Suppose that $b=(|b|e^{i\alpha_b},0)^T$ satisfies  $|b|>\tfrac12$ and $|\alpha_b|<\frac12\theta_0$. Then
     \begin{align*}
         \gamma(b,b)=\frac{1}{2|b|}\left[\hat S_1(K)+O(K^3\alpha_b^2)\right].
     \end{align*}
     where $\hat S_1(K)$ is defined in Lemma \ref{lem:csc3csc5}.
 \end{lemma}
 \begin{proof}
 Using \eqref{z-b-1}, \eqref{z-b-2} and $|\alpha_b|<\tfrac12\theta_0$, we obtain $|b e^{4ji\theta_0}-b|=2|b|\sin 2j\theta_0$ and $|\bar b e^{(4j+2)i\theta_0}-b|=2|b|\sin ((2j+1)\theta_0-\alpha_b)$. Then
 \begin{align*}
     \gamma(b,b)
     &=\frac{1}{2|b|}\sum_{j=0}^{K/2-1}\frac{1}{\sin ((2j+1)\theta_0-\alpha_b)}-\frac{1}{2|b|}\sum_{j=1}^{K/2-1}\sin (2j\theta_0).
 \end{align*}
The following expansion holds uniformly for any $\theta\in (0,\pi)$ and $|\alpha_b|<\min\{\tfrac{\theta}{2},\pi-\tfrac{\theta}{2}\}$,
 \begin{align*}
     &\frac{1}{\sin (\theta-\alpha_b)}=\frac{1}{\sin \theta}+\frac{\cos\theta}{\sin^2 \theta}\alpha_b+O\left(\frac{\alpha_b^2}{\sin^3\theta}\right).
 \end{align*}
 Replacing $\theta=(2j+1)\theta_0$ in the above and summing on $j$ from 0 to $K/2-1$, we have
 \begin{align*}
     \sum_{j=0}^{K/2-1}\frac{1}{\sin ((2j+1)\theta_0-\alpha_b)}=\sum_{j=0}^{K/2-1}\frac{1}{\sin (2j+1)\theta_0}+O(K^3\alpha_b^2).
 \end{align*}
 Here the odd term on $\alpha_b$ vanishes because $\sum_{j=0}^{K/2-1}\csc ((2j+1)\theta_0-\alpha_b)$ is an even function on $\alpha_b$. One can see this by changing $j$ to $K/2-1-j$ in the summation.
 Consequently
 \begin{align*}
     \gamma(b,b)&=\frac{1}{2|b|}\left(\sum_{j=1}^{K-1}\frac{(-1)^{j+1}}{\sin (j\theta_0)}+O(K^3\alpha_b^2)\right)=\frac{1}{2|b|}\left[\hat S_1(K)+O(K^3\alpha_b^2)\right],
 \end{align*}
 where we have used the notations in Lemma \ref{lem:csc3csc5}.
%  and
% %  the first two terms on the right-hand side are
% % \begin{align*}
% %     \mathcal{C}_0^\gamma:=\sum_{j=1}^{K-1}\frac{(-1)^{j+1}}{\sin (j\theta_0)}
% % \end{align*}
% \begin{align*}
%     (\mathcal{A}_0^\gamma)_{22}:=\sum_{j=0}^{K/2-1}\frac{2-\sin^2(2j+1)\theta_0}{2\sin^3 (2j+1)\theta_0}=S_3^o(K,0)-\frac{1}{2}S_1^o(K,0).
% \end{align*}
 The proof is complete.
 \end{proof}

\begin{lemma}\label{lem:w_dot_z} Suppose that $w=(|w|e^{i\alpha_w},0)^T$ and $b=(|b|e^{i\alpha_b},0)^T$ satisfy $|b|>\frac12$ and $|\alpha_b|<\frac12\theta_0$. Then
    \begin{align*}
        w\cdot \nabla_z\gamma(b,b)=\frac{|w|}{4|b|^2}\left[-\hat S_1(K)+O(K(\log K) \alpha_w^2+K^3|\alpha_w\alpha_b|+K^3\alpha_b^2)\right]
    \end{align*}
and the same expansion applies to $w\cdot \nabla_p\gamma(b,b)$.
\end{lemma}
\begin{proof}
{\bf Step 1}: It is straightforward to verify that for any $\theta\in [0,2\pi]$,
    \begin{align*}
        \nabla_z |ze^{i\theta}-p|^{-1}=-|ze^{i\theta}-p|^{-3}(z-pe^{-i\theta}).
    \end{align*}
Thus
    \begin{align}\label{wnabla-1}
        w\cdot \nabla_z |ze^{i\theta}-p|^{-1}\big|_{z=b,p=b}=\frac{w\cdot (be^{-i\theta}-b)}{|be^{i\theta}-b|^3}=-\frac{|w|}{(2|b|)^2}\frac{\sin (\tfrac{\theta}{2}+\alpha_w-\alpha_b)}{\sin^2 \tfrac{\theta}{2}},
    \end{align}
    where we have used $|b e^{i\theta}-b|=2|b|\sin \tfrac{\theta}{2}$ and
    \begin{align}
    \begin{split}\label{wbeitheta-1}
        w\cdot (be^{-i\theta}-b)&=|w||b|(\cos (\alpha_b-\theta-\alpha_w)-\cos (\alpha_b-\alpha_w))\\
        &=2|w||b|\sin \tfrac{\theta}{2}\sin (\alpha_b-\alpha_w-\tfrac{\theta}{2}).
        \end{split}
    \end{align}
%     where we have used $b e^{-i\theta}-b=b(e^{-i\theta}-1)=-2|b|(\sin \tfrac{\theta}{2})e^{(\frac{\pi}{2}-\frac{\theta}{2}+\alpha_b)i}$ and $w=|w|e^{i\alpha_w}$,
% \[\frac{w\cdot(be^{-i\theta}-b)}{|be^{i\theta}-b|}=-|w|\sin (\tfrac{\theta}{2}-\alpha_b+\alpha_w).\]
% When $\theta=4j\theta_0$, we have $|b e^{4ji\theta_0}-b|=2|b|\sin (2j\theta_0)$ and
% \begin{align}\label{wnabla-1}
%     w\cdot \nabla_{z}|ze^{4ji\theta_0}-p|^{-1}\big|_{z=b,p=b}=-\frac{|w|}{(2|b|)^2}\frac{\sin (2j\theta_0+\alpha_w-\alpha_b)}{\sin^2 2j\theta_0}.
% \end{align}

{\bf Step 2}: It is easy to get that for any $\theta\in [0,2\pi]$
\begin{align*}
        \nabla_z |\bar ze^{i\theta}-p|^{-1}=-|\bar ze^{i\theta}-p|^{-3}(z-\bar p e^{i\theta}).
    \end{align*}
    Thus when $|\alpha_b|<\frac12\theta$
\begin{align}\label{wnabla-2}
        w\cdot \nabla |\bar ze^{i\theta}-p|^{-1}\big|_{z=b,p=b}=\frac{w\cdot (\bar be^{i\theta}-b)}{|\bar be^{i\theta}-b|^3}=-\frac{|w|}{(2|b|)^2}\frac{\sin (\tfrac{\theta}{2}-\alpha_w)}{\sin^2(\tfrac{\theta}{2}-\alpha_b)},
    \end{align}
    where we have used $|\bar b e^{i\theta}-b|=|b|\sin(\frac{\theta}{2}-\alpha_b)$ and
    \begin{align}
    \begin{split}\label{wbbeitheta-2}
        w\cdot (\bar be^{i\theta}-b)&=|w||b|(\cos (\theta -\alpha_b-\alpha_w)-\cos (\alpha_b-\alpha_w))\\
        &=2|w||b|\sin (\alpha_b-\tfrac{\theta}{2})
        \sin (\tfrac{\theta}{2}-\alpha_w).
        \end{split}
    \end{align}

%     $\bar b e^{i\theta}-b=b(e^{i(\theta-2\alpha_b)}-1)=2|b|\sin(\frac{\theta}{2}-\alpha_b)e^{(\frac{\pi}{2}+\frac{\theta}{2})i}$ and
% \[\frac{w\cdot (\bar b e^{i\theta}-b)}{|\bar b e^{i\theta}-b|} =-|w|\sin (\tfrac{\theta}{2}-\alpha_w)\quad \text{when}\quad |\alpha_b|<\frac12\theta.\]
% When $\theta=(4j+2)\theta_0$, we have $|\bar b e^{(4j+2)i\theta_0}-b|=2|b|\sin ((2j+1)\theta_0-\alpha_b)$ and
% \begin{align}\label{wnabla-2}
%     w\cdot \nabla |\bar ze^{i\theta}-p|^{-1}\big|_{z=b,p=b}=-\frac{|w|}{(2|b|)^2}\frac{\sin ((2j+1)\theta_0-\alpha_w)}{\sin^2((2j+1)\theta_0-\alpha_b)}
% \end{align}

{\bf Step 3}: Plugging in $\theta=4j\theta_0$ in  \eqref{wnabla-1} and $\theta=(4j+2)\theta_0$ in  \eqref{wnabla-2}, taking sum on $j$, we obtain
\begin{align}\label{w_dot_nabla_z}
    w\cdot \nabla_z\gamma(z,p)\big|_{z=b,p=b}=\frac{|w|}{(2|b|)^2}[I_1-I_2],
\end{align}
where
\begin{align*}
    I_1=\sum_{j=1}^{K/2-1}\frac{\sin (2j\theta_0+\alpha_w-\alpha_b)}{\sin^2 2j\theta_0},\quad I_2=\sum_{j=0}^{K/2-1}\frac{\sin ((2j+1)\theta_0-\alpha_w)}{\sin^2 ((2j+1)\theta_0-\alpha_b)}.
\end{align*}
% \begin{align*}
%     I_1=\sum_{j=1}^{K/2-1}\frac{\sin (2j\theta_0+\alpha_w-\alpha_b)}{\sin^2 2j\theta_0}
% \end{align*}
Think of $I_1=I_1(\alpha_w,\alpha_b)$ and $I_2=I_2(\alpha_w,\alpha_b)$ as a function of $\alpha_w$ and $\alpha_b$. We want to do the Taylor expansion of $I_1$ and $I_2$ for $\alpha_b,\alpha_w$. Changing $j$ to $K/2-j$, we find that $I_1(\alpha_w,\alpha_b)=I_1(-\alpha_w,-\alpha_b)$. Changing $j$ to $K/2-1-j$ in $I_2$, we find that $I_2(\alpha_w,\alpha_b)=I_2(-\alpha_w,-\alpha_b)$. In fact, using the estimates in Lemma \ref{lem:csc3csc5}, we obtain that
% \begin{align*}
%     I_1&=\sum_{j=1}^{K/2-1}\frac{\cos (\alpha_w-\alpha_b)}{\sin 2j\theta_0}+\sin (\alpha_w-\alpha_b)\sum_{j=1}^{K/2-1}\frac{\cos 2j\theta_0}{\sin^2 2j\theta_0}\\
%     &=\cos (\alpha_w-\alpha_b)\sum_{j=1}^{K/2-1}\csc 2j\theta_0=e_1(1-\frac12 (\alpha_w-\alpha_b)^2+O(\alpha_w^4+\alpha_b^4))
% \end{align*}
\begin{align*}
    I_1&=\sum_{j=1}^{K/2-1}\csc 2j\theta_0+O((K\log K)(\alpha_w^2+\alpha_b^2)),\\
    I_2&=\sum_{j=0}^{K/2-1}\csc (2j+1)\theta_0+O(K(\log K) \alpha_w^2+K^3|\alpha_w\alpha_b|+K^3\alpha_b^2).
\end{align*}
% \begin{align*}
%     I_2=\sum_{j=0}^{K/2-1}\csc (2j+1)\theta_0+O(K(\log K) \alpha_w^2+K^3|\alpha_w\alpha_b|+K^3\alpha_b^2).
% \end{align*}
% \begin{align*}
%     I_2=\sum_{j=0}^{K/2-1}\csc (2j+1)\theta_0+[A_{11}\alpha_w^2+2A_{12}\alpha_w\alpha_b+A_{22}\alpha_b^2]+O(...)
% \end{align*}
% where
% \begin{align*}
%     &A_{11}=\sum_{j=0}^{K/2-1}\frac{-\tfrac12}{\sin (2j+1)\theta_0},\quad A_{12}=\sum_{j=0}^{K/2-1}\frac{-\cos^2(2j+1)\theta_0}{\sin^3 (2j+1)\theta_0}\\
%     &A_{22}=\sum_{j=0}^{K/2-1}\frac{3-2\sin^2(2j+1)\theta_0}{\sin^3 (2j+1)\theta_0}
% \end{align*}
% Therefore $A_{11}=-\frac12 d_1$, $A_{12}=-d_3+d_1$, $A_{22}=3d_3-2d_1$
% \begin{align*}
%     I_2-I_1=e_1-d_1+[\frac12(d_1-e_1)\alpha_w^2+(2d_3-2d_1-e_1)\alpha_w\alpha_b+(-3d_3+2d_1-\frac12e_1)\alpha_b^2]
% \end{align*}
Therefore, using $\sum_{j=0}^{K-1}(-1)^{j}\csc j\theta_0=-\hat S_1(K)$ from Lemma \ref{lem:csc3csc5},
\begin{align*}
    I_1-I_2
    &=-\hat S_1(K)+O(K(\log K) \alpha_w^2+K^3\alpha_w\alpha_b+K^3\alpha_b^2).
\end{align*}
% \begin{align}
%     w\cdot \nabla_z\gamma(z,p)\big|_{z=b,p=b}=\frac{|w|}{(2|b|)^2}
% \end{align}
Inserting this back to \eqref{w_dot_nabla_z}, we get the expansion of $w\cdot \nabla_z\gamma(b,b)$.

For $w\cdot \nabla_p\gamma(b,b)$, the proof goes like the previous one by using the following two identities
    \begin{align*}
        \nabla_p |ze^{i\theta}-p|^{-1}=|ze^{i\theta}-p|^{-3}(ze^{i\theta}-p),\\
        \nabla_p |\bar ze^{i\theta}-p|^{-1}=|\bar ze^{i\theta}-p|^{-3}(\bar ze^{i\theta}- p ).
    \end{align*}
\end{proof}

\begin{lemma}\label{lem:wTn2w-g}
Let $w=(|w|e^{i\alpha_w},0)^T\in \R^3$ and  $b=(|b|e^{i\alpha_b},0)^T$. If $\alpha_w\leq 1$ and $|\alpha_b|\leq \frac12 \theta_0$, we have
\begin{align*}
    w^T\nabla^2_{z,p}\gamma(b,b)w
    =\frac{|w|^2}{8|b|^3}\left[\hat S_1(K)+\hat S_3(K)+(\alpha_w,\alpha_b)\mathcal{A}^\gamma(\alpha_w,\alpha_b)^T+O(\alpha_w^4K^3+\alpha_b^4K^{7})\right],
\end{align*}
where we used the notations in the Lemma \ref{lem:csc3csc5}.
\begin{align*}
    \mathcal{A}^\gamma
    &=\begin{pmatrix}
        S_3^o(K)-2S_1^o(K)+3\hat S_3(K)& -3\hat S_3(K)+3S_1^o(K)\\
        -3\hat S_3(K)+3S_1^o(K)& \frac32[4S_5^o(K)+S_3^o(K)-3S_1^o(K)]+3\hat S_3(K)
    \end{pmatrix}.
\end{align*}
% \begin{align*}
%     &\mathcal{A}^\gamma_{11} =S_3^o(K)-2S_1^o(K)+3\hat S_3(K),\quad \mathcal{A}^\gamma_{12}=-3\hat S_3(K)+3S_1^o(K),\\
%     &\mathcal{A}^\gamma_{22} =\frac32[4S_5^o(K)+S_3^o(K)-3S_1^o(K)]+3\hat S_3(K).
% \end{align*}
\end{lemma}
% \begin{remark}
% Later on, we shall choose $w=R_\beta^T\nabla q(\hat \xi)$. Since we want to $Q_A$ has $\frac{(z-b)_1}{|z-b|^3}$ away from $b$, it follows from \eqref{QA-expn-m} that $w\sim (1,0,0)$.
% \end{remark}
\begin{proof}
{\bf Step 1}: Recall
\begin{align*}
    \partial_{z_\ell,p_j}^2| z -p|^{-1}=\frac{1}{|z-p|^3}\left(\delta_{\ell j}-3\frac{(z-p)_\ell(z-p)_j}{|z-p|^2}\right).
\end{align*}
Using $ze^{i\theta}=R_{\theta}z$,
it is easy to verify that
\[(\nabla^2_{z,p}| z e^{i\theta}-p|^{-1})=|ze^{i\theta}-p|^{-3}R_{-\theta}\left(\delta_{\ell j}-3\frac{(ze^{i\theta}-p)_{\ell}(z e^{i\theta}-p)_j}{|ze^{i\theta}-p|^2}\right).\]
% In matrix notation,
% \begin{align*}
%     \nabla_{z,p}^2| z e^{i\theta}-p|^{-1}=|ze^{i\theta}-p|^{-3}\left[R_{-\theta}(I_3-3vv^T)\right].
% \end{align*}
% where $v=\frac{(ze^{i\theta}-p)^T}{|z e^{i\theta}-p|}$. Then for any vector $w=(|w|e^{i\alpha_w},0)\in \R^3$
% \begin{align*}
%     w^T(\nabla^2_{z,p}| z e^{i\theta}-p|^{-1})w
%     &=\frac{1}{|ze^{i\theta}-p|^3}\left[|w|^2\cos \theta-3 \frac{[w\cdot(ze^{i\theta}-p)][w\cdot (z-pe^{-i\theta})]}{|ze^{i\theta}-p|^2}\right].
% \end{align*}
For any vector $w=(|w|e^{i\alpha_w},0)\in \R^3$ and $z=b$ with $p=b$,
\begin{align*}
    &w^T(\nabla^2_{z,p}| z e^{i\theta}-p|^{-1})w|_{z=b,p=b}\\
    &=\frac{1}{|be^{i\theta}-b|^3}\left[|w|^2\cos \theta+3 \frac{[w\cdot(be^{i\theta}-b)][w\cdot (be^{-i\theta}-b)]}{|be^{i\theta}-b|^2}\right]\\
    &=\frac{|w|^2}{|be^{i\theta}-b|^3}\left[\cos \theta-3 \sin (\tfrac{\theta}{2}+\alpha_b-\alpha_w)\sin (-\tfrac{\theta}{2}+\alpha_b-\alpha_w)\right]\\
     &=\frac{|w|^2}{| be^{i\theta}-b|^3}\left[\frac32-\frac12\cos (\theta-2\alpha_w)\right],
\end{align*}
where we have used \eqref{wbeitheta-1} in the second equal sign.
When $\theta=4j\theta_0$, we have
\begin{align}
\begin{split}\label{2gamma-1}
    &w^T(\nabla^2_{z,p}| z e^{4ji\theta_0}-p|^{-1})w|_{z=b,p=b}=\frac{|w|^2}{(2|b|)^3}\frac{\left[-\cos 4j\theta_0+3 \cos 2(\alpha_b-\alpha_w)\right]}{2(\sin 2j\theta_0)^3}.
\end{split}
\end{align}

{\bf Step 2}: Similarly, we can compute the other one. Recall that
\[\bar ze^{i\theta}=\begin{pmatrix}
\cos \theta & \sin\theta & 0 \\
\sin \theta & -\cos\theta & 0 \\
0 & 0 & 1
\end{pmatrix} \begin{pmatrix}z_1\\z_2\\z_3\end{pmatrix}.\]
It is easy to verify that
\begin{align*}
(\nabla_{z,p}^2| \bar z e^{i\theta}-p|^{-1})=\frac{1}{|\bar z e^{i\theta}-p|^3}\begin{pmatrix}
\cos \theta & \sin\theta & 0 \\
\sin \theta & -\cos\theta & 0 \\
0 & 0 & 1
\end{pmatrix}\left(\delta_{\ell j}-3\frac{(\bar ze^{i\theta}-p)_\ell(\bar z e^{i\theta}-p)_j}{|\bar z e^{i\theta}-p|^2}\right).
\end{align*}
For any $w=(|w|e^{i\alpha_w},0)^T$, and $z=b$ with $p=b$
% \begin{align*}
%     &w^T(\nabla^2_{z,p}|\bar z e^{i\theta}-p|^{-1})w\\
%     &=\frac{1}{|\bar ze^{i\theta}-p|^3}\left[|w|^2\cos(\theta-2\alpha_w)-3\frac{[w\cdot (\bar z e^{i\theta}-p)][w\cdot (z-\bar p e^{i\theta})]}{|\bar z e^{i\theta}-p|^2} \right].
% \end{align*}
%In the case $z=b$ and $p=b$,
\begin{align*}
    w^T(\nabla^2_{z,p}|\bar z e^{i\theta}-p|^{-1})w|_{z=b,p=b}
    &=\frac{1}{|\bar be^{i\theta}-b|^3}\left[|w|^2\cos (\theta-2\alpha_w)+3\frac{[w\cdot (\bar b e^{i\theta}-b)]^2}{|\bar b e^{i\theta}-b|^2} \right]\\
    % &=\frac{1}{|\bar be^{i\theta}-b|^3}\left[(w_1^2-w_2^2)\cos \theta+2w_1w_2\sin\theta+3\frac{[w\cdot (\bar b e^{i\theta}-b)]^2}{|\bar b e^{i\theta}-b|^2} \right]\\
    % &=\frac{1}{|\bar be^{i\theta}-b|^3}\left[|w|^2\cos \theta+2w_2(w_1\sin\theta-w_2\cos\theta)+3|w|^2|\sin (\frac{\theta}{2}-t_w\theta_0)|^2 \right]\\
    % &=\frac{1}{|\bar be^{i\theta}-b|^3}\left[|w|^2(\cos \theta+2\sin(t_w\theta_0)\sin (\theta-t_w\theta_0))+3|w|^2|\sin (\frac{\theta}{2}-t_w\theta_0)|^2 \right]\\
    &=\frac{|w|^2}{|\bar be^{i\theta}-b|^3}\left[\cos (\theta-2\alpha_w)+3|\sin (\tfrac{\theta}{2}-\alpha_w)|^2 \right]\\
     &=\frac{|w|^2}{|\bar be^{i\theta}-b|^3}\left[\frac32-\frac12\cos (\theta-2\alpha_w)\right],
\end{align*}
where we have used \eqref{wbbeitheta-2}.
When $\theta=(4j+2)\theta_0$, we have $|\bar b e^{(4j+2)i\theta_0}-b|=2|b|\sin (2j+1)\theta_0-\alpha_b)$ and
\begin{align}\label{2gamma-2}
     w^T(\nabla^2_{z,p}|\bar z e^{(4j+2)i\theta_0}-p|^{-1})w|_{z=b,p=b}=\frac{|w|^2}{(2|b|)^3}\frac{3-\cos ((4j+2) \theta_0-2\alpha_w)}{2(\sin ((2j+1) \theta_0-\alpha_b))^3}.
\end{align}

{\bf Step 3}: Inserting the above formula \eqref{2gamma-1} and \eqref{2gamma-2} to \eqref{def:gamma1}, we have
\begin{align}\label{gamma2-step3}
    &w^T(\nabla_{z,p}^2\gamma(z,p))w|_{z=b,p=b}=\frac{|w|^2}{(2|b|)^3}[I_1-I_2],
\end{align}
where
\begin{align*}
    I_1=\sum_{j=0}^{K/2-1}\frac{3-\cos ((4j+2)\theta_0-2\alpha_w))}{2(\sin ((2j+1)\theta_0-\alpha_b))^3},  \quad I_2=\sum_{j=1}^{K/2-1}\frac{-\cos 4j\theta_0 +3\cos 2(\alpha_b-\alpha_w)}{2(\sin 2j\theta_0)^3}.
\end{align*}

Think of $I_1=I_1(\alpha_w,\alpha_b)$ and $I_2=I_2(\alpha_w,\alpha_b)$ as a function of $\alpha_w$ and $\alpha_b$. It is easy to know $I_2(\alpha_w,\alpha_b)=I_2(-\alpha_w,-\alpha_b)$ and $I_1(\alpha_w,\alpha_b)=I_1(-\alpha_w,-\alpha_b)$.

% $\alpha_b=K^{-1}\tilde \alpha_b$,

% Note that $\alpha_w=o(1)$ and $\tilde \alpha_b=o(1)$, we do the Taylor expansion $I_1$ and $I_2$ at $\alpha_w=0$ and $\tilde \alpha_b=0$.
 %Since $\alpha_b\ll \theta_0$, we shall let $\alpha_b=t_b\theta_0$.
 %We do the Taylor expansion $I_1$ and $I_2$ at $\alpha_w=0$ and $\tilde \alpha_b=0$.

For $I_1$, the following expansion holds uniformly for $K$ when $|\alpha_w|\leq 1$ and $|\alpha_b|\leq \frac12\theta_0$,
\begin{align*}
    I_1&=\sum_{j=0}^{K/2-1}\frac{3-\cos (4j+2)\theta_0}{2(\sin (2j+1)\theta_0)^3}+[A_{11}\alpha_w^2+2A_{12}\alpha_w\alpha_b+A_{22} \alpha_b^2]+O(\alpha_w^4K^{3}+ \alpha_b^4K^{7}),
\end{align*}
where
\begin{align*}
    A_{11}&=\sum_{j=0}^{K/2-1}\frac{\cos (4j+2)\theta_0}{(\sin (2j+1)\theta_0)^3},\quad A_{12}=\sum_{j=0}^{K/2-1}\frac{-3(\cos ((2j+1)\theta_0))^2}{(\sin (2j+1)\theta_0)^3},\\
    A_{22}
    &=\sum_{j=0}^{K/2-1}\frac{3[4+(\sin (2j+1)\theta_0)^2-3(\sin (2j+1)\theta_0)^4]}{2(\sin (2j+1)\theta_0)^5}.
\end{align*}
% Denote
% \begin{align*}
%     a_1=\sum_{j=0}^{K/2-1}\frac{1}{\sin (2j+1)\theta_0},\quad a_3=\sum_{j=0}^{K/2-1}\frac{1}{(\sin (2j+1)\theta_0)^3},\quad a_5=\sum_{j=0}^{K/2-1}\frac{1}{(\sin (2j+1)\theta_0)^5}
% \end{align*}
% Then $A_{11}=a_3-2a_1$, $A_{12}=-3(a_3-a_1)$, $A_{22}=\frac32[4a_5+a_3-3a_1]$. It is easy to verify that
% \[a_1\sim K\ln K,\quad a_3\sim K^3,\quad a_5\sim K^5\]
% up to some positive constants respectively as $K\to \infty$. Thus $A$ is positively definite when $K$ is large enough.

For $I_2$, the following expansion holds uniformly for $K$ when $|\alpha_w|\leq 1$ and $| \alpha_b|\leq \frac12\theta_0$,
\begin{align*}
    I_2=\sum_{j=1}^{K/2-1}\frac{3-\cos 4j\theta_0}{2(\sin 2j\theta_0)^3}+[B_{11}\alpha_w^2+2B_{12}\alpha_w\alpha_b+B_{22} \alpha_b^2]+O(\alpha_w^4K^{3}+ \alpha_b^4K^3),
\end{align*}
where
\begin{align*}
    B_{11}=\sum_{j=1}^{K/2-1}\frac{-3}{(\sin 2j\theta_0)^3}\quad B_{12}=-B_{11},\quad B_{22}=B_{11}.
\end{align*}
% Denote
% \begin{align*}
%     d_1=\sum_{j=0}^{K/2-1}\frac{1}{\sin (2j+1)\theta_0},\quad d_3=\sum_{j=0}^{K/2-1}\frac{1}{(\sin (2j+1)\theta_0)^3},\quad d_5=\sum_{j=0}^{K/2-1}\frac{1}{(\sin (2j+1)\theta_0)^5},
% \end{align*}
% \begin{align*}
%     e_1=\sum_{j=1}^{K/2-1}\frac{1}{\sin 2j\theta},\quad e_3=\sum_{j=1}^{K/2-1}\frac{1}{(\sin 2j\theta)^3}.
% \end{align*}
%The constant term in $I_1$ and $I_2$
Combining the leading terms of $I_1$ and $I_2$, we obtain
\begin{align*}
    \sum_{j=0}^{K/2-1}\frac{1+(\sin (2j+1)\theta_0)^2}{(\sin (2j+1)\theta_0)^3}-\sum_{j=1}^{K/2-1}\frac{1+(\sin 2j\theta_0)^2}{(\sin 2j\theta_0)^3}&=\sum_{j=1}^{K-1}\frac{(-1)^{j+1}}{(\sin (j\theta_0))^3}+\sum_{j=1}^{K-1}\frac{(-1)^{j+1}}{(\sin (j\theta_0))}\\
    &=\hat S_3(K)+\hat S_1(K).
\end{align*}
Here we have used the notations in Lemma \ref{lem:csc3csc5}.
For the quadratic terms, using some trigonometric identities, we can rewrite them as
$A_{11}=S_3^o(K)-2S_1^o(K)$, $A_{12}=-3(S_3^o(K)-S_1^o(K))$, $A_{22}=\frac32[4S_5^o(K)+S_3^o(K)-3S_1^o(K)]$, $B_{11}=-3\hat S_3^e(K)$, $B_{12}=3\hat S_3^e(K)$, $B_{22}=-3\hat S_3^e(K)$. We denote the coefficients of quadratic terms in $w^T\nabla^2_{z,p}\gamma(b,b)w$ is $\mathcal{A}^\gamma$, then $\mathcal{A}^\gamma=A-B$ and consequently
\begin{align*}
    \mathcal{A}^\gamma:
    =\begin{pmatrix}
        S_3^o(K)-2S_1^o(K)+3\hat S_3^e(K)& -3\hat S_3(K)+3S_1^o(K)\\
        -3\hat S_3(K)+3S_1^o(K)& \frac32[4S_5^o(K)+S_3^o(K)-3S_1^o(K)]+3\hat S_3^e(K)
    \end{pmatrix}.
\end{align*}
The proof is complete by plugging in the expansion of $I_1$ and $I_2$ back to \eqref{gamma2-step3}.
\end{proof}

\subsection{Estimates on the Type II}
In this subsection, we will prove the corresponding estimate of $H_0^e$. Since most of the computations are similar to that of $\gamma(z,p)$, we will sketch the proof. Recall
\begin{align*}
    H_0^e(z,p)=\sum_{j=0}^{K/2-1}[H_0(ze^{4j\theta_0},p)-H_0(\bar ze^{(4j+2)\theta_0},p)].
\end{align*}
The following lemma gives a rough bound of $H_0^e$ and its derivatives.
\begin{lemma}\label{lem:H0e-d} For any $z,b\in\Sigma_K$ with $b=(|b|e^{i\alpha_b},0)^T$ and $1-|b|=(1+o_K(1))\frac{\log K}{K}$,
\begin{align*}
        |\nabla^m H_0^e(z,p)|\big|_{p=b}\leq CK^{m+1}/(\log K)^{m+1},\quad m=0,1,2,3,4.
\end{align*}
% Moreover,
% \begin{align*}
%     \gamma(b,b)&\gtrsim K\\
%     |\nabla_x \gamma(z,b)|&\leq \gamma_2(z,b),\quad |\nabla_p \gamma(z,b)|\leq \gamma_2(z,b)\\
% \end{align*}
\end{lemma}
\begin{proof}
The proof is similar to that of $\gamma$ in Lemma \ref{lem:gamma-first}. We sketch the main steps.
By \eqref{H0-expn},  we know $H_0(ze^{4ji\theta_0},b)=(1+|z|^2|b|^2-2z e^{4ji\theta_0}\cdot b)^{-1/2}$, $
    H_0(\bar ze^{(4j+2)i\theta_0},b)=(1+|z|^2|b|^2-2\bar z e^{(4j+2)i\theta_0}\cdot b)^{-1/2}$ and
% \begin{align*}
%     H_0(ze^{4ji\theta_0},b)&=\frac{1}{\sqrt{1+|z|^2|b|^2-2z e^{4ji\theta_0}\cdot b}},\\
%     H_0(\bar ze^{(4j+2)i\theta_0},b)&=\frac{1}{\sqrt{1+|z|^2|b|^2-2\bar z e^{(4j+2)i\theta_0}\cdot b}}.
% \end{align*}
%     Suppose $z=(z_1,z_2,z_3)=(re^{i\alpha_z},z_3)^T$ and $b=(|b|e^{i\alpha_b},0)^T$. Then
\begin{align}
\begin{split}\label{z-b-H}
    1+|z|^2|b|^2-2z e^{4ji\theta_0}\cdot b
    &=(1-r|b|)^2+z_3^2|b|^2+4r|b|\sin^2(2j\theta_0+\tfrac12\alpha_z-\tfrac12\alpha_b)
\end{split}
\end{align}
and
\begin{align*}
    &1+|z|^2|b|^2-2\bar z e^{(4j+2)i\theta_0}\cdot b=(1-r|b|)^2+z_3^2|b|^2+4r|b|\sin^2((2j+1)\theta_0-\tfrac12\alpha_z-\tfrac12\alpha_b).
\end{align*}
Since $|\alpha_z\pm \alpha_b|<2\theta_0$, we see that
%one can see that
%\[4j\theta_0+\alpha_z-\alpha_b\in [(4j-2)\theta_0,(4j+2)\theta_0],\]
%\[(4j+2)\theta_0-\alpha_z-\alpha_b\in [4j\theta_0,(4j+4)\theta_0].\]
%Moreover, they both belong to $[0,\pi]$ when $1\leq j\leq  j_0=\lfloor K/4-1\rfloor$, the largest integer less than or equal to $K/4-1$. Thus, \eqref{z-b-H} and the monotonicity of Cosine function on $[0,\pi]$ imply that
\begin{align*}
    H_0(z,b)>H_0(\bar ze^{2i\theta_0},b)>H_0(ze^{4j\theta_0},b)>\cdots>H_0(\bar ze^{(4j_0+2)\theta_0},b).
\end{align*}
for $1\leq j\leq  j_0=\lfloor K/4-1\rfloor$ and
\begin{align*}
    H_0(z e^{4j_1i\theta_0},b)<H_0(\bar z e^{(4j_1+2)i\theta_0},b)<\cdots<H_0(\bar z e^{2(K-1)i\theta_0},b).
\end{align*}
for $j\geq j_1=\lceil K/4 +1/2\rceil$.
Therefore, we have
\begin{align*}
   0< \sum_{j=0}^{j_0}\left[H_0(z e^{4ji\theta_0},b)-H_0(\bar z e^{(4j+2)i\theta_0},b)\right]<H_0(z,b)-H_0(\bar z e^{(4j_0+2)i\theta_0},b),
   \end{align*}
   and
   \begin{align*}
   H_0(ze^{4j_1i\theta_0},b)-H_0(\bar z e^{2(K-1)i\theta_0},b)< \sum_{j=j_1}^{K/2-1}\left[H_0(z e^{4ji\theta_0},b)-H_0(\bar z e^{(4j+2)i\theta_0},b)\right]< 0\notag .
\end{align*}

%When $j\geq j_1=\lceil K/4 +1/2\rceil$, the least integer greater than or equal to $K/4+1/2$. In this case, we have
%$4j\theta_0+\alpha_z-\alpha_b$ and $(4j+2)\theta_0-\alpha_z-\alpha_b$ belong to $[\pi,2\pi]$. Thus \eqref{z-b} and the monotonicity of Cosine function on $[\pi,2\pi]$ imply that
% \begin{align*}
%     H_0(z e^{4j_1i\theta_0},b)<H_0(\bar z e^{(4j_1+2)i\theta_0},b)<\cdots<H_0(\bar z e^{2(K-1)i\theta_0},b).
% \end{align*}
% Therefore
% \begin{align}\label{j1-1-H}
%     H_0(ze^{4j_1i\theta_0},b)-H_0(\bar z e^{2(K-1)i\theta_0},b)< \sum_{j=j_1}^{K/2-1}\left[H_0(z e^{4ji\theta_0},b)-H_0(\bar z e^{(4j+2)i\theta_0},b)\right]< 0 .
% \end{align}
When $j_0\leq j\leq j_1$, we have that $ze^{4ji\theta_0}\cdot b<0$ and $\bar ze^{(4j+2)i\theta_0}\cdot b<0$. This implies $H_0(ze^{4ji\theta_0},b)\leq 1$ and $H_0(ze^{4ji\theta_0},b)\leq 1$.
Since $j_1-j_0\leq 3$, then
% \begin{align*}
%     \left|\sum_{j=j_0+1}^{j_1-1}\left[H_0(z e^{4ji\theta_0},b)-H_0(\bar z e^{(4j+2)i\theta_0},b)\right]\right|\leq 6.
% \end{align*}
combining this with the above two inequalities, we get
\begin{align}
\begin{split}\label{H0e<}
    H_0^e(z,b)&\leq H_0(z,b)-H_0(\bar z e^{(4j_0+2)i\theta_0},b)+6,\\
    H_0^e(z,b)&\geq H_0(z e^{4j_1i\theta_0},b)-H_0(\bar ze^{2(K-1)i\theta_0},b)-6.
\end{split}
\end{align}
Note that for any $z,b\in B_1$, one has $ 1+|z|^2|b|^2-2z\cdot b\leq 4$ and
% \begin{align*}
%     \frac{1}{\sqrt{1-2z\cdot b+|z|^2|b|^2}}\leq \frac{}{}
% \end{align*}
\begin{align*}
    %1+|z|^2|b|^2-2z\cdot b&\leq 4,\\
    1+|z|^2|b|^2-2z\cdot b&\geq 1+|z|^2|b|^2-2|z| |b|=(1-|z||b|)^2\geq (1-|b|)^2.
\end{align*}
Therefore $1/2\leq H_0(z,b)\leq (1-|b|)^{-1}$. As a consequence, \eqref{H0e<} implies that
\[|H_0^e(z,b)|\leq C(1-|b|)^{-1}\leq CK/\log K.\]

For the derivative of $H_0^e$. It is straightforward to verify that
\begin{align*}
    \left|\nabla^mH_0(z,p)\big|_{p=b}\right|\leq C |z-b^*|^{-m-1},\quad m=1,2,3,4
\end{align*}
Similar to \eqref{lowerbd-|z-b|}, we can prove
\begin{align*}
    |ze^{4ji\theta_0}-b^*|^2&\geq (1-|b|)^2+\sin^2 2j\theta_0\\
    |ze^{(4j+2)i\theta_0}-b^*|^2&\geq (1-|b|)^2+\sin^2 (2j+1)\theta_0
\end{align*}
for $0\leq j\leq K/2-1$. Thus
\begin{align*}
\left|\nabla^mH_0^e(z,p)\big|_{p=b}\right|
\leq \sum_{j=0}^{K-1}\frac{1}{((1-|b|)^2+\sin^2 j\theta_0)^{\frac{m+1}{2}}}\leq C\frac{K^{m+1}}{(\log K)^m}
\end{align*}
where we apply Lemma \ref{lem:rough-Sk} since $1-|b|=(1+o_K(1))\frac{\log K}{K}$.
\end{proof}

\begin{lemma}\label{lem:H0e(b,b)} Suppose that $b=(|b|e^{i\alpha_b},0)^T$ satisfies  $d=(1-|b|^2)/(2|b|)=(1+o_K(1))\frac{\log K}{K}$ and $|\alpha_b|\leq \frac12 \theta_0$. One has
    \begin{align*}
        H_0^e(b,b)=\frac{1}{2|b|}[S_1(K,d)+O(Kd^{-2}\alpha_b^2)].
    \end{align*}
\end{lemma}
\begin{proof}
Note that
\begin{align}
    |be^{i\theta}-b^*|&=(|b|^2+|b|^{-2}-2\cos \theta)^{1/2}=2(d^2+\sin^2\tfrac{\theta}{2} )^{\frac12},\label{|b-b|-1}\\
    |\bar b e^{i\theta}-b^*|&=(|b|^{2}+|b|^{-2}-2\cos(\theta-2\alpha_b))^{\frac12}=2(d^2+\sin^2(\tfrac{\theta}{2}-\alpha_b))^{\frac12}.\label{|bb-b|-2}
\end{align}
Since $H(z,p)=(|p||z-p^*|)^{-1}$ and the definition of $H_0^e(z,p)$ in \eqref{def:H0e}, then
\begin{align*}
    H_0^e(b,b)=\frac{1}{2|b|}\sum_{j=0}^{K/2-1}(d^2+\sin^2 2j\theta_0)^{-\frac12}-(d^2+\sin^2 ((2j+1)\theta_0-\alpha_b))^{-\frac12}.
\end{align*}
The following expansion holds uniformly for $|\alpha_b|<\frac12 \theta_0$
\begin{align*}
    \sum_{j=0}^{K/2-1}(d^2+\sin^2 ((2j+1)\theta_0-\alpha_b))^{-\frac12}=\sum_{j=0}^{K/2-1}(d^2+\sin^2 ((2j+1)\theta_0))^{-\frac12}+O(Kd^{-2}\alpha_b^2).
\end{align*}
Thus
\begin{align*}
    H_0^e(b,b)=\frac{1}{2|b|}\sum_{j=0}^{K-1}\frac{(-1)^j}{(d^2+\sin^2 j\theta_0)^{\frac12}}+O(Kd^{-2}\alpha_b^2)=\frac{1}{2|b|}[S_1(K,d)+O(Kd^{-2}\alpha_b^2)].
\end{align*}
\end{proof}

\begin{lemma}\label{lem:w_dot_zH} Suppose that $b=(|b|e^{i\alpha_b},0)^T$ satisfies $1-|b|=(1+o_K(1))\frac{\log K}{K}$ and $|\alpha_b|\leq \frac12 \theta_0$. One has
    \begin{align*}
        w\cdot \nabla_zH_0^e(z,p)\big|_{z=b,p=b}=\frac{|w|}{4|b|^2}&\left[-S_1(K,d)+d\sqrt{1+d^2}S_3(K,d)+O(Kd^{-1}\alpha_w^2+Kd^{-3}\alpha_b^2)\right],
    \end{align*}
    where $d=(1-|b|^2)/(2|b|)$. The expansion of $w\cdot \nabla_pH_0^e(z,p)\big|_{z=b,p=b}$ is similar to the above.
%      \begin{align*}
%        w\cdot \nabla_pH_0^e(z,p)\big|_{z=b,p=b}=\frac{|w|}{4|b|^2}&\left[\sqrt{\frac{8}{\pi}}K^2(Kd)^{-\frac12}e^{-Kd}(1+O((Kd)^{-1})+\right.\\
%       &+ O(Kd^{-2}\alpha_w^2+Kd^{-3}\alpha_w\alpha_b+Kd^{-4}\alpha_b^2)]
%    \end{align*}
\end{lemma}
\begin{proof}
    Recall that $H_0(z,p)=(|p||z-p^*|)^{-1}$. It is easy to see that
    \begin{align*}
        \nabla_zH_0(ze^{i\theta},p)=-\frac{z-p^*e^{-i\theta}}{|p||ze^{i\theta}-p^*|^3},\quad  \nabla_zH_0(\bar ze^{i\theta},p)=-\frac{z-(\bar p e^{i\theta})^*}{|p||\bar z e^{i\theta}-p^*|^2}.
    \end{align*}
    Then
    \begin{align}\label{wdzH-1}
    \begin{split}
        w\cdot \nabla_zH_0(ze^{i\theta},p)|_{z=b,p=b}
        &=\frac{|w|}{8|b|^2}\frac{(\cos(\alpha_b-\alpha_w-\theta)-|b|^2\cos(\alpha_b-\alpha_w))}{(d^2+\sin^2\tfrac{\theta}{2})^{3/2}},
        \end{split}
    \end{align}
    \begin{align}\label{wdzH-2}
    \begin{split}
        w\cdot \nabla_zH_0(\bar z e^{i\theta},p)\big|_{z=b,p=b}
        &=\frac{|w|}{8|b|^2}\frac{\cos(\theta-\alpha_b-\alpha_w)-|b|^2\cos(\alpha_b-\alpha_w)}{(d^2+\sin^2 (\tfrac{\theta}{2}-\alpha_b))^{3/2}}.
    \end{split}
    \end{align}

Plugging in $\theta= 4j\theta_0 $ in \eqref{wdzH-1} and $\theta=(4j+2)\theta_0$ in \eqref{wdzH-2}, and summing on $j$, we get
    \begin{align*}
       w\cdot \nabla_zH_0^e(z,p)\big|_{z=b,p=b}=\frac{|w|}{8|b|^2}[I_1-I_2],
    \end{align*}
    where
\begin{align*}
    I_1&=\sum_{j=0}^{K/2-1}\frac{\cos (4j\theta_0+\alpha_w-\alpha_b)-|b|^2\cos (\alpha_w-\alpha_b)}{(d^2+\sin^2 2j\theta_0)^{3/2}},\\
       I_2&=\sum_{j=0}^{K/2-1}\frac{\cos ((4j+2)\theta_0-\alpha_w-\alpha_b)-|b|^2\cos (\alpha_w-\alpha_b)}{(d^2+\sin^2 ((2j+1)\theta_0-\alpha_b))^{3/2}}.
\end{align*}
One can compute the Taylor expansion of $I_1$ and $I_2$ with respect to $\alpha_w$, $\alpha_b$.
%\begin{align*}
%    I_2=\sum_{j=0}^{K/2-1}\frac{\cos (\alpha_b+\alpha_w-(4j+2)\theta_0)-|b|^2\cos (\alpha_w-\alpha_b)}{(d^2+\sin^2 ((2j+1)\theta_0-\alpha_b))^{3/2}}
%\end{align*}
% Computing the Taylor expansion of them,
% \begin{equation*}
% \begin{aligned}
%     I_1=\sum_{j=0}^{K/2-1}\frac{\cos (4j\theta_0)-|b|^2}{(d^2+\sin^2 2j\theta_0)^{3/2}}+O(Kd^{-2}(\alpha_b^2+ \alpha_w^2) ),%\\
%     %\textcolor{red}{I_1&=\sum_{j=0}^{K/2-1}\frac{\cos (4j\theta_0)-|b|^2}{(d^2+\sin^2 2j\theta_0)^{3/2}}+O(K(\alpha_b^2+ \alpha_w^2) )+O(Kd^{-2}(\alpha_b^4+ \alpha_w^5)),}
% \end{aligned}
% \end{equation*}
% \begin{align*}
%     I_2=\sum_{j=0}^{K/2-1}\frac{\cos ((4j+2)\theta_0)-|b|^2}{(d^2+\sin^2 (2j+1)\theta_0)^{3/2}}+O(Kd^{-1}\alpha_w^2+Kd^{-2}\alpha_w\alpha_b+Kd^{-3}\alpha_b^2).
% \end{align*}
The leading term in $I_1-I_2$ is
\begin{align*}
    &\sum_{j=0}^{K/2-1}\frac{\cos (4j\theta_0)-|b|^2}{(d^2+\sin^2 2j\theta_0)^{3/2}}- \frac{\cos ((4j+2)\theta_0)-|b|^2}{(d^2+\sin^2 (2j+1)\theta_0)^{3/2}}=-2S_1(K,d)+2d\sqrt{1+d^2}S_3(K,d),
       % &=-2S_1+2d\sqrt{1+d^2}S_3
     \end{align*}
     where we used  $ \cos 2\theta -|b|^2=-2(d^2+\sin^2\theta)+2d^2+1-|b|^2$ and $2d^2+1-|b|^2=2d\sqrt{1+d^2}$ and the notation \eqref{def:Sk(n,x)} in Section \ref{sec:series}.
     This completes the proof of $w\cdot \nabla_z H_0^e(z,p)|_{z=b,p=b}$.

   For $w\cdot \nabla_p H_0^e(z,p)|_{z=b,p=b}$,
one can repeat the above proof verbatim.
\end{proof}

\begin{lemma}\label{lem:wTn2w-h}
Suppose that $w=(|w|e^{i\alpha_w},0)^T$ and $b=(|b|e^{i\alpha_b},0)^T$. Assume that $|\alpha_w|\leq 1$ and $|\alpha_b|\leq \frac12\theta_0$, $1-|b|=(1+o_K(1))\frac{\log K}{K}$ as $K\to \infty$. Then
\begin{align*}
    &w^T\nabla^2_{z,p}H_0^e(z,p)w|_{z=b,p=b}\\
    &=\frac{|w|^2}{8|b|^3}\left[S_1-(d+\sqrt{1+d^2})^2S_3+3(d^2+d^4)S_5+O(Kd^{-2}\alpha_w^2+Kd^{-4}\alpha_b^2)\right],
\end{align*}
where $d=(1-|b|^2)/(2|b|)$ and $S_i=S_i(K,d)$ for $i=1,3,5$.
\end{lemma}
\begin{proof}
{\bf Step 1}: It is easy to see that
\begin{align*}
    \partial_{z_l,p_j}^2H_0(z,p)=\frac{1}{(|p|\left|z-p^*\right|)^3}\left(\delta_{\ell j}-2p_jz_\ell+3|z|^2\frac{(p^*-z)_\ell(z^*-p)_j}{\left|z-p^*\right|^2}\right),
\end{align*}
where $z^*=\frac{z}{|z|^2}$.
% Using $ze^{i\theta}=R_\theta z$,
% one can get
% \[\nabla^2_{z,p} H_0(ze^{i\theta},p)=\frac{1}{(|p||ze^{i\theta}-p^*|)^3}R_{-\theta}\left(\delta_{\ell j}-2p_j(ze^{i\theta})_\ell+3|z|^2\frac{(p^*-z e^{i\theta})_\ell((ze^{i\theta})^*-p)_j}{|ze^{i\theta}-p^*|^2}\right).\]
% In matrix notation,
% \begin{align*}
%     \nabla_{z,p}^2 H_0(z e^{i\theta},p)=\frac{1}{(|p||ze^{i\theta}-p^*|)^3}\left[R_{-\theta}-2z p^T+3|z|^2\frac{(p^*e^{-i\theta}-z)(z^*e^{i\theta}-p)^T}{|ze^{i\theta}-p^*|^2}\right].
% \end{align*}
%  Then for any vector $w=(|w|e^{i\alpha_w},0)\in \R^3$, we have
% \begin{align*}
%     &w^T(\nabla^2_{z,p}H_0(z e^{i\theta},p)w\\
%     &=\frac{1}{(|p||ze^{i\theta}-p^*|)^3}\left[|w|^2\cos \theta-2(p\cdot w)(z\cdot w)+3 |z|^2\frac{[w\cdot (z^*e^{i\theta}-p)][w\cdot(p^*e^{-i\theta}-z)]}{\left|ze^{i\theta}-p^*\right|^{2}}\right].
% \end{align*}
Then for any vector $w=(|w|e^{i\alpha_w},0)\in \R^3$,
\begin{align*}
    &w^T(\nabla^2_{z,p}H_0( z e^{i\theta},p)w|_{z=b,p=b}\\
    &=\frac{1}{(|b||be^{i\theta}-b^*|)^3}\left[|w|^2\cos \theta-2(b\cdot w)^2+3 |b|^2\frac{[w\cdot(b^*e^{i\theta}-b)][w\cdot (b^*e^{-i\theta}-b)]}{|be^{i\theta}-b^*|^2}\right].
\end{align*}

To simplify this, we shall use the computation in $\eqref{wdzH-1}$ to achieve
\begin{align*}
    &{[w\cdot(b^*e^{i\theta}-b)][w\cdot (b^*e^{-i\theta}-b)]}\\
    &=\frac{|w|^2}{|b|^2}(\cos (\theta+\alpha)-|b|^2\cos(\alpha))(\cos(-\theta+\alpha)-|b|^2 \cos \alpha)\\
    &=\frac{|w|^2}{|b|^2}\left[\cos (\theta+\alpha)\cos (-\theta +\alpha)-2|b|^2\cos^2 \alpha \cos \theta+|b|^4\cos^2\alpha \right]\\
    &=\frac{|w|^2}{|b|^2}\left[-\sin^2\theta+\cos^2\alpha (1+|b|^4-2|b|^2\cos \theta)\right],
\end{align*}
where $\alpha=\alpha_b-\alpha_w$. Using this computation and  \eqref{|b-b|-1}, we have
\begin{align}\label{2H-1}
    w^T(\nabla^2_{z,p}H_0( z e^{i\theta},p)w|_{z=b,p=b}
    =\frac{|w|^2}{8|b|^3}\left[\frac{\cos \theta+|b|^2\cos^2 (\alpha_b-\alpha_w)}{(d^2+\sin^2\frac{\theta}{2})^{3/2}}-\frac{3\sin^2\theta}{4(d^2+\sin^2\frac{\theta}{2})^{5/2}}\right].
\end{align}
% When $\theta=4j\theta_0$, one has
% \begin{align}
% \begin{split}\label{2H-1}
%     &w^T(\nabla^2_{z,p}H_0( z e^{4ji\theta_0},p)w|_{z=b,p=b}\\
%     &\quad =\frac{|w|^2}{8|b|^3}\left[\frac{\cos 4j\theta_0+|b|^2\cos^2 (\alpha_b-\alpha_w)}{(d^2+\sin^2 2j\theta_0)^{3/2}}-\frac{3\sin^24j\theta_0}{4(d^2+\sin^22j\theta_0)^{5/2}}\right]
% \end{split}
% \end{align}

{\bf Step 2}: Similarly, we can compute the other one.
% Recall that
% \[\bar ze^{i\theta}=\begin{pmatrix}
% \cos \theta & \sin\theta & 0 \\
% \sin \theta & -\cos\theta & 0 \\
% 0 & 0 & 1
% \end{pmatrix} \begin{pmatrix}z_1\\z_2\\z_3\end{pmatrix}.\]
% It is easy to verify that
% \begin{align*}
% &\nabla_{z,p}^2 H_0(\bar z e^{i\theta},p)\\
% &=\frac{1}{(|p||\bar z e^{i\theta}-p^*|)^3}\begin{pmatrix}
% \cos \theta & \sin\theta & 0 \\
% \sin \theta & -\cos\theta & 0 \\
% 0 & 0 & 1
% \end{pmatrix}\left(\delta_{\ell j}-2p_j(\bar z e^{i\theta})_\ell+3|z|^2\frac{(p^*-\bar ze^{i\theta})_\ell((\bar z e^{i\theta})^*-p)_j}{|\bar z e^{i\theta}-p^*|^2}\right).
% \end{align*}
% For any $w\in \R^3$,
% \begin{align*}
% \nabla^2_{z,p}H_0(\bar ze^{i\theta},p)w=\frac{1}{(|p||\bar z e^{i\theta}-p^*|)^3}\left[\bar w e^{i\theta}-2z(p\cdot w)+3|z|^2\frac{((\bar p e^{i\theta})^*-z)_i[w\cdot((\bar ze^{i\theta})^*-p)]}{|\bar z e^{i\theta}-p^*|^2}\right].
% \end{align*}
% For any $w=(|w|e^{i\alpha_w},0)^T$,
% \begin{align*}
%     w^T&(\nabla^2_{z,p}H_0(\bar z e^{i\theta},p))w=\frac{1}{(|p||\bar ze^{i\theta}-p^*|)^3}\\
%     &\times\left[|w|^2\cos(\theta-2\alpha_w)-2(z\cdot w)(p\cdot w)+3|z|^2\frac{[w\cdot ((\bar p e^{i\theta})^*-z)][w\cdot ((\bar z e^{i\theta})^*-p)]}{|\bar z e^{i\theta}-p^*|^2} \right].
% \end{align*}
% In the case $z=b$ and $p=b$,
For any $w=(|w|e^{i\alpha_w},0)^T$,
\begin{align}
\begin{split}\label{2H-2}
        &w^T\nabla^2_{z,p}H_0(\bar z e^{i\theta},p)w|_{z=b,p=b}
    \\
    &=\frac{1}{(|b||\bar be^{i\theta}-b^*|)^3}\left[|w|^2\cos (\theta-2\alpha_w)-2(b\cdot w)^2+3|b|^2\frac{[w\cdot ((\bar b e^{i\theta})^*-b)]^2}{|\bar b e^{i\theta}-b^*|^2} \right]\\
    %&=\frac{|w|^2}{(|b||b^*-\bar be^{i\theta}|)^3}\left[\cos (\theta-2\alpha_w)-2|b|^2\cos^2\alpha+3|b|^2\cos ^2(\alpha+\vartheta) \right].\\
    &=\frac{|w|^2}{8|b|^3}\left[\frac{\cos(\theta-2\alpha_w)-2|b|^2 \cos^2\alpha}{(d^2+\sin^2 (\frac{\theta}{2}-\alpha_b))^{3/2}}+\frac{3\left[\cos (\theta-\alpha_b-\alpha_w)-|b|^2 \cos \alpha\right]^2}{4(d^2+\sin^2(\frac{\theta}{2}-\alpha_b))^{5/2}}\right]\\
    &=:\frac{|w|^2}{8|b|^3}f(\alpha_b,\alpha_w,\theta),
\end{split}
\end{align}
where $\alpha=\alpha_b-\alpha_w$ and we have used the computation in \eqref{wdzH-2} and \eqref{|bb-b|-2} in the last step.

% When $\theta=(4j+2)\theta_0$,
% \begin{align}
% \begin{split}
%      &w^T\nabla^2_{z,p}H_0(\bar z e^{(4j+2)\theta_0},p)w|_{z=b,p=b}\\
%      &=\frac{|w|^2}{8|b|^3}\left[\frac{\cos((4j+2)\theta-2\alpha_w)-2|b|^2 \cos^2\alpha}{(d^2+\sin^2 ((2j+1)\theta_0-\alpha_b))^{3/2}}+\frac{3\left[\cos (\theta-\alpha_b-\alpha_w)-|b|^2 \cos \alpha\right]^2}{4(d^2+\sin^2\frac{\theta}{2})^{5/2}}\right]
% \end{split}
% \end{align}
% where $\vartheta_{4j+2}$ is the one satisfies \eqref{sinR-2} for $\theta=(4j+2)\theta_0$.

{\bf Step 3}:
Inserting $\theta=4j\theta_0$ in \eqref{2H-1} and $\theta=(4j+2)\theta_0$ in \eqref{2H-2}, summing on $j$, we obtain
\begin{align}\label{wn2H}
    w^T\nabla^2_{z,p}H^e_0(z,p)w|_{z=b,p=b}=\frac{|w|^2}{(2|b|)^3}[I_1-I_2],
\end{align}
where
\begin{align*}
    I_1&=\sum_{j=0}^{K/2-1}\frac{\cos 4j\theta_0+|b|^2\cos^2 (\alpha_b-\alpha_w)}{(d^2+\sin^2 2j\theta_0)^{3/2}}-\frac{3\sin^24j\theta_0}{4(d^2+\sin^22j\theta_0)^{5/2}},\\
    I_2&=\sum_{j=0}^{K/2-1}f(\alpha_b,\alpha_w,(4j+2)\theta_0).
\end{align*}
Here $f(\alpha_b,\alpha_w,\theta)$ is defined in  \eqref{2H-2}.

Next, we need to compute $I_1$ and $I_2$. Again, we have $I_j(\alpha_b,\alpha_w)=I_j(-\alpha_b,-\alpha_w)$ for $j=1,2$.
% Using
% \begin{align*}
%     \sum_{j=0}^{K/2-1}\frac{|b|^2\sin^2(\alpha_b-\alpha_w)}{(d^2+\sin^2 2j\theta_0)^{3/2}}\leq C(\alpha_b-\alpha_w)^2S_3(K,d)\leq C Kd^{-2}(\alpha_b^2+\alpha_w^2)
% \end{align*}
It is easy to have the expansion of $I_1$ as
  \begin{align*}
        I_1=\sum_{j=0}^{K/2-1}\frac{\cos (4j\theta_0)+|b|^2}{(d^2+\sin^2 (2j\theta_0))^{3/2}}-\sum_{j=0}^{K/2-1}\frac{3\sin^2(4j\theta_0)}{4(d^2+\sin^2(2j\theta_0))^{5/2}}+O(Kd^{-2}(\alpha_b^2+\alpha_w^2)).
    \end{align*}
After some lengthy calculation, the expansion of $I_2$ is
   \begin{align*}
        I_2=&\sum_{j=0}^{K/2-1}\frac{\cos ((4j+2)\theta_0)+|b|^2}{(d^2+\sin^2 ((2j+1)\theta_0))^{3/2}}-\sum_{j=0}^{K/2-1}\frac{3\sin^2 ((4j+2)\theta_0)}{(d^2+\sin^2 ((2j+1)\theta_0))^{5/2}}\\
        &+O(Kd^{-2}\alpha_w^2+Kd^{-4}\alpha_b^2).
    \end{align*}
%We put it off in the following Lemma \ref{lem:H-I2}.
Using the expansion of $I_1$ and $I_2$
, we have
% \begin{align*}
%    I_1-I_2=&\ \sum_{j=0}^{K-1}(-1)^j\frac{\cos 2j\theta_0+|b|^2}{(d^2+\sin^2 j\theta_0)^{3/2}}-\frac34\sum_{j=0}^{K-1}(-1)^j\frac{\sin^2 2j\theta_0}{(d^2+\sin^2 j\theta_0)^{5/2}}\\
%    &+O(Kd^{-2}\alpha_w^2+Kd^{-4}\alpha_b^2).
% \end{align*}
\begin{align*}
   I_1-I_2=\sum_{j=0}^{K-1}\frac{(-1)^j(\cos 2j\theta_0+|b|^2)}{(d^2+\sin^2 j\theta_0)^{3/2}}-\frac34\sum_{j=0}^{K-1}\frac{(-1)^j\sin^2 2j\theta_0}{(d^2+\sin^2 j\theta_0)^{5/2}}+O(Kd^{-2}\alpha_w^2+Kd^{-4}\alpha_b^2).
  % &+O(Kd^{-2}\alpha_w^2+Kd^{-4}\alpha_b^2).
\end{align*}
The leading term in the above can be rewritten in terms of $S_i(K,d)$ defined in \eqref{def:Sk(n,x)}.
Using $\cos 2j\theta_0=1+2d^2-2(d^2+\sin^2 j\theta_0)$ and
\begin{align*}
    \frac14\sin^2 2j\theta_0&=\sin^2 j\theta_0\cos^2 j\theta_0=\sin^2 j\theta_0-\sin^4 j\theta_0\\
    &=(1+2d^2)(d^2+\sin^2 j\theta_0)-(d^2+\sin^2 j\theta_0)^2 -d^4-d^2,
    %&=(1+2d^2)\sin^2 2j\theta_0+d^4-(d^2+\sin^2 2j\theta_0)^2
\end{align*}
we obtain
\begin{align*}
&\sum_{j=0}^{K-1}\frac{(-1)^j(\cos 2j\theta_0+|b|^2)}{(d^2+\sin^2 j\theta_0)^{3/2}}-\frac34\sum_{j=0}^{K-1}\frac{(-1)^j\sin^2 2j\theta_0}{(d^2+\sin^2 j\theta_0)^{5/2}}\\
    &=(1+2d^2)S_3-2S_1+[|b|^2-3(1+2d^2)]S_3+3S_1+3(d^2+d^4)S_5\\
    &=S_1+[|b|^2-2(1+2d^2)]S_3+3(d^2+d^4)S_5,
\end{align*}
where $S_i=S_i(K,d)$.

Since $|b|=\sqrt{1+d^2}-d$, then $|b|^2-2(1+2d^2)=-(d+\sqrt{1+d^2})^2$. The proof is complete.
\end{proof}

\section{Some estimates on finite sum}\label{sec:series}

This section will prove several estimates of finite sums involving the trigonometric function. These estimates are needed for the previous section.
For any even integer $n\geq 2$, any integer $k\geq 1$ and $x\in \R$, we define
\begin{align}
    S^o_k(n,x)&:=\sum_{j=0}^{n/2-1}(x^2+\sin^2(\tfrac{(2j+1)\pi}{n}))^{-\frac{k}{2}},\quad
    S^e_k(n,x):=\sum_{j=0}^{n/2-1}(x^2+\sin^2(\tfrac{2j\pi}{n}))^{-\frac{k}{2}},\\
      S_k(n,x)&:=S^e_k(n,x)-S^o_k(n,x)=\sum_{j=0}^{n-1}(-1)^{j} (x^2+(\sin (\tfrac{j\pi}{n}))^2)^{-\frac{k}{2}}.\label{def:Sk(n,x)}
\end{align}
Here $S^o$ denotes the sum of odd parts, and $S^e$ denotes the sum of even parts. Note that when $x=0$, the sum in $S_k^e$ has a singular term, thus we also define\footnote{Note that the order of difference in $\hat S_k(n,x)$ is different from that in $\hat S_k(n,x)$. The purpose of doing this is to make both of them positive when $x=0$. }
\begin{align}
\hat S^e_k(n,x)&:=\sum_{j=1}^{n/2-1}(x^2+\sin^2(\tfrac{2j\pi}{n}))^{-\frac{k}{2}},\\
      \hat S_k(n,x):=S^o_k(n,x,\theta)&-\hat S^e_k(n,x)=\sum_{j=1}^{n-1}(-1)^{j+1} (x^2+(\sin (\tfrac{j\pi}{n}))^2)^{-\frac{k}{2}}.\label{def:hatSk(n,x)}
\end{align}

 If $x=0$, we denote $S_k^o(n,0)$ as $S^o_k(n)$. The same rule applies to $S^e_k$, $S_k$, $\hat S^e_k$, and $\hat S_k$.

% Note that $S^o_k, S^e_k, S_k$ are periodic functions.
% \[S_1^o(n,x,\theta+\frac{2\pi}{n})=S_1^o(n,x,\theta),\quad S_1^e(n,x,\theta+\frac{2\pi}{n})=S_1^e(n,x,\theta).\]
% Changing $j$ to $n/2-1-j$, we observe that $S^o_k,S^e_k,S_k$ are even on $\theta$.
% \begin{align}
%        S_k(n,x,\theta)&:=S^o_k(n,x,\theta)-S^e_k(n,x,\theta)=\sum_{j=0}^{n-1}(-1)^j (x^2+(\sin (\tfrac{j\pi}{n}+\theta))^2)^{-\frac{k}{2}}.\label{def:Sk(n,x)}
% \end{align}
% \begin{align*}
%     \sum_{j=0}^{n/2-1}\left(x^2+\sin^2 \left(\tfrac{(2j+1)\pi}{n}+\theta\right)\right)^{-\frac12}=\frac{n}{\pi}\int_{\arcsinh x}^\infty\frac{\tanh \frac{n t}{2}}{(1-\sech^2 \frac{nt}{2}\sin^2\frac{n\theta}{2})\sqrt{\sinh^2 t-x^2}}dt
% \end{align*}
% \begin{align*}
%     \sum_{j=0}^{n/2-1}\left(x^2+\sin^2 \left(\tfrac{2j\pi}{n}+\theta\right)\right)^{-\frac12}=\frac{n}{\pi}\int_{\arcsinh x}^\infty\frac{\coth \frac{n t}{2}}{(1+\csch^2 \frac{nt}{2}\sin^2\frac{n\theta}{2})\sqrt{\sinh^2 t-x^2}}dt
% \end{align*}
% \begin{align*}
%     S_1(n,x,\theta):&=\sum_{j=0}^{n-1
%     }(x^2+\sin^2 \left(\tfrac{j\pi}{n}+\theta\right))^{-\frac12}\\
%     &=\frac{n}{\pi}\int_{\arcsinh x}^\infty\frac{\tanh \frac{n t}{2}+\coth \frac{nt}{2}}{(1+\csch^2 nt\sin^2 n\theta)\sqrt{\sinh^2 t-x^2}}dt\\
%     &=\frac{2n}{\pi}\int_{\arcsinh x}^\infty\frac{\coth nt}{(1+\csch^2 nt\sin^2 n\theta)\sqrt{\sinh^2 t-x^2}}dt
% \end{align*}

\begin{lemma}\label{lem:csc3csc5}
Suppose $n$ is an even number. As $n\to \infty$,
\begin{align*}
  S_3^o(n)&=\sum_{j=0}^{n/2-1}\csc^3\tfrac{(2j+1)\pi}{n}=\frac{7\zeta(3)}{4\pi^3}n^3+O(n\log n),\\
     S_5^o(n)&=\sum_{j=0}^{n/2-1}\csc^5\tfrac{(2j+1)\pi}{n}=\frac{93\zeta(5)}{48\pi^5}n^5+O(n^3),\\
         \hat S^e_3(n)&=\sum_{j=1}^{n/2-1}\csc^3\tfrac{2j\pi}{n}=\frac{\zeta(3)}{4\pi^3}n^3+O(n\log n),\\
         \hat S_1(n)&=\sum_{j=1}^{n-1}(-1)^{j+1}\csc\tfrac{j\pi}{n}=\frac{n}{\pi}\log 4+O(1).
\end{align*}
\end{lemma}
\begin{proof}
The proof here is partially inspired by \cite{Watson1916sum,blagouchine2024finite}.
    Using the classical result of Euler
\begin{align*}
    \pi \csc a\pi =\int_{0}^\infty \frac{x^{a-1}}{1+x}dx,\quad 0<a<1,
\end{align*}
one can derive that
\begin{align*}
     &\sum_{j=0}^{n/2-1}\csc \left(\phi+\tfrac{2j\pi }{n}\right)=\frac{1}{\pi}\int_0^\infty \frac{x^{\frac{\phi}{\pi}}}{x(1+x)}\sum_{j=0}^{n/2-1}x^{\frac{2j}{n}}dx=\frac{1}{\pi}\int_0^\infty \frac{x^{\frac{\phi}{\pi}}(x-1)}{x(1+x)(x^{\frac{2}{n}}-1)}dx\\
    &=\frac{n}{2\pi}\int_{-\infty}^{\infty} \frac{e^{\frac{n\phi y}{2\pi}}(e^{\frac{ny}{2}}-1)}{(1+e^{\frac{ny}{2}})(e^y-1)}dy=\frac{n}{2\pi}\int_0^\infty \frac{[e^{\frac{n\phi y}{2\pi}}+e^{y-\frac{n\phi y}{2\pi}}](e^{\frac{yn}{2}}-1)}{(1+e^{\frac{yn}{2}})(e^y-1)}dy,
    %&=\frac{2n}{\pi}\int_0^\infty \frac{\tanh nt}{\sinh t}\cosh\left(\frac{2n\phi}{\pi}t-t\right)dt
\end{align*}
% \begin{align*}
%      &\sum_{l=0}^{n-1}\csc \left(\phi+\frac{\pi l}{n}\right)=\frac{1}{\pi}\int_0^\infty \frac{x^{\frac{\phi}{\pi}}}{x(1+x)}\sum_{l=0}^{n-1}x^{\frac{l}{n}}dx=\frac{1}{\pi}\int_0^\infty \frac{x^{\frac{\phi}{\pi}}(x-1)}{x(1+x)(x^{\frac{1}{n}}-1)}dx\\
%     &=\frac{n}{\pi}\int_{-\infty}^{\infty} \frac{e^{\frac{n\phi y}{\pi}}(e^{yn}-1)}{(1+e^{yn})(e^y-1)}dy=\frac{n}{\pi}\int_0^\infty \frac{[e^{\frac{n\phi y}{\pi}}+e^{y-\frac{n\phi y}{\pi}}](e^{yn}-1)}{(1+e^{yn})(e^y-1)}dy
%     %&=\frac{2n}{\pi}\int_0^\infty \frac{\tanh nt}{\sinh t}\cosh\left(\frac{2n\phi}{\pi}t-t\right)dt
% \end{align*}
where we make a change of variable $x=e^{\frac{ny}{2}}$, split the integral and change the integral on $(-\infty,0)$ to $(0,\infty)$. Furthermore, if we set $y=2t$, then
\begin{align}\label{cscphi}
    \sum_{j=0}^{n/2-1}\csc \left(\phi+\tfrac{2j\pi }{n}\right)=\frac{n}{\pi}\int_0^\infty \frac{\tanh \frac{nt}{2}}{\sinh t}\cosh\left(\tfrac{n\phi}{\pi}t-t\right)dt.
\end{align}
Setting $\phi=\pi/n$, we have
\begin{align}\label{csc1-o}
\sum_{j=0}^{n/2-1}\csc\tfrac{(2j+1)\pi}{n}=\frac{n}{\pi}\int_0^\infty \frac{\tanh \frac{n}{2}t}{\sinh t}dt.
\end{align}
  Taking derivative  of \eqref{cscphi} on $\phi$ twice and using $(\csc x)''=2\csc^3 x-\csc x$, we get
\begin{align*}
    \sum_{j=0}^{n/2-1}2\csc^3\left(\phi+\tfrac{2j\pi}{n}\right)-\csc\left(\phi+\tfrac{2j\pi }{n}\right)=\left(\frac{n}{\pi}\right)^3\int_0^\infty \frac{t^2\tanh \frac{nt}{2}}{\sinh t}\cosh \left(\tfrac{n\phi}{\pi}t-t\right)dt.
\end{align*}
Setting $\phi=\pi/n$,
\begin{align*}
    %\sum_{j=0}^{n-1}\csc\frac{(2j+1)\pi}{2n}&=\frac{2n}{\pi}\int_0^\infty \frac{\tanh nt}{\sinh t}dt\\
    \sum_{j=0}^{n/2-1}\csc^3\tfrac{(2j+1)\pi}{n}&=\frac12 \sum_{j=0}^{n/2-1}\csc\frac{(2j+1)\pi}{n}+\frac12\left(\frac{n}{\pi}\right)^3\int_0^\infty \frac{t^2\tanh \frac{n}{2}t}{\sinh t}dt.
\end{align*}
To compute its leading order. First, it is easy to show $\sum_{j=0}^{n/2-1}\csc\frac{(2j+1)\pi}{n}=O(n\log n)$. Second,
\begin{align*}
    \int_0^\infty \frac{t^2\tanh \frac{n}{2}t}{\sinh t}dt=\int_0^\infty \frac{t^2}{\sinh t}\left(1+\frac{-2}{e^{nt}+1}\right)dt.
\end{align*}
For the first part,
\begin{align*}
    \int_0^\infty \frac{t^2}{\sinh t}dt=\int_0^\infty \frac{2t^2e^{-t}}{1-e^{-2t}}dt=2\int_0^\infty t^2 \sum_{j=0}^\infty e^{-(2j+1)t}dt=4\sum_{j=0}^\infty (2j+1)^{-3}=\frac{7}{2}\zeta(3),
\end{align*}
where $\zeta$ is the Riemann zeta function.
For the second part, using $\sinh t\geq t$ for $t>0$
\begin{align*}
    \int_0^\infty \frac{t^2}{\sinh t}\frac{1}{1+e^{nt}}dt\leq \int_0^\infty te^{-nt}dt=n^{-2}.
\end{align*}
Summarizing the above analysis,
\begin{align*}
  \sum_{j=0}^{n/2-1}\csc^3\tfrac{(2j+1)\pi}{n}=\frac{7\zeta(3)}{4\pi^3}n^3+O(n\log n).
\end{align*}
Differentiating \eqref{cscphi} five times on $\phi$ and
using $\frac{d^4}{dx^4}\csc x=(24-20 (\sin x)^2+(\sin  x)^4)(\csc x)^5$, one has
\begin{align*}
     \sum_{j=0}^{n/2-1}\csc^5\tfrac{(2j+1)\pi}{n}&=\frac{1}{24}\left(\frac{n}{\pi}\right)^5\int_0^\infty \frac{t^4\tanh \frac{n}{2}t}{\sinh t}dt+O(n^3)=\frac{93\zeta(5)}{48\pi^5}n^5+O(n^3).
\end{align*}
This finishes the odd sum. For the even sum, we can compute similarly.
\begin{align*}
    \sum_{j=1}^{n/2-1}\csc \left(\phi+\tfrac{2j\pi}{n}\right)=\frac{n}{\pi}\int_0^\infty \frac{\sinh [(\frac{n}{2}-1)t]}{(\sinh t)( \cosh \frac{nt}{2})}\cosh \left(\frac{n\phi }{\pi}t\right)dt.
\end{align*}
Letting $\phi=0$, we have
\begin{align}\label{csc1-e}
\sum_{j=1}^{n/2-1}\csc \tfrac{2j\pi}{n} =\frac{n}{\pi}\int_0^\infty \frac{\sinh[(\frac{n}{2}-1)t]}{(\sinh t)(\cosh \frac{n}{2}t)}dt.
\end{align}
Taking derivative on $\phi$ twice, we get
\begin{align*}
    \sum_{j=1}^{n/2-1}2\csc^3\left(\phi+\tfrac{2j\pi }{n}\right)-\csc\left(\phi+\tfrac{2j\pi }{n}\right)=\left(\frac{n}{\pi}\right)^3\int_0^\infty \frac{t^2\sinh [(\frac{n}{2}-1)t]}{(\sinh t)( \cosh \frac{nt}{2})}\cosh \left(\frac{n\phi}{\pi}t\right)dt.
\end{align*}
Setting $\phi=0$,
\begin{align*}
   % b_1&=\sum_{j=1}^{K/2-1}\csc(2j\theta_0)=\frac{K}{\pi}\int_0^\infty \frac{\sinh[(\frac{K}{2}-1)t]}{(\sinh t)(\cosh \frac{K}{2}t)}dt\\
    \sum_{j=1}^{n/2-1}\csc^3\tfrac{2j\pi}{n}=\frac12\sum_{j=1}^{n/2-1}\csc\tfrac{2j\pi}{n}+\frac{1}{2}\frac{n^3}{\pi^3}\int_0^\infty \frac{t^2\sinh [(\tfrac{n}{2}-1)t]}{(\sinh t)( \cosh \tfrac{n}{2}t)}dt.
\end{align*}
Compute the leading order of them, we have
\begin{align*}
    %b_1=O(K\log K),\quad
    \sum_{j=1}^{n/2-1}\csc^3\tfrac{2j\pi}{n}=\frac{\zeta(3)}{4\pi^3}n^3+O(n\log n).
\end{align*}

Using \eqref{csc1-o} and \eqref{csc1-e}, one has
\begin{align*}
    \sum_{j=1}^{n-1}(-1)^{j+1}\csc\frac{j\pi}{n}=\frac{n}{\pi}\int_0^{\infty}\frac{\sinh \frac{n}{2}t-\sinh (\frac{n}{2}-1)t}{\sinh t \cosh \frac{n}{2}t}dt.
\end{align*}
It is easy to verify the following identity
\begin{align*}
    \sinh \tfrac{n}{2}t-\sinh (\tfrac{n}{2}-1)t=(\cosh \tfrac{n}{2}t)(1-e^{-t})+\frac12e^{-\frac{n}{2}t}(e^{t}+e^{-t}-2).
\end{align*}
Note that
\begin{align*}
    \int_0^\infty \frac{\frac12e^{-\frac{n}{2}t}(e^t+e^{-t}-2)}{\sinh t \cosh \frac{n}{2}t}dt=2\int_0^\infty\frac{e^t-1}{(1+e^{nt})(e^{t}+1)}dt=O(n^{-1}).
\end{align*}
Therefore
\begin{align*}
     \sum_{j=1}^{n-1}(-1)^{j+1}\csc\tfrac{j\pi}{n}=\frac{n}{\pi}\int_0^\infty\frac{1-e^{-t}}{\sinh t}dt+O(1)=\frac{n}{\pi}\log 4+O(1).
\end{align*}

\end{proof}

\begin{lemma}\label{lem:rough-Sk}
% When $x\in [0,\frac1n]$,
% \begin{align*}
%     S^o_k(n,x)+S^e_k(n,x)=\sum_{j=0}^{n-1}(x^2+\sin^2\tfrac{j\pi}{n})^{-\frac{k}{2}}\leq C \begin{cases}n|\log n|&\text{if }k=1,\\ n^{-k}&\text{if }k\geq 2.\end{cases}
% \end{align*}
When $x\in[0,1]$,
\begin{align*}
    S^o_k(n,x)+S^e_k(n,x)=\sum_{j=0}^{n-1}(x^2+\sin^2\tfrac{j\pi}{n})^{-\frac{k}{2}}\leq C \begin{cases}n|\log x|&\text{if }k=1,\\ nx^{1-k}&\text{if }k\geq 2.\end{cases}
\end{align*}
\end{lemma}
\begin{proof}
It is easy to see the following
\begin{align*}
    S^o_k(n,x)+S^e_k(n,x)=\sum_{j=0}^{n-1}(x^2+\sin^2\tfrac{j\pi}{n})^{-\frac{k}{2}}=2\sum_{j=0}^{n/2-1}(x^2+\sin^2\tfrac{j\pi}{n})^{-\frac{k}{2}}.
\end{align*}
We divide the sum into two parts. Let $j_0=\lfloor nx \rfloor$, the greatest integer less than or equal to $nx$. Then,
    % \begin{align*}
    %     S^o_k(n,x)+S^e_k(n,x)=\sum_{j=0}^{n-1}(x^2+\sin^2\tfrac{j\pi}{n})^{-\frac{k}{2}}
    % \end{align*}
    \begin{align*}
        \sum_{j=0}^{j_0}(x^2+\sin^2\tfrac{j\pi}{n})^{-\frac{k}{2}}&\leq x^{-k} nx =Cnx^{1-k},\\
         \sum_{j=j_0}^{n-1}(x^2+\sin^2\tfrac{j\pi}{n})^{-\frac{k}{2}}&\leq \sum_{j=j_0}^{n-1} \csc^{k} \tfrac{j\pi}{n}\leq Cn^k\sum_{j_0}^{n-1}j^{-k}\leq C\begin{cases}n|\log x|&\text{if }k=1,\\ nx^{1-k}&\text{if }k\geq 2.\end{cases}
    \end{align*}
    % \begin{align*}
    %     \sum_{j=j_0}^{n-1}(x^2+\sin^2\tfrac{j\pi}{n})^{-\frac{k}{2}}\leq \sum_{j=j_0}^{n-1} \csc^{k} \tfrac{j\pi}{n}\leq Cn^k\sum_{j_0}^{n-1}j^{-k}
    % \end{align*}
    %using $\sum_{j=j_0}^{n-1}j^{-1}\leq C|\log x|$ when $k=1$, and $\sum_{j=j_0}^{n-1}j^{-k}\leq (nx)^{1-k}$ when $k\geq 2$,
\end{proof}
\begin{lemma}\label{S-integral} Suppose that $n$ is even. The following formula holds
    \begin{align*}
    % S_1^o(n,x,\theta)=\frac{n}{\pi}\int_{\arcsinh x}^\infty\frac{\tanh \frac{n t}{2}}{(1-\sech^2 \frac{nt}{2}\sin^2\frac{n\theta}{2})\sqrt{\sinh^2 t-x^2}}dt,\\
    % S_1^e(n,x,\theta)=\frac{n}{\pi}\int_{\arcsinh x}^\infty\frac{\coth \frac{n t}{2}}{(1+\csch^2 \frac{nt}{2}\sin^2\frac{n\theta}{2})\sqrt{\sinh^2 t-x^2}}dt,\\
    % S_1(n,x,\theta)=\frac{2n}{\pi}\int_{\arcsinh x}^\infty\frac{\csch nt\cos n\theta}{(1+\csch^2 nt\sin^2 n\theta)\sqrt{\sinh^2 t-x^2}}dt.
     S_1(n,x)=\frac{2n}{\pi}\int_{\arcsinh x}^\infty\frac{\csch nt}{\sqrt{\sinh^2 t-x^2}}dt.
    % T_1(n,x,\theta)=\frac{2n}{\pi}\int_{\arcsinh x}^\infty\frac{\csch nt\cos n\theta}{(1+\csch^2 nt\sin^2 n\theta)\sqrt{\sinh^2 t-x^2}}dt.
    \end{align*}
\end{lemma}
\begin{proof}
%     When $n$ is even, the alternating sum has an integral representation
% \begin{align*}
%     &S_1(n,x)=\sum_{j=0}^{n-1}(-1)^j (x^2+(\sin \frac{j\pi}{n})^2)^{-\frac12}=\frac{2n}{\pi}\int_{\arcsinh x}^\infty \frac{\csch n\phi}{ \sqrt{\sinh^2 \phi-x^2}}d\phi.
% \end{align*}

By the periodic and even property, we assume $\theta\in (0,\pi/n)$.
To prove the first one, one can apply Cauchy's residue theorem to
\[\frac{\csc(nz)}{\sqrt{x^2+\sin^2 z}}\]
in the complex plane. It has poles $z=j\pi/n$ for $j=0,1,\cdots n-1$ in $[0,\pi)$. Choose a branch of  ${\sqrt{x^2+\sin^2 z}}$ by removing the rays $[i\arcsinh x,i\infty)+l\pi$ and $(-i\infty,-i\arcsinh x]+l\pi$ for all $l\in \mathbb{Z}$. Draw an integration contour including the poles of $\csc(nz)$ in $[0,\pi)$ and deform it along the branch cut of $\sqrt{x^2+\sin^2 z}$. See figure \ref{fig:contour}. We omit the details.

%The second identity follows similarly. The third one follows the difference between the first two integrals
\begin{figure}[ht]
    \centering
    \includegraphics[width=0.35\textwidth]{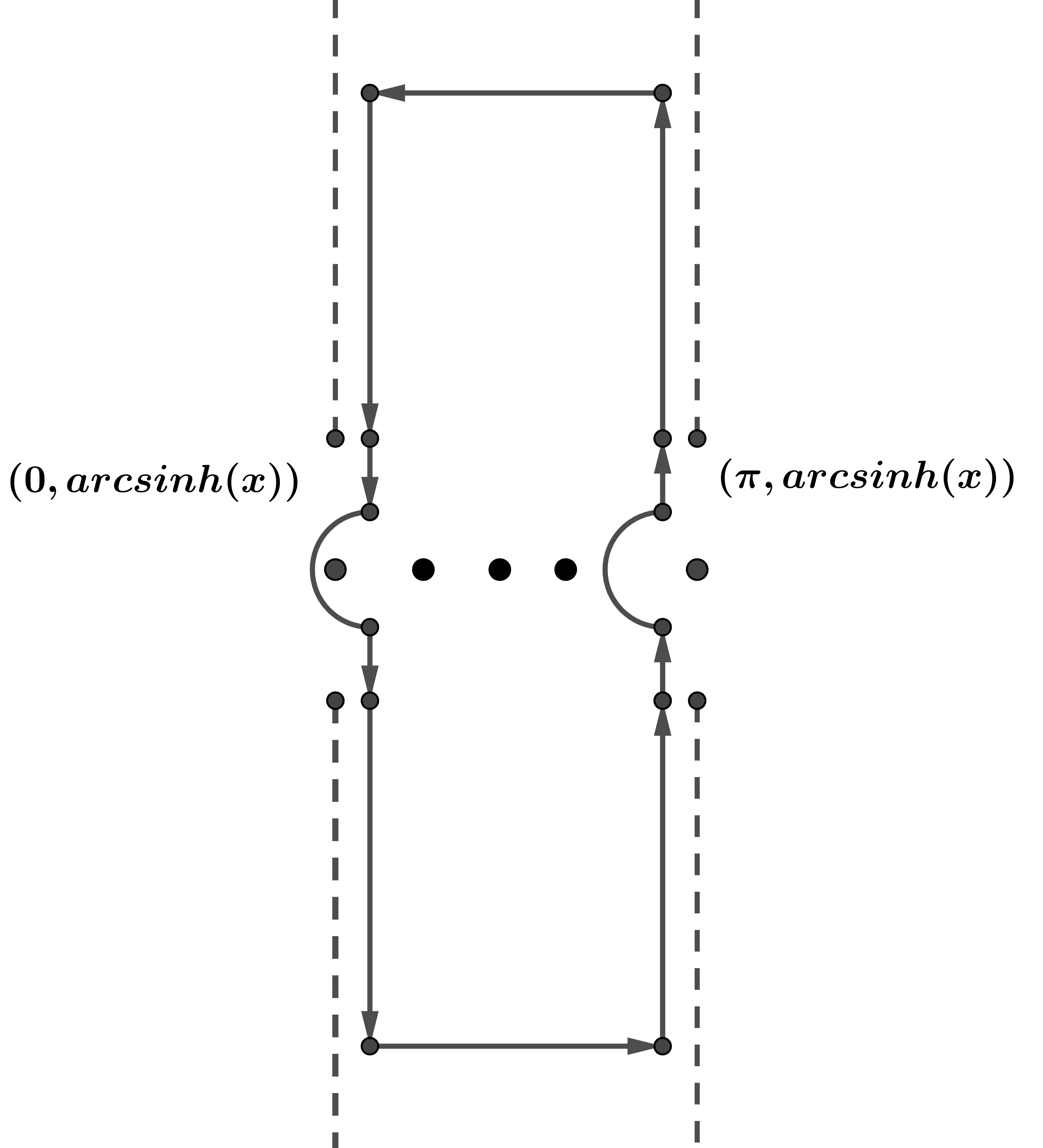}
    \caption{The illustration of the integration contour in Lemma \ref{S-integral}.}
    \label{fig:contour}
\end{figure}
\end{proof}

\begin{lemma}\label{lem:alt.sum-d}
    Suppose that $n$ is even and  $x=(1+o_n(1))\frac{\log n}{n}$ as $n\to \infty$. Then
\begin{align*}
    S_1(n,x)=\sqrt{\frac{8}{\pi  }}n(nx)^{-\frac12}e^{-nx}(1+O((nx)^{-1})),\\
    S_3(n,x)=\sqrt{\frac{8}{\pi  }}n^3(nx)^{-\frac32}e^{-nx}(1+O((nx)^{-1})),\\
    S_5(n,x)=\frac13\sqrt{\frac{8}{\pi  }}n^5(nx)^{-\frac52}e^{-nx}(1+O((nx)^{-1})).
\end{align*}
\end{lemma}
\begin{proof}
 It follows from Lemma \ref{S-integral} that
\begin{align*}
    &S_1(n,x)=\sum_{j=0}^{n-1}(-1)^j (x^2+(\sin \frac{j\pi}{n})^2)^{-\frac12}=\frac{2n}{\pi}\int_{\arcsinh x}^\infty \frac{\csch n\phi}{ \sqrt{\sinh^2 \phi-x^2}}d\phi.
\end{align*}
Making a change of variable $\sinh \phi=x\cosh u$ to the above formula, one has
\begin{align}\label{alt.sum-1/2}
    S_1(n,x)=\frac{2n}{\pi}\int_0^\infty \frac{\csch (n \arcsinh (x\cosh u))}{\sqrt{1+x^2\cosh ^2u}}du.
\end{align}

%Thus it suffices to get the aysmptotics of the integral on the RHS of \eqref{alt.sum-1/2}-\eqref{alt.sum-5/2}.

First, the integral is very small when $u$ is away from $0$. In fact, since $\arcsinh t\sim t$ as $t\to 0$ and $\cosh(2)>3$, then $\csch (n \arcsinh (x\cosh u))\leq   Ce^{-3nx} $ if $u>2$ and  $n$ is large enough. Thus
    \begin{align}\label{csch/sqrt}
        \int_{2}^\infty \frac{\csch (n \arcsinh (x\cosh u))}{\sqrt{1+x^2\cosh ^2u}}du\leq Ce^{-3nx}\int_{2}^\infty \frac{1}{\sqrt{1+x^2\cosh^2 u}}du.
    \end{align}
Making a change of variable $\tanh u=t$, we obtain
\begin{align*}
    \int_0^\infty \frac{1}{\sqrt{1+x^2\cosh^2 u}}du=\int_0^1 \frac{1}{\sqrt{(1-t^2)(1+x^2-t^2)}}dt=\frac{\mbox{EllipticK}((1+x^2)^{-\frac12})}{\sqrt{1+x^2}},
\end{align*}
where
\begin{align}
\label{2.elliptic-1}
\mbox{EllipticK}(\sigma)=\int_0^1\frac{dt}{\sqrt{(1-t^2)(1-\sigma^2t^2)}}.
\end{align}
Using \cite[Formula 900.05]{byrd1971handbook} we have
\begin{equation}
\label{2.elliptic}
\mbox{EllipticK}(\sigma)=\sum_{\ell=0}^\infty\left(\begin{matrix}
-\frac12\\ \ell
\end{matrix}\right)^2\left[\log\frac{4}{\sqrt{1-\sigma^2}}-b_\ell\right](1-m^2)^\ell,
\end{equation}
with
$$b_0=0,\quad b_\ell=2\sum_{j=1}^{2\ell}\frac{(-1)^{j-1}}{j}=b_{\ell-1}+\frac{2}{2\ell(2\ell-1)}.$$
Using \eqref{2.elliptic}, we have
\begin{align}
\label{2.elliptic-2-est}
\frac{\mbox{EllipticK}((1+x^2)^{-\frac12})}{\sqrt{1+x^2}}=-\log x+\log 4+O(x^2)\quad \text{as } x\to 0.
\end{align}
Insert the above estimates back to \eqref{csch/sqrt},
   \begin{align*}
        \int_{2}^\infty \frac{\csch (n \arcsinh (x\cosh u))}{\sqrt{1+x^2\cosh ^2u}}du \leq Ce^{-3nx}\int_{2}^\infty \frac{1}{\sqrt{1+x^2\cosh^2 u}}du\leq C|\log x| e^{-3nx}.
    \end{align*}

Second, for $u\in [0,2]$, we recall the Taylor expansion
\begin{align}
\begin{split}\label{3-taylor}
    \arcsinh(t)&=t-\frac{t^3}{6}+O(t^5) \quad \text{as }t\to 0,\\
    \csch t&=2e^{-t}+2e^{-3t}+O(e^{-5t}), \quad \text{as }t\to \infty,\\
    \frac{1}{\sqrt{1+t^2}}&=1-\frac{1}{2}t^2+O(t^4),\quad \text{as }t\to 0.
\end{split}
\end{align}
% Then we have the following expansion $t=O(\frac{\log n}{n})$ as $n\to \infty$
% \begin{align*}
%      \frac{\csch (n \arcsinh (t))}{\sqrt{1+t^2}}=2e^{-nt}+O(nt^3+t^2)e^{-nt}=(2+O(nt^3))e^{-nt}.
% \end{align*}
Then we have the following expansion as $n\to \infty$
\begin{align*}
     \frac{\csch (n \arcsinh (x\cosh u))}{\sqrt{1+(x\cosh u)^2}}&=(2+O(nx^3))e^{-nx\cosh u}
\end{align*}
holds uniformly for $u\in [0,2]$.
Using this expansion, we have
\begin{align*}
    \int_0^{2} \frac{\csch (n \arcsinh (x\cosh u))}{\sqrt{1+x^2\cosh ^2u}}du
    &=(2+O(nx^3))\int_0^{2}e^{-nx\cosh u}du\\
    &=(2+O(nx^3))\int_0^\infty e^{-nx\cosh u}du+O(e^{-3nx})\\
    &=(2+O(nx^3))\mbox{BesselK}_0(nx)+O\left(e^{-3nx}\right),
\end{align*}
where $\mbox{BesselK}_\alpha(t)$ is the modified Bessel function.
\begin{align*}
    \mbox{BesselK}_\alpha(t)=\int_0^\infty e^{-t\cosh u }\cosh (\alpha u)du.
\end{align*}
 It is well known that
\[\mbox{BesselK}_0(t)=\sqrt{\frac{\pi}{2}}t^{-\frac12}e^{-t}(1+\frac{a_1}{t}+\frac{a_2}{t^2}+O(t^{-3})),\quad \text{as }t\to \infty.\]
where $a_1,a_2$ are some constants.
% Since $nx\to \infty$ as $n\to \infty$, then
% \begin{align*}
%      \int_0^2 \frac{\csch (n \arcsinh (x\cosh u))}{\sqrt{1+x^2\cosh ^2u}}du&=\sqrt{2\pi}(nx)^{-\frac12}e^{-nx}(1+O((nx)^{-1})).
% \end{align*}
% Combining with the integral on $[2,\infty)$, we get the result of $S_1(n,x)$.
Since $nx\to \infty$ as $n\to \infty$, we have
\begin{align*}
    \int_0^2 \frac{\csch (n \arcsinh (x\cosh u))}{\sqrt{1+x^2\cosh ^2u}}du&=\sqrt{2\pi}(nx)^{-\frac12}e^{-nx}(1+O((nx)^{-1})).
\end{align*}
Combining the two integral estimates and using $|\log x|e^{-2nx}=O((nx)^{-3/2})$, we get the asymptotic of $S_1$.

%Inserting these two equalities and the estimates of $u>2$ back to \eqref{alt.sum-1/2}- \eqref{alt.sum-5/2}, we get the conclusion.

To get the asymptotic of $S_3$, we
take the derivative of \eqref{alt.sum-1/2} on $x$
\begin{align}
\begin{split}\label{alt.sum-3/2}
    S_3(n,x)&=-\frac{1}{x}\frac{d}{dx}S_1(n,x)=\frac{2n}{\pi }\int_0^\infty \left(\frac{-1}{x}\frac{d}{dx}\right)\frac{\csch (n \arcsinh (x\cosh u))}{\sqrt{1+x^2\cosh ^2u}}du.
    %&=\frac{2n}{\pi}\int_0^\infty \frac{n}{x}\cdot\frac{f_1(n,x,u)}{1+x^2\cosh^2u}+  \frac{f_2(n,x,u)}{(1+x^2\cosh^2 u)^{3/2}}du
    % S_5(n,x)&=\frac{2n}{3\pi}\left(\frac{1}{x}\frac{d}{dx}\right)\left(\frac{1}{x}\frac{d}{dx}\right)\int_0^\infty \frac{\csch (n \arcsinh (x\cosh u))}{\sqrt{1+x^2\cosh ^2u}}du.\label{alt.sum-5/2}
    \end{split}
\end{align}
It is straightforward to get
\begin{align*}
    \left(\frac{-1}{x}\frac{d}{dx}\right)\frac{\csch (n \arcsinh (x\cosh u))}{\sqrt{1+x^2\cosh ^2u}}=\frac{n}{x}\cdot\frac{f_1(n,x,u)}{1+x^2\cosh^2u}+  \frac{f_2(n,x,u)}{(1+x^2\cosh^2 u)^{3/2}},
\end{align*}
where
 \[f_1(n,x,u)=\coth(n \arcsinh (x\cosh u))\csch(n \arcsinh (x\cosh u))\cosh u,\]
\[f_2(n,x,u)=\csch(n \arcsinh (x\cosh u))\cosh^2u.\]
We shall study the right-hand side of \eqref{alt.sum-3/2} as we did for $S_1$. First, if $u>2$
\begin{align*}
    \left(\frac{-1}{x}\frac{d}{dx}\right)\frac{\csch (n \arcsinh (x\cosh u))}{\sqrt{1+x^2\cosh ^2u}}&\leq Ce^{-3nx}\left(\frac{n}{x}\frac{\cosh u}{1+x^2\cosh^2u}+\frac{\cosh^2u}{(1+x^2\cosh^2u)^{3/2}}\right)\\
    &\leq Ce^{-3nx}\frac{n}{x}\frac{\cosh u}{1+x^2\cosh^2u}.
\end{align*}
Making a change of variable $\cosh u=t^{-1}$,
\begin{align*}
    \int_2^\infty \frac{\cosh u}{1+x^2\cosh^2u}du=\int_0^{1/\cosh 2}\frac{1}{(x^2+t^2)\sqrt{1-t^2}}dt\leq C\int_0^{1/\cosh 2}\frac{1}{x^2+t^2}dt\leq \frac{C}{x}.
\end{align*}
Thus
\begin{align*}
    \int_2^\infty \left(\frac{-1}{x}\frac{d}{dx}\right)\frac{\csch (n \arcsinh (x\cosh u))}{\sqrt{1+x^2\cosh ^2u}}du\leq Cn x^{-2}e^{-3nx}.
\end{align*}
% \begin{align*}
%     \frac{n}{x}\cdot\frac{f_1(n,x,u)}{1+x^2\cosh^2u}\leq Ce^{-3nx}\frac{n}{x}\frac{\cosh u}{1+x^2\cosh^2u}\\
%     \frac{f_2(n,x,u)}{(1+x^2\cosh^2 u)^{3/2}}\leq Ce^{-3nx}\frac{\cosh^2u}{(1+x^2\cosh^2u)^{3/2}}
% \end{align*}
% \begin{align*}
%     \int_2^\infty \frac{f_2(n,x,u)}{(1+x^2\cosh^2u)^{3/2}}du\leq e^{-3nx}\int_0^\infty \frac{\cosh^2u}{1+x^2\cosh^2u }du\lesssim x^{-2}e^{-3nx}
% \end{align*}
Second, consider $u\in [0,2]$. Using the expansion \eqref{3-taylor}, we obtain
\begin{align*}
    \left(\frac{-1}{x}\frac{d}{dx}\right)\frac{\csch (n \arcsinh (x\cosh u))}{\sqrt{1+x^2\cosh ^2u}}=2nx^{-1}e^{-nx\cosh u}\cosh u(1+O(nx^3)),
\end{align*}
holds uniformly for $u\in [0,2]$. Thus
\begin{align*}
    \int_0^2\left(\frac{-1}{x}\frac{d}{dx}\right)&\frac{\csch (n \arcsinh (x\cosh u))}{\sqrt{1+x^2\cosh ^2u}}du=2nx^{-1}(1+O(nx^3))\int_0^2 e^{-nx \cosh u}\cosh udu\\
    &=2nx^{-1}(1+O(nx^3))\int_0^\infty e^{-nx \cosh u}\cosh udu+O(nx^{-1}e^{-3nx})\\
    &=-2nx^{-1}(1+O(nx^3))\mbox{BesselK}'_0(nx)+O(nx^{-1}e^{-3nx}),
\end{align*}
where $\mbox{BesselK}'_\alpha(t)=\tfrac{d}{dt}\mbox{BesselK}_\alpha(t)$. It follows from the expansion of $\mbox{BesselK}_0(t)$ that
\begin{align*}
   \mbox{BesselK}'_0(t) =-\sqrt{\frac{\pi}{2}}t^{-\frac12}e^{-t}(1+O(t^{-1})),\quad \text{as }t\to \infty.
\end{align*}
Therefore
\begin{align*}
   \int_0^2\left(\frac{-1}{x}\frac{d}{dx}\right)\frac{\csch (n \arcsinh (x\cosh u))}{\sqrt{1+x^2\cosh ^2u}}du=\sqrt{2\pi}nx^{-1}(nx)^{-\frac12}e^{-nx}(1+O((nx)^{-1})).
\end{align*}
Combining the above estimates with the one in $(2,\infty)$, we get the asymptotic of $S_3$.

The proof of $S_5$ can be obtained similarly to $S_3$. We omit the details.
\end{proof}

\medskip

\begin{center}
{\bf Acknowlegdement}
\end{center}
The research of L. Sun is partially supported by the National Natural Science Foundation of China(Grant No.\,1247012316), and the Strategic Priority Research Program of the Chinese Academy of Sciences (No.\,XDB0510201), and CAS Project for Young Scientists in Basic Research Grant (No.\,YSBR-031). The research of J.C. Wei is supported by GRF New frontier in singular limits of nonlinear partial differential equations. The research of W. Yang is supported by the National
Key R\&D Program of China 2022YFA1006800, NSFC China No.
12171456, NSFC, China No. 12271369, FDCT No. 0070/2024/RIA1, Start-up
Research Grant No. SRG2023-00067-FST,
Multi-Year Research Grant No. MYRG-GRG2024-00082-FST and UMDF-TISF/2025/006/FST.

\appendix
\section{Some technical computations in Section 2}\label{appendix-A}
In this section, we shall provide the details for the proof of \eqref{2.psi-m2-0} in Section 2.

\begin{proof}
[Proof of \eqref{2.psi-m2-0}.]
Indeed, we have
\begin{equation}
\label{2.connect}
\begin{aligned}
&\left\|5\sum_{j=1}^m\frac{3^\frac14\mu_m^{\frac12}}{|x-\xi_{j,0}|}U^4+\left(1-\sum_{j=1}^m\zeta_j\right)E\right\|_{L_w^q}\\
&\leq \left\|5\sum_{j=1}^m \frac{3^\frac14\mu_m^{\frac12}}{|z-\xi_{j,0}|}U^4-5\left(1-\sum_{j=1}^m\zeta_j\right)\sum_{j=1}^m \frac{3^\frac14\mu_m^{\frac12}}{|z-\xi_{j,0}|}U^4\right\|_{L_w^q}\\
&\quad+\left\|5\left(1-\sum_{j=1}^m\zeta_j\right)\sum_{j=1}^m3^\frac14\mu_m^{\frac12}\left(\frac{U^4}{|z-\xi_{j,0}|}-\frac{U^4}{|z-\xi_{j}|}\right)\right\|_{L_w^q}\\
&\quad+\left\|\left(1-\sum_{j=1}^m\zeta_j\right)\left(
5\sum_{j=1}^m \frac{3^\frac14\mu_m^{\frac12}}{|z-\xi_{j}|}U^4+E\right)\right\|_{L_w^q}.
\end{aligned}
\end{equation}
It is not difficult to check that
\begin{equation}
\label{2.connect-1}
\left\|\sum_{j=1}^m \frac{3^\frac14\mu_m^{\frac12}U^4}{|z-\xi_{j,0}|}-\left(1-\sum_{j=1}^m\zeta_j\right)\sum_{j=1}^m \frac{3^\frac14\mu_m^{\frac12}U^4}{|z-\xi_{j,0}|}\right\|_{L_w^q}\leq \frac{C}{(\log m)^2},
\end{equation}
and
\begin{equation}
\label{2.connect-2}
\left\|\left(1-\sum_{j=1}^m\zeta_j\right)\sum_{j=1}^m3^\frac14\mu_m^{\frac12}\left(\frac{U^4}{|z-\xi_{j,0}|}-\frac{U^4}{|z-\xi_{j}|}\right)\right\|_{L_w^q}\leq \frac{C}{(\log m)^2}.
\end{equation}
Consider the last term. Recall that the support of $1-\sum_{j=1}^m\zeta_j$ is $\bigcap_{j=1}^m\left\{|z-\xi_j|>\frac{\eta}{m}\right\}$. Thus
\begin{equation*}
\begin{aligned}
\left|5\sum_{j=1}^m \frac{3^\frac14\mu_m^{\frac12}}{|z-\xi_{j}|}U^4+E\right|
\leq C\frac{\mu_m}{(1+|z|^2)^\frac32}\left(\sum_{j=1}^{m}\frac{1}{|z-\xi_j|}\right)^2.
\end{aligned}
\end{equation*}
We write $z=(r\cos\theta,r\sin\theta,z_3)$ as Section 2.2, by \eqref{2.com-1} we have
\begin{equation}
\label{2.com-3}\sum_{j=1}^m\frac{1}{|z-\xi_j|}
=\frac{1}{2(r\sqrt{1-\mu_m^2})^\frac12}\sum_{j=0}^{m-1}\frac{1}{\left(h^2+\sin^2\left(\frac{j\pi}{m}-\frac{\theta}{2}\right)\right)^\frac12},
\end{equation}
where $h$ is given in \eqref{2.com-2}. It is not difficult to check that
$$\sum_{j=1}^m\frac{1}{|z-\xi_j|}\leq Cm\log m,~\mbox{for}~z\in \bigcap_{j=1}^m\left\{|z-\xi_j|>\frac{\eta}{m}\right\} ~ \mbox{with}~ h\leq Cm^{-\frac12}.$$
Next we claim that for $h\in \left(m^{-\frac12},\frac13\right)$, it holds that
\begin{equation}
\label{2.claim-h}
\sum_{j=0}^{m-1}\frac{1}{\left(h^2+\sin^2\left(\frac{j\pi}{m}-\frac{\theta}{2}\right)\right)^\frac12}\leq Cm\log\frac{1}{h}.
\end{equation}
By symmetry, we may assume that $\theta\in\left(0,\frac{\pi}{m}\right)$. We set $j_h$ to be the greatest number such that $\frac{j\pi}{m}-\frac{\theta}{2}\geq h$. Then we have
\begin{align*}
\sum_{j=0}^{m-1}\frac{1}{\left(h^2+\sin^2\left(\frac{j\pi}{m}-\frac{\theta}{2}\right)\right)^\frac12}
\leq~&2\sum_{j=0}^{\frac{m}{2}-1}\frac{1}{\left(h^2+\sin^2\left(\frac{j\pi}{m}-\frac{\theta}{2}\right)\right)^\frac12}\\
\leq~&2\sum_{j=0}^{j_h}\frac{1}{h}
+2\sum_{j=j_h}^{\frac{m}{2}-1}\frac{1}{\sin\left(\frac{j\pi}{m}-\frac{\theta}{2}\right)}\\
\leq~&C\frac{j_h}{h}+Cm\log\frac{m}{j_h}\leq Cm\log\frac{1}{h}.
\end{align*}
Hence, we proved the claim \eqref{2.claim-h}. While for $h\geq\frac13$, we have
\begin{equation}
\label{2.claim-h2}
\left(1-\sum_{j=1}^m\zeta_j\right)\sum_{j=1}^{m}\frac{1}{|z-\xi_j|}\leq Cm.
\end{equation}
Using \eqref{2.com-3}, \eqref{2.claim-h} and \eqref{2.claim-h2} we derive that
\begin{equation}
\label{2.psi-m2-est}
\begin{aligned}
&\left\|\left(1-\sum_{j=1}^m\zeta_j\right)\frac{\mu_m}{(1+|z|^2)^\frac32}\left(\sum_{j=1}^{m}\frac{1}{|z-\xi_j|}\right)^2\right\|_{L_w^q}\\
&\leq \frac{C}{(\log m)^2}\left(\int_{\mathbb{R}^3}\frac{(1+|z|)^{5q-6}}{(1+|z|)^{5q}}dz\right)^\frac1q+\frac{C}{(\log m)^2}\left(\int_{m^{-\frac12}}^{\frac13}s\left(\log\frac{1}{s}\right)^qds\right)^\frac1q\\
&\quad +C\left|\left\{z\mid (r-\sqrt{1-\mu_m^2})^2+z_3^2\leq Cm^{-1}\right\}\right|^\frac1q\\
&\leq\frac{C}{(\log m)^2}.
\end{aligned}
\end{equation}
By \eqref{2.psi-m2-est} and \eqref{2.connect}-\eqref{2.connect-2} we get \eqref{2.psi-m2-0}.
\end{proof}

\small
\bibliographystyle{plainnat}
\bibliography{ref.bib}
\end{document}